\documentclass[11pt, reqno]{amsart}
\usepackage[includehead,includefoot,margin=19mm]{geometry}
\usepackage[latin1]{inputenc}
\usepackage{amssymb}
\usepackage{amsmath}
\usepackage{amsfonts}
\usepackage{mathtools}
\usepackage{mathrsfs}
\usepackage{hyperref}
\usepackage[all]{xy}  
\usepackage{verbatim}
 \usepackage{graphicx}
\usepackage[usenames,dvipsnames]{xcolor}
\usepackage{xy}
\hypersetup{colorlinks=true,citecolor=green,linkcolor=red,urlcolor=red}

\theoremstyle{plain}
\newtheorem{theorem}{Theorem}[section]

\newtheorem{lemma}[theorem]{Lemma}
\newtheorem{proposition}[theorem]{Proposition}

\newtheorem{corollary}[theorem]{Corollary}
\newtheorem*{theorem*}{Theorem}
 
\newtheorem*{thmA}{Theorem A}
\newtheorem*{thmB}{Theorem B}
\newtheorem*{thmC}{Theorem C}
\newtheorem*{thmD}{Theorem D}
\newtheorem*{thmE}{Theorem E}
 
\theoremstyle{definition}
\newtheorem{definition}[theorem]{Definition}

\newtheorem{example}[theorem]{Example}

\theoremstyle{remark}

\newtheorem{remark}[theorem]{Remark}

\author[]{Marta Agustin Vicente}
\address{Basque Center of Applied Mathematics (BCAM)\\  Mazarredo, 14. 48009 Bilbao \\ Spain}
\email{martaav22@gmail.com} 

\author[]{Narasimha Chary Bonala$^\ast$}
\address{Ruhr-Universit\"at Bochum\\ Fakult\"at f\"ur Mathematik\\ 
D-4478 Bochum, Germany}
\curraddr{Department of Mathematics and Statistics, Indian Institute of Technology Kanpur,
U.P. India, 208016.
}
\email{chary@iitk.ac.in}
\thanks{$^\ast$Corresponding author: Narasimha Chary Bonala (chary@iitk.ac.in)}

\author[]{Kevin Langlois}
\address{Departemento de Matem\'atica \\ Universidade Federal do Cear\'a (UFC) \\ Campus do Pici, Bloco 914, CEP 60455 -760. Fortaleza-Ce, Brazil}
\email{kevin.langlois@mat.ufc.br}

\title []
{On intersection cohomology with torus \\  action of complexity one, II}

\begin{document}
\begin{abstract}
We show that the components, appearing in the decomposition theorem for contraction maps of torus actions of complexity one, are intersection cohomology complexes  of even codimensional subvarieties.
 As a consequence, we obtain the vanishing of the odd dimensional intersection cohomology for rational complete varieties with torus action of complexity one. The article also presents structural results on linear torus action in order to
 compute the intersection cohomology from the weight matrix. In particular, we determine the intersection cohomology Betti numbers of affine trinomial hypersurfaces in terms of their defining equation. 
\end{abstract}  

\maketitle

{\bf Mathematics Subject Classification (2020) } 32S60, 14L30, 14M25.

\vspace{0.5cm}

{\bf Keywords:} Intersection Cohomology and Perverse Sheaves, Decomposition Theorem, Algebraic Torus Actions, Toric Varieties, Singularity Theory.

\section{Introduction}
We are studying  intersection cohomology of complete complex algebraic varieties endowed with an action
of an algebraic torus $\mathbb{T} =  (\mathbb{C}^{\star})^{n}$. An important invariant in the classification of torus actions is the \emph{complexity}. It is defined, for a $\mathbb{T}$-variety $X$, as the number $c(X):= {\rm dim}\, X -  {\rm dim}\, \mathbb{T}/\mathbb{T}_{0}$, where $\mathbb{T}_{0}$ is the kernel of the torus action. \emph{Toric varieties} are precisely normal varieties with torus action of complexity $0$. They admit a combinatorial description by objects of convex geometry such as fans, polytopes, etc.  
From this combinatorial data, one can recover both the isomorphism classes of toric varieties and many of their geometric properties.
This was performed for the intersection cohomology \cite{Sta87,  DL91, Fie91, BBFK99, BM99, BBFK02, CMM18}, where in the case of complete simplicial fans, this description is related to the \emph{$f$-vector} of the toric variety (that is, the vector encoding the number of cones of the fan of each dimension).

For torus actions of complexity one, there is still a dictionary between $\mathbb{T}$-varieties and objects lying between geometry and combinatorics \cite{AIPSV12}. This class of $\mathbb{T}$-varieties 
encompasses the \emph{$\mathbb{C}^{\star}$-surfaces},  whose intersection cohomology has been studied by Fieseler and Kaup \cite{FK86, FK91}. 
 In \cite{AL18, AL21}, we started a program for determining  the intersection cohomology Betti numbers in the complexity-one case. The topology of $\mathbb{T}$-varieties of complexity
one has also been studied in \cite{LLM20}.

More precisely, let $X$ be a complete variety with torus action of complexity one, let  $\widetilde{X}$ be the normalization of the graph of the rational quotient $\iota:  X\dashrightarrow C$, where $C$ is the smooth projective curve parameterizing the general orbits, and let $\pi: \widetilde{X}\rightarrow X$ be the natural map, called the \emph{contraction map}. We say $\widetilde{X}$ is the \emph{contraction space} of $X$. Our approach consisted of the following two steps. 
\begin{itemize}
\item[$(1)$]   The determination of the intersection cohomology Betti numbers of $\widetilde{X}$ (see \cite{AL18}, \cite[Section 5.1]{AL21}) from its toroidal structure and the toric decomposition theorem \cite{CMM18}. 
\item[$(2)$] The elaboration of an algorithm \cite{AL21} expressing the intersection cohomology 
of $X$ in terms of the one of $\widetilde{X}$ via the decomposition theorem of the contraction map.
\end{itemize}
In this paper, we carry out the final step of our program for the complexity-one case, namely:
\begin{itemize}
\item[$(3)$] We provide a simple formula for the intersection cohomology Betti numbers. Moreover,  we present structural results that describe these numbers from the defining equations. Finally, we treat the example of affine trinomial hypersurfaces. 
\end{itemize}
We now discuss the results of the paper. Our first upshot is an improvement of \cite[Theorem 1.1]{AL21}. It gives a simple form of the decomposition theorem for the contraction map. For a variety $V$, we denote by $IC_{V}$ its intersection cohomology complex and by $IH^{j}(V; \mathbb{Q})$ the intersection cohomology groups.
\begin{thmA}[See Theorem \ref{theo-even}]\label{TheoA-TheoA}
Let $X$ be a normal variety with torus action of complexity one. Denote by $E$ the image of the exceptional
locus of the contraction map $\pi: \widetilde{X}\rightarrow X$, and let $\rm{Orb}_{\rm even}(E)$ be the set of even codimensional torus orbits  of $X$ contained in $E$. Then we have an isomorphism
$$\pi_{\star} IC_{\widetilde{X}} \simeq IC_{X}\oplus \bigoplus_{O \in \rm{Orb}_{\rm even}(E)} (\iota_{O})_{\star} IC_{\bar{O}}^{\oplus s_{O}},$$
where $\iota_{O} : \bar{O}\rightarrow X$ is the inclusion and $s_{O}\in \mathbb{Z}_{\geq 0}$ for any $O\in \rm{Orb}_{\rm even}(E)$.
\end{thmA}
Next result, which is consequence of Theorem A, provides a cohomological criterion of rationality in the setting of torus actions of complexity one.  
\begin{thmB}[See Theorem \ref{theo-rat-imp-coho}]\label{TheoB-TheoB}
Let $X$ be a complete variety with torus action of complexity one. Then the following are equivalent.
\begin{itemize}
\item[$(1)$] The variety $X$ is rational.
\item[$(2)$] We have $IH^{2j+1}(X; \mathbb{Q})  = 0$  for any $j\in \mathbb{Z}$. 
\end{itemize}
\end{thmB}
Application of Theorem A is a  formula for
the intersection cohomology from the description of \cite{AIPSV12}, in terms of divisorial fans, extending the $3$-dimensional case \cite[Theorem 1.2]{AL21}.  \emph{Divisorial fans} are analogous to fans in toric geometry for complexity-one torus actions. They consist of finite collections $\mathscr{E}$ of divisors with polyhedral coefficients (called \emph{polyhedral divisors}) over Zariski open subsets of a smooth projective curve $C$ with additional conditions.  For a complexity-one $\mathbb{T}$-variety $X$ with defining divisorial fan $\mathscr{E}$,
the curve $C$ represents the rational quotient of the torus action, while the polyhedral coefficients encode the geometry of the fibers of the global quotient $\widetilde{X}\rightarrow C$ of the contraction space of $X$. 

Let us introduce further notations. We call \emph{Poincar\'e polynomial} of a variety $V$, the polynomial
$P_{V}(t) =  \sum_{j\in \mathbb{Z}}b_{j}(V)t^{j}$, where $b_{j}(V):= {\rm dim}\, IH^{j}(V; \mathbb{Q})$. For a strictly convex polyhedral cone $\sigma$ (resp. a fan $\Sigma$), let $X_{\sigma}$ (resp. $X_{\Sigma}$)
be the associated toric variety. We denote by
$g(\sigma; t^{2}) =  \sum_{j\in \mathbb{Z}_{\geq 0}}g_{j}(\sigma)t^{j}$ the \emph{$g$-polynomial} of $\sigma$, where $g_{j}(\sigma)$ is
the $j$-th local intersection cohomology of $X_{\sigma}$ along the closed orbit. Moreover, we write $h(\Sigma; t^{2})$ for 
the \emph{$h$-polynomial} of a  complete fan $\Sigma$, which is the Poincar\'e polynomial of $X_{\Sigma}$. Note that $g(\sigma; t^{2})$ and $h(\Sigma; t^{2})$ are combinatorial objects, see \cite{DL91, Fie91, BBFK02}. Similarly, the $h$-polynomial $h_{\widetilde{\mathscr{E}}}(t)$ of a divisorial fan $\widetilde{\mathscr{E}}$ describing
the contraction space $\widetilde{X}$ of a complete normal variety $X$ with torus action of complexity one is the Poincar\'e polynomial of $\widetilde{X}$. The polynomial $h_{\widetilde{\mathscr{E}}}(t)$ is explicit from the divisorial fan of $X$ (see the reminder in Section \ref{sec: h-poly23}).

The formula for the Betti numbers can be expressed as follows (see Corollary \ref{corcor-gD-Ic} for the affine case). 
 Let $X$ be a complete normal $\mathbb{T}$-variety of complexity one
with divisorial fan $\mathscr{E}$. 
Denote by $\widetilde{\mathscr{E}}$ the divisorial fan of the contraction space $\widetilde{X}$ of $X$. Let us consider set $HF(\mathscr{E})$ (see \ref{Beettidivfanop})  consists of polyhedral cones that are in one-to-one
correspondence with the orbits in the image of the exceptional locus of the contraction map of $X$. Moreover, for $\tau\in HF(\mathscr{E})$, the symbol ${\rm Star}(\mathscr{E}, \tau)$ stands for the fan of the normalization of the orbit closure associated with $\tau$. Finally, for any $\tau\in HF(\mathscr{E})$ the numbers $n(\tau)$ and $c(\tau)$ are respectively ${\rm dim}\, \tau -1$ and ${\rm dim}\, \tau +1$. Then, we have the following theorem.

\begin{thmC}[See Corollary \ref{coro-Betti-divfan}]
The 
Poincar\'e polynomial of $X$ is given by the formula
$$P_{X}(t)  =  h_{\widetilde{\mathscr{E}}}(t) -  \sum_{\tau\in HF(\mathscr{E})}g_{n(\tau)}(\tau)t^{c(\tau)}h({\rm Star}(\mathscr{E}, \tau); t^{2}).$$

\end{thmC}
Part of the present work develops computation methods. More specifically, consider the projective space $\mathbb{P}^{\ell}_{\mathbb{C}}$ with its natural toric structure for the torus $\mathbb{G} =  (\mathbb{C}^{\star})^{\ell}$. Let $X\subset \mathbb{P}^{\ell}_{\mathbb{C}}$ be a subvariety with linear torus action of complexity one, i.e. $X$ meets
 the open $\mathbb{G}$-orbit and there is a subtorus $\mathbb{T}\subset\mathbb{G}$ acting on $X$ with complexity one.
 The inclusion $\mathbb{T}\hookrightarrow \mathbb{G}$ induces a linear map $\mathbb{Z}^{n}\rightarrow \mathbb{Z}^{\ell}$ whose matrix $F$ is the \emph{weight matrix} of the $\mathbb{T}$-action. In this article,
we solve the following question.
\\

{\bf Question}
$(\star)$: \emph{How do we compute the intersection cohomology of the variety $X$ from the matrix $F$?}
\\

Since the variety $X$ has same intersection cohomology Betti numbers as its normalization $\widehat{X}$ (see \cite[Section 5, Lemma 1]{GS93}), answer of Question  $(\star)$ means to build a divisorial fan of $\widehat{X}$ from the matrix $F$ and apply Theorem C. 
For this, we use ingredients of \cite{KSZ91}, \cite[Section 4]{AIPSV12} that we now recall.

Start with the short exact sequence
$$0\rightarrow \mathbb{Q}^{n}\xrightarrow{F} \mathbb{Q}^{\ell}\xrightarrow{P} {\rm Coker}(F)\rightarrow 0,$$
a section $S: \mathbb{Z}^{\ell} \rightarrow \mathbb{Z}^{n}$ of $F$, and the fan $\Sigma$ 
generated by the images of $P$ of the faces of the first quadrant $\delta:= \mathbb{Q}_{\geq 0}^{\ell}$. We say that the data $\theta =  (\overline{N} = \mathbb{Z}^{\ell}, N = \mathbb{Z}^{n}, F, S, \Sigma)$ is a \emph{weight package}. We attach to the weight package $\theta$ a piecewise linear map
$$\mathfrak{D}_{\theta}:  m \mapsto \sum_{\rho\in \Sigma(1)}\min_{v\in S(\delta \cap P^{-1}(v_{\rho}))}\langle m , v\rangle \cdot Z_{\rho},$$
which goes from the dual  cone of $\sigma_{\theta}:=  S(\delta\cap F(N_{\mathbb{Q}}))$ to the vector space of Cartier $\mathbb{Q}$-divisors of $X_{\Sigma}$. Note that we are  using toric notations, namely $\Sigma(1)$ is the set of rays of $\Sigma$, $v_{\rho}$ is the primitive generator of $\rho$ and $Z_{\rho}\subset X_{\Sigma}$ is the corresponding toric divisor.

 Denoting by $\psi: \mathbb{G}\rightarrow \mathbb{T}_{{\rm Coker}(F)}$ the quotient map onto the torus associated with ${\rm Coker}(F)$,  the pullback $\kappa^{\star}(\mathfrak{D}_{\theta})$ is seen as a polyhedral divisor, where $\kappa$ is the composition of the normalization $\widehat{C}_{\theta}\rightarrow C_{\theta}$ of the curve $C_{\theta}$ obtained as the Zariski closure of $\psi(X\cap \mathbb{G})$ in $X_{\Sigma}$ and the inclusion $C_{\theta}\subset X_{\Sigma}$. Precisely, write $x_{0}, \ldots, x_{\ell}$ for the homogeneous coordinates of $\mathbb{P}_{\mathbb{C}}^{\ell}$ coming from the toric structure. Then
the polyhedral divisor $\kappa^{\star}(\mathfrak{D}_{\theta})$ describes the normalization of the variety $X^{(0)} =  X\setminus \mathbb{V}(x_{0})$
(see Theorem \ref{theo-aff-impl}). Via matrix operations (see Definition \ref{def-enhance}), one associates weight packages $$\theta^{(i)} =   (\overline{N}, N, F^{(i)}, S^{(i)}, \Sigma^{(i)}), \,\,0\leq i\leq \ell,$$ that define the $\mathbb{T}$-actions on the charts  $X^{(i)}:= X\setminus \mathbb{V}(x_{i})$.  

The following result, extending \cite[Section 4]{AIPSV12} to the non-normal case, gives an interrelation between
weight packages and defining equations of  projective varieties with linear torus action of complexity one. 
\begin{thmD}[See  Theorem \ref{theo-main-nonnorm}]\label{TheoC-TheoC}
\begin{itemize}
\item[(1)] Let $X\subset \mathbb{P}^{\ell}_{\mathbb{C}}$ be a subvariety with linear torus action of complexity one. Let $\theta$ be the weight package of $X$ and let $\theta^{(i)} =   (\overline{N}, N, F^{(i)}, S^{(i)}, \Sigma^{(i)})$ be the weight package corresponding to the chart $X^{(i)}$. 
Let $\overline{\Sigma}$ be a fan with support $\bigcup_{i =  0}^{\ell}|\Sigma^{(i)}|$ such that for $0\leq i\leq \ell$
the set
$\overline{\Sigma}^{(i)} :=  \{\sigma \in \overline{\Sigma}\,|\, \sigma \subset |\Sigma^{(i)}|\}$
is a projective fan subdivision of $\Sigma^{(i)}$. Denote by $\kappa^{(i)}: \widehat{C_{\theta^{(i)}}}\rightarrow X_{\Sigma^{(i)}}$ the map, which is the composition of the projective modification $f^{(i)}: X_{\overline{\Sigma}^{(i)}}\rightarrow X_{\Sigma^{(i)}}$, the inclusion $C_{\theta^{(i)}}'\rightarrow X_{\overline{\Sigma}^{(i)}},$ where $C_{\theta^{(i)}}'$ is the proper transform of $C_{\theta^{(i)}}$ under $f^{(i)}$, and the normalization 
$\widehat{C_{\theta^{(i)}}}\rightarrow C_{\theta^{(i)}}'$. 
Then the set
$$\{\bar{\mathfrak{D}}_{\theta}^{(i)} :=  \kappa^{(i) \star}\mathfrak{D}_{\theta^{(i)}}\,\,|\,\, i =  0,1, \ldots, \ell\}$$
generates  a divisorial fan $\mathscr{E}_{\theta}$ describing the normalization of $X$.
\item[(2)] Conservely, let $\theta =  (\overline{N} =  \mathbb{Z}^{\ell}, N =  \mathbb{Z}^{n}, F, S, \Sigma)$ be a weight package. Let $C\subset \mathbb{T}_{{\rm Coker}(F)} =  (\mathbb{C}^{\star})^{s}$ be an irreducible curve with defining equations $f_{i}$  $(1\leq i\leq a)$.
Define the matrix
$$\widehat{P} =   \begin{bmatrix} 
    b_{1,0} & b_{1,1} & \dots & b_{1, \ell}\\
    \vdots &  \vdots &     & \vdots  \\
    b_{s, 0} & b_{s,1} & \dots & b_{s, \ell}
    \end{bmatrix}\in {\rm Mat}_{s\times \ell + 1}(\mathbb{Z}),$$
as the addition of $P$ of a first column so that the sum of the entries of each row is $0$. Set 
$$g_{i}(T_{0},T_{1}, \ldots, T_{\ell}):= f_{i}\left(\prod_{j = 0}^{\ell}T_{j}^{b_{1, j}}, \ldots, \prod_{j = 0}^{\ell}T_{j}^{b_{s, j}}\right)\text{ for } 1\leq i\leq a,$$
and assume that there exist Laurent monomials $u_{i}\in \mathbb{C}[T_{0}, T_{0}^{-1}, \ldots, T_{\ell}, T_{\ell}^{-1}]$ 
such that $X:= \mathbb{V}(u_{1}g_{1}, \ldots, u_{a}g_{a})\subset \mathbb{P}^{\ell}_{\mathbb{C}}$ is irreducible. Then $X$ is a subvariety with linear torus action of complexity one and its normalization is described by the divisorial fan $\mathscr{E}_{\theta}$ obtained from Contruction $(1)$. 
\item[$(3)$] If $X \subset \mathbb{P}^{\ell}_{\mathbb{C}}$ is a projective subvariety with linear torus of complexity one with weight package $\theta$, then $X$ admits a decomposition $X  =  \mathbb{V}(u_{1}g_{1}, \ldots, u_{a}g_{a})$ as in Assertion $(2)$. 
\end{itemize}
\end{thmD}
We  illustrate our method with the example of \emph{affine trinomial hypersurfaces,} which are zero loci
$$X  =  \mathbb{V}(T_{1}^{\underline{n}_{1}} + T_{2}^{\underline{n}_{2}} + T_{3}^{\underline{n}_{3}}) \subset \mathbb{A}_{\mathbb{C}}^{\ell}$$
such that $T_{i}^{\underline{n}_{i}}$ is a monomial
$ \prod_{j = 1}^{r_{i}} T_{i, j}^{n_{i,j}} \text{ for } i = 1,2,3,$ with
$r_{i}, n_{i,j}\in\mathbb{Z}_{>0}$ and $\ell =  r_{1} + r_{2} + r_{3}$. 
Set
$$ u_{i}:= {\rm gcd}(n_{i,1}, \ldots, n_{i, r_{i}}) \text{ for } i = 1,2,3, \,\, d =  {\rm gcd}(u_{1}, u_{2}, u_{3}),$$
$$ d_{1}  =  {\rm gcd}(u_{2}/d, u_{3}/d), d_{2}  =  {\rm gcd}(u_{1}/d, u_{3}/d), d_{3}  =  {\rm gcd}(u_{1}/d, u_{2}/d) \text{ and } u = dd_{1}d_{2}d_{3}.$$ From the  trinomial equation, one associates a weight package $\theta$ (see Section \ref{sec-trinomial-aff}). 

 Using the result of Kruglov \cite[Theorem 3.1]{Kru19} and Corollary \ref{corcor-gD-Ic} (the affine version of Theorem C)  we compute the intersection cohomology of any affine trinomial hypersurface. 
\begin{thmE}[See Corollary \ref{cor-trinome}]\label{TheoD-TheoD}
Let $$X  =  \mathbb{V}(T_{1}^{\underline{n}_{1}} + T_{2}^{\underline{n}_{2}} + T_{3}^{\underline{n}_{3}}) \subset \mathbb{A}_{\mathbb{C}}^{\ell}$$
be an affine trinomial hypersurface with natural weight package $\theta =  (\overline{N} =  \mathbb{Z}^{\ell}, N, F, S, \Sigma)$ and let $$(e_{1,1}, \ldots, e_{1, r_{1}}, e_{2, 1}, \ldots, e_{2, r_{2}}, e_{3,1}, \ldots, e_{3, r_{3}})$$ be 
the canonical basis of $\overline{N}.$
Set  $\gamma =  d(d_{1} + d_{2} + d_{3})$
and $$\Pi_{i} := {\rm Cone}\left( (\sigma_{\theta}\times \{0\})\cup \left(\left\{ S\left( \frac{d}{d_{j} n_{i,j}}e_{i,j}\right)\,|\, 1\leq j\leq r_{i}\right\}\times\{1\}\right)\right)$$
for $i = 1, 2, 3$. Then the Poincar\'e polynomial of the contraction space $\widetilde{X}$ of $X$ is given by
$$P_{\widetilde{X}}(t) =  \left(t^{2} + (du - \gamma + 2)t - \gamma + 1\right)\cdot g(\sigma_{\theta}; t^{2}) +  \sum_{i = 1}^{3} dd_{i}\cdot g(\Pi_{i}; t^{2}).$$   
Furthermore, write 
$$H(\theta, \underline{n}_{1},\underline{n}_{2},\underline{n}_{3}):= \left\{\tau \text{ face of }\sigma_{\theta}\,|\, \tau\cap \left\{\sum_{i =  1}^{3} S\left( \frac{d}{d_{j} n_{i,j_{i}}}e_{i,j_{i}}\right)\,|\, (j_{1}, j_{2}, j_{3})\in \prod_{i = 1}^{3}\{1, \ldots, r_{i}\}\right\} \neq \emptyset\right\}.$$
Then the Poincar\'e polynomial of $X$ is obtained from the relation
$$P_{X}(t)  = P_{\widetilde{X}}(t)  -  \sum_{\tau \in H(\theta, \underline{n}_{1},\underline{n}_{2},\underline{n}_{3})} g_{n(\tau)}(\tau) t^{c(\tau)}, $$
where $n(\tau) =  {\rm dim}\, \tau - 1$ and $c(\tau) =  {\rm dim}\, \tau +1$.      
\end{thmE}
Let us give a brief summary of the contents of each section. Section \ref{sec-two} gives preliminaries  on intersection cohomology
and torus actions. As  preparation for Theorem A, we study  in Section \ref{sec-three} intersection cohomology with finite group action. In Subsection \ref{subsec-pullbackfiniteGaction} we obtain  results  on pullbacks of quotient maps of finite group actions for intersection cohomology complexes that might be of independent interest (see Propositions \ref{prop-IC-G-X}, \ref{theo-finitegrouptoric}). Section \ref{sec-four} introduces weight package theory, and
illustrates the concept with examples.
We then show in Subsection \ref{sec-projproj} Theorem D.
 Moreover, we prove a similar result as Theorem D for contraction spaces of  torus actions of complexity one (see Theorem \ref{theo-main-contr-theorDelta}).  Section \ref{sec-five} is devoted to the proofs of Theorem A and Theorem B.
Finally, in Section \ref{sec-six}, we discuss some
consequences of the results of Sections \ref{sec-four} and \ref{sec-five} and describe the intersection cohomology for trinomial hypersurfaces, where we treat the affine case and partially the projective case. 
\\

\emph{Perspective.}
It was conjectured \cite{CHL18} that  the decomposition theorem exists, in a strong form, over finite fields, and confirmed for toric varieties \cite{Cat15} and convolution morphisms of partial affine flag varieties \cite[Section 6]{CHL18}. Can this be verified in the setting of torus actions of complexity one? 
\\

\emph{Acknowledgement.}
The second author was supported by the SFB/TRR 191 Symplectic Structures in Geometry, Algebra and Dynamics of the Deutsche Forschungsgemeinschaft. The second and third authors would like to thank Max Planck Institute for Mathematics (Bonn) for providing a great environment to make this collaboration possible in the initial stages.

The authors would like to thank the referees for carefully reading the manuscript
and for helpful comments and suggestions.

\subsection{Convention} \label{sec-varalg}
 \emph{Local systems} are  locally constant sheaves of $\mathbb{Q}$-vector spaces for the Euclidean topology with finite-dimensional stalks. \emph{Semi-projective} means projective over affine.

\section{Preliminaries}\label{sec-two}
This section is devoted to the preliminaries on intersection cohomology and torus actions. 
\subsection{Algebraic varieties}
We recall some basic notions on divisor theory.
We write $\leq$ for the coefficient-wise inequality between $\mathbb{Q}$-divisors. Given a $\mathbb{Q}$-divisor $D$ on  a normal variety $Y$ we denote by $H^{0}(Y, \mathcal{O}_{Y}(D))$ its space of global sections, and for a rational function $f\in \mathbb{C}(Y)^{\star}$ we write ${\rm div}(f)$ for its principal divisor and ${\rm ord}_{Z}(f)$ for its vanishing order along a prime divisor $Z\subset Y$. For $s\in H^{0}(Y, \mathcal{O}_{Y}(D))\setminus\{0\}$, the \emph{zero locus}  $Z_{Y, D}(s)$ is the union of the prime divisors that are in the support
of ${\rm div}(s) + D$. We denote by $Y_{D, s} :=  Y\setminus Z_{Y, D}(s)$ the complement. We say that 
the $\mathbb{Q}$-divisor $D$ is \emph{semi-ample} if for some $r\in \mathbb{Z}_{>0}$, the open subsets $Y_{rD, s}$ cover $Y$, where
$s$ runs over $H^{0}(Y, \mathcal{O}_{Y}(rD))\setminus\{0\}$.
The following notion generalize the usual notion of a big divisor on a projective variety, see \cite[Lemma 2.60]{KM98}.
\begin{definition}
The $\mathbb{Q}$-divisor $D$ on the normal variety $Y$ is \emph{big} if there are $r\in \mathbb{Z}_{>0}$  and $s\in H^{0}(Y, \mathcal{O}_{Y}(rD))\setminus\{0\}$ such that $Y_{rD, s}$ is affine. 
\end{definition}
Later we will use the following observation. 
\begin{lemma}\label{lem1-Qdivisor}
Let $D, D'$ be two $\mathbb{Q}$-divisors on a normal variety $Y$. Assume that $D$ is semi-ample and that $H^{0}(Y, \mathcal{O}_{Y}(rD))\subset H^{0}(Y, \mathcal{O}_{Y}(rD'))$ for any $r\in \mathbb{Z}_{>0}$. Then  $D\leq D'$. 
\end{lemma}
\begin{proof}
Since $D$ is semi-ample, there exist $r\in\mathbb{Z}_{>0}$,
an open covering $(U_{i})_{i\in I}$ of $Y$ and a family of nonzero global sections $(f_{i})_{i\in I}$ of $rD$ such that $rD_{|U_{i}} = -{\rm div}(f_{i})_{|U_{i}}$ for any $i$. Since each $f_{i}$ is in $H^{0}(Y, \mathcal{O}_{Y}(rD'))$, we have $rD_{|U_{i}} = -{\rm div}(f_{i})_{|U_{i}}\leq rD_{|U_{i}}'$ for any $i$, proving $D\leq D'$. 
\end{proof}
\subsection{Intersection cohomology}
We set our convention on intersection cohomology theory (see \cite{BBD82,  GM83, CM09}, \cite[Section 1]{Wil17}, \cite{Max19}). For a variety $X$
we denote by $D_{\rm const}^{b}(X)$ the constructible derived category of sheaves of $\mathbb{Q}$-vector spaces on $X$;
this is a triangulated category with shift functor $[1]$. Given a morphism $f: X\rightarrow Y$ of varieties, we write  
$$ f_{\star}, f_{!}: D_{\rm const}^{b}(X)\rightarrow D_{\rm const}^{b}(Y)\text{ and }  f^{\star}, f^{!}: D_{\rm const}^{b}(Y)\rightarrow D_{\rm const}^{b}(X)$$
for the derived functors $Rf_{\star},$ $Rf_{!}$, etc. For a complex $\mathcal{F}\in D_{\rm const}^{b}(X)$ we denote
by $\mathcal{H}^{j}(\mathcal{F})$ its $j$-th cohomology sheaf and the arrow
$\mathbb{D}: D_{\rm const}^{b}(X)\rightarrow D_{\rm const}^{b}(X)$
will be the Verdier duality. 
\\

Let
$X =  \bigcup_{\lambda\in I} X_{\lambda}$
be  an algebraic Whitney stratification,
where $X_{\lambda_{0}}$ is the open stratum and the $i_{\lambda}: X_{\lambda}\rightarrow X$ are inclusions of strata.
Let $\mathscr{L}$ be a local system on $X_{\lambda_{0}}$. We denote by $IC_{X}(\mathscr{L})$ the \emph{intersection cohomology complex} with coefficients in $\mathscr{L}$. According to Deligne, it is uniquely determined by:
\begin{itemize}
\item[(1)] the open stratum condition: $i_{\lambda_{0}}^{\star} IC_{X}(\mathscr{L}) =  \mathscr{L}[{\rm dim}\, X]$,
\item[(2)] the support conditions: $\mathcal{H}^{j}(i_{\lambda}^{\star} IC_{X}(\mathscr{L})) =  0$ for $j\geq -{\rm dim}\, X_{\lambda}$
and $\lambda\neq \lambda_{0}$, and
\item[(3)] the co-support conditions: $\mathcal{H}^{j}(i_{\lambda}^{!} IC_{X}(\mathscr{L})) =  0$ for $j \leq -{\rm dim}\, X_{\lambda}$
and $\lambda \neq \lambda_{0}$. 
\end{itemize}
Intersection cohomology complexes belong to the category of \emph{perverse sheaves} for the middle perversity $p$, i.e.,
it is an element of the heart of the category $D_{\rm const}^{b}(X)$ for the $t$-structure
$${}^{p}D^{\leq 0}(X) :=  \{ \mathcal{F}\in D_{\rm const}^{b}(X)\,|\, {\rm dim}\, {\rm Supp}(\mathcal{H}^{j}(\mathcal{F})\leq -j\text{ for all } j\}, $$
$${}^{p}D^{\geq 0}(X) :=  \{ \mathcal{F}\in D_{\rm const}^{b}(X)\,|\, {\rm dim}\, {\rm Supp}(\mathcal{H}^{j}(\mathbb{D}\mathcal{F})\leq -j\text{ for all } j\}.$$
The \emph{intersection cohomology groups} with coefficients in $\mathscr{L}$ are the hypercohomology groups 
$$ IH^{\star}(X; \mathscr{L}):= \mathbb{H}^{\star}(X, IC_{X}(\mathscr{L})[-{\rm dim}\, X]).$$
In particular, if $\mathscr{L} =  \mathbb{Q}$, then we set $IC_{X}:= IC_{X}(\mathscr{L})$ and $IH^{\star}(X; \mathbb{Q}):=\mathbb{H}^{\star}(X, IC_{X}[-{\rm dim}\, X]).$
Observe 
 that $IH^{\star}(X; \mathbb{Q}) =  H^{\star}(X; \mathbb{Q})$ whenever $X$ is smooth.

\begin{definition}\label{def2}
An object in $D^{b}_{\rm const}(X)$ is \emph{semi-simple} or \emph{pure} \cite[Section 5.4]{BBD82} if it is a finite direct sum
of objects  $\iota_ {\star}IC_{Z}(\mathscr{L})[r],$
where $r\in \mathbb{Z}$, $\iota: Z\rightarrow X$ is the inclusion of  a Zariski closed irreducible subvariety,  and $\mathscr{L}$ is a simple local system on a smooth Zariski open subset of $Z$. 
\end{definition}
For a morphism of varieties $f:X\rightarrow Y$
let $Y^{i}:= \{y\in Y\,|\, {\rm dim}\, f^{-1}(y) = i\}.$ The \emph{defect}
of $f$ is the number 
${\rm def}(f):= {\rm max}\{2i + {\rm dim}\, Y^{i}- {\rm dim}\, X\,|\, Y^{i}\neq \emptyset\}.$
The map $f$ is \emph{semi-small} if ${\rm def}(f) = 0$. Recall that $X$ is \emph{rationally smooth} if for any $x\in X$ the cohomology $H^{i}_{x}(X; \mathbb{Q})$ with support in $\{x\}$ is $0$ when $i\neq {\rm dim}(X)$. 
The following is the \emph{decomposition theorem}.
\begin{theorem}\cite[Theorem 6.25]{BBD82}, \cite[Theorem 2.4]{Wil17}
Let $f:X\rightarrow Y$ be a proper morphism of varieties. Then
$f_{\star} IC_{X}$ is semi-simple. If further $f$ is birational, then $IC_{Y}$ is a summand of $f_{\star}IC_{X}$. Moreover, if $X$ is rationally smooth and $f$ is semi-small, then $f_{\star} IC_{X}$ is semi-simple and perverse.
\end{theorem}
Next results collect  properties of intersection cohomology.
\begin{lemma}\label{lemma-smooth}
Let $f: X\rightarrow Y$ be a smooth morphism of varieties of relative dimension $r$.
Then we have $f^{\star}IC_{Y}[r]\simeq IC_{X}$. 
\end{lemma}
\begin{proof}
Same argument as \cite[Lemma 2.4]{AL18}.
\end{proof}
\begin{lemma}\label{LelENORMAphhh}\cite[Section 5, Lemma 1]{GS93}.
If $f:X\rightarrow Y$ is a finite birational morphism of varieties, then we have $f_{\star}IC_{X}\simeq IC_{Y}$. In particular, if $\widehat{Y}$
is the normalization of $Y$, then $IH^{\star}(Y; \mathbb{Q}) =  IH^{\star}(\widehat{Y}; \mathbb{Q})$.
\end{lemma}
\begin{lemma}\label{GtrivGtrivlemma}
Let $X$ be a quasi-projective variety with action of a connected algebraic group $\Gamma$ and let
$G\subset \Gamma$ be a finite subgroup. Then $IH^{\star}(X/G; \mathbb{Q})\simeq IH^{\star}(X; \mathbb{Q})$.
\end{lemma}
\begin{proof}
The group $G$ acts on $IH^{\star}(X;\mathbb{Q})$ and $IH^{\star}(X;\mathbb{Q})^{G}\simeq IH^{\star}(X/G;\mathbb{Q})$ \cite[Lemma 2.12]{Kir86}.
Using decomposition theorem for equivariant desingularizations as in \cite{AW97}, it suffices 
to prove that $G$ trivially acts on $IH^{\star}(X;\mathbb{Q})$ when $X$ is smooth, which follows from the proof of \cite[Proposition 6.4]{DL76}. 
\end{proof}
\begin{lemma} \label{LemmaCstarfixepointLemma}
Let $X$ be an affine variety with a non-hyperbolic $\mathbb{C}^{\star}$-action (i.e. all the weights have the same sign), let 
$Y :=  X^{\mathbb{C}^{\star}}$ be the fixed point set, and let $d\in\mathbb{Z}_{>0}$ such that 
the $\mathbb{C}^{\star}$-action on $X_{[d]}:= X/\mu_{d}(\mathbb{C})$ is free outside $Y_{[d]}:=Y/\mu_{d}(\mathbb{C})$.
If $\iota: Y_{[d]}\rightarrow X_{[d]}$ is the inclusion and $\mathcal{F}\in D_{\rm const}^{b}(X_{[d]})$ is $\mathbb{C}^{\star}$-equivariant,  then 
$\mathbb{H}^{\star}(X_{[d]}, \mathcal{F}) \simeq \mathbb{H}^{\star}(Y_{[d]}, \iota^{\star}\mathcal{F}).$
\end{lemma}
\begin{proof}
Same proof as in \cite[Lemma 6.5]{DL91}, \cite[Lemma 4.2]{CMM18}.
\end{proof}

\subsection{Algebraic torus actions}\label{sec-class-act}
We recall basic notions on torus actions from the perspective of Altmann-Hausen's theory \cite{AIPSV12}. We refer to \cite{Tim97, LT16, Lan16} for generalization to reductive group actions.
\\

Classification of  torus actions is intimately related to questions of convex geometry.
Let $N\simeq \mathbb{Z}^{n}$ be a lattice, let $M$ be its dual, and let 
$N_{\mathbb{Q}}: =  \mathbb{Q}\otimes_{\mathbb{Z}}N$
and $M_{\mathbb{Q}}:= \mathbb{Q}\otimes_{\mathbb{Z}} M$ be the associated $\mathbb{Q}$-vector spaces.
We denote by
$$M_{\mathbb{Q}}\times N_{\mathbb{Q}}\rightarrow \mathbb{Z}, \,\, (m,v)\mapsto \langle m, v\rangle$$
the natural pairing and by
$\mathbb{T} = \mathbb{T}_{N}:= \mathbb{C}^{\star}\otimes_{\mathbb{Z}}N\simeq (\mathbb{C}^{\star})^{n}$ the associated torus.
By \emph{Polyhedral cones} in a finite dimensional $\mathbb{Q}$-vector space we mean sets of non-negative linear combinations of finitely many vectors. The polyhedral cone $\sigma\subset N_{\mathbb{Q}}$ is  \emph{strictly convex} if  $\{0\}$ is a face, or equivalently,  if the \emph{dual cone}
$\sigma^{\vee}: =  \{m\in M_{\mathbb{Q}}\, |\, \langle m, v\rangle \geq 0\text{ for any }v\in \sigma\} $
is full-dimensional.
\\

Fix a strictly convex polyhedral cone $\sigma\subset  N_{\mathbb{Q}}$. Let
${\rm Pol}_{\sigma}(N_{\mathbb{Q}}) =  \{ \sigma + Q\, |\, Q \text{ polytopes of } N_{\mathbb{Q}}\}$
be the set of \emph{$\sigma$-polyhedra} and  let $Y$ be a normal semi-projective variety. A \emph{$\sigma$-polyhedral divisor} $\mathfrak{D}$ is a formal sum 
$$\mathfrak{D} = \sum_{Z\subset Y} \mathfrak{D}_{Z}\cdot [Z],$$
where $Z\subset Y$ runs over the set of prime divisors of $Y$, $\mathfrak{D}_{Z}\in {\rm Pol}_{\sigma}(N_{\mathbb{Q}})$
and $\mathfrak{D}_{Z} =  \sigma$ for all but finitely many prime divisors $Z$. The \emph{evaluation} at $m\in\sigma^{\vee}$ of the polyhedral divisor $\mathfrak{D}$
is the $\mathbb{Q}$-divisor 
$$\mathfrak{D}(m) :=  \sum_{Z\subset Y} \min_{v\in \mathfrak{D}_{Z}}\langle m, v\rangle \cdot [Z],$$
and its $M$-graded algebra is the subalgebra 
$$A[Y, \mathfrak{D}]:= \bigoplus_{m\in\sigma^{\vee}\cap M} H^{0}(Y, \mathcal{O}_{Y}(\mathfrak{D}(m)))\chi^{m}$$
of the group algebra $\mathbb{C}(Y)[M]$, where $\chi^{m}$ is the Laurent monomial corresponding to $m\in M$. A $\sigma$-polyhedral divisor $\mathfrak{D}$ is \emph{proper} if for any $m\in\sigma^{\vee}$,
\begin{itemize}
\item[(i)] $\mathfrak{D}(m)$ is $\mathbb{Q}$-Cartier and semi-ample, and 
\item[(ii)] $\mathfrak{D}(m)$ is big whenever $m$ is in the relative interior of the cone $\sigma^{\vee}$.
\end{itemize}
\

Note  that an action of the torus $\mathbb{T} =  \mathbb{T}_{N}$ on the affine variety $X  =  {\rm Spec}\, A$ is equivalent to endow 
$A  =  \mathbb{C}[X]$ with an $M$-grading, and that the action is faithful if the weights of $A$ generate the lattice $M$.  Passing to the quotient, we may always transform torus actions into effective ones. The following is due 
to Altmann-Hausen \cite[Theorem 7]{AIPSV12}.
\begin{theorem}
\begin{itemize}
\item[(1)] If $\sigma\subset N_{\mathbb{Q}}$ is a strictly convex polyhedral cone, $Y$ is a normal semi-projective variety,
and $\mathfrak{D}$ is a proper $\sigma$-polyhedral divisor on $Y$, then the $\mathbb{T}$-scheme $X(\mathfrak{D}) =  X(Y, \mathfrak{D}):= {\rm Spec}\, A[Y, \mathfrak{D}]$ is a normal affine variety equivariantly birational to the product $\mathbb{T}\times Y$. 
\item[(2)] Any normal affine variety with an effective  torus action arise from a proper polyhedral divisor.
\end{itemize}
\end{theorem}
Let $\xi =  (f_{1}\chi^{m_{1}}, \ldots, f_{r}\chi^{m_{r}})$ be a sequence of homogeneous elements of
$\mathbb{C}(Y)[M]$, where $Y$ is a normal semi-projective variety, $f_{1}, \ldots, f_{r}\in \mathbb{C}(Y)^{\star}$,
and $m_{1}, \ldots, m_{r}\in M$ generate the dual  ${}^{\xi}\sigma^{\vee}\subset  M_{\mathbb{Q}}$ of a strictly convex polyhedral cone ${}^{\xi}\sigma \subset N_{\mathbb{Q}}$. We denote by ${}^{\xi}\mathfrak{D}$ the ${}^{\xi}\sigma$-polyhedral divisor 
$${}^{\xi}\mathfrak{D} :=  \sum_{Z\subset Y} {}^{\xi}\mathfrak{D}_{Z}\cdot [Z],$$
where ${}^{\xi}\mathfrak{D} _{Z}:= \{v\in N_{\mathbb{Q}}\,  |\, \langle m_{i}, v\rangle \geq -{\rm ord}_{Z}(f_{i}) \text{ for }i=1, \ldots, r\}$. The following explains the relation between the homogeneous generators of $\xi$ and the polyhedral divisor ${}^{\xi}\mathfrak{D}$.
\begin{lemma}\label{lem0-gen}
Let $\mathfrak{D}$ be a proper $\sigma$-polyhedral divisor on a normal semi-projective variety $Y$. Let $\xi =  (f_{1}\chi^{m_{1}}, \ldots, f_{r}\chi^{m_{r}})$ be a sequence of homogeneous elements of  generating the algebra 
$A[Y, \mathfrak{D}]$, where $f_{1}, \ldots, f_{r}\in \mathbb{C}(Y)^{\star}$. Then ${}^{\xi}\mathfrak{D} =  \mathfrak{D}$. 
\end{lemma}
\begin{proof}
We follow the proof of \cite[Th\'eor\`eme 2.4]{Lan13}. By construction
we have 
$${}^{\xi}\mathfrak{D}(m_{i}) + {\rm div}(f_{i})\geq 0 \text{ for } 1\leq i\leq r.$$
So $A[Y, \mathfrak{D}]\subset  A[Y, {}^{\xi}\mathfrak{D}]$. Moreover, for any $m\in \sigma^{\vee} = {}^{\xi}\sigma^{\vee}$, one has 
$$H^{0}(Y, \mathcal{O}_{Y}(\mathfrak{D}(m)))\subset H^{0}(Y, \mathcal{O}_{Y}({}^{\xi}\mathfrak{D}(m))).$$
Using Lemma \ref{lem1-Qdivisor}, it follows that $\mathfrak{D}(m)\leq {}^{\xi}\mathfrak{D}(m)$ for any $m\in \sigma^{\vee}$.
But we also have $\mathfrak{D}(m_{i})  + {\rm div}(f_{i})\geq 0$ for $1\leq i\leq r$  since $f_{i}\chi^{m_{i}}\in A[Y, \mathfrak{D}]$. Hence $\langle m_{i}, v\rangle \geq -{\rm ord}_{Z}(f_{i})$ for all $v\in \mathfrak{D}_{Z}$ and $i\in\{1, \ldots, r\}$, that is 
$\mathfrak{D}_{Z}\subset {}^{\xi}\mathfrak{D}_{Z}$. This implies that $\mathfrak{D}(m) =  {}^{\xi}\mathfrak{D}(m)$ for any $m\in \sigma^{\vee}$, and so $\mathfrak{D} =  {}^{\xi}\mathfrak{D}$, ending
the proof of the lemma. 
\end{proof}
\begin{definition}
Let $Y$ be a normal variety.
A \emph{divisorial fan} over $(Y, N)$ is a finite set $\mathscr{E} =  \{\mathfrak{D}^{i}\,|\, i\in I\}$
of proper polyhedral divisors defined over semi-projective Zariski subsets of $Y$ with lattice $N$, stable by (component-wise) intersections and such that
the natural maps 
$X(\mathfrak{D}^{i}\cap \mathfrak{D}^{j})\rightarrow X(\mathfrak{D}^{i})$
are open immersions for all $i, j$, where $\mathfrak{D}^{i}\cap \mathfrak{D}^{j}$ is the intersection between
$\mathfrak{D}^{i}$ and $\mathfrak{D}^{j}$. 
\end{definition}
 A divisorial fan $\mathscr{E}$ over $(Y, N)$ defines a finite type normal $\mathbb{T}$-scheme $X(\mathscr{E})  =  X(Y, \mathscr{E})$ by gluing the charts $X(\mathfrak{D})$ for $\mathfrak{D}\in \mathscr{E}$.
When $Y$ is one-dimensional, the scheme $X(\mathscr{E})$ is a variety. Furthermore, any normal variety with effective torus action arises from a divisorial fan \cite[Section 5]{AIPSV12}. 
\\

Next is the pullback operation for polyhedral divisors.
\begin{definition}\label{def-poulebaq}
Let $Y, Y'$ be normal semi-projective varieties, let $\mathfrak{D}$ be a proper $\sigma$-polyhedral divisor over $Y$ and write  $\mathfrak{D} =  \sum_{i = 1}^{r}\mathfrak{D}_{E_{i}}\cdot E_{i},$ where $\mathfrak{D}_{E_{i}}\in {\rm Pol}_{\sigma}(N_{\mathbb{Q}})$ and $E_{i}$ is a Cartier divisor over $Y$. Let $\varphi: Y'\rightarrow Y$ be a generically finite morphism such that $Z\not\subset\varphi(Y')$ for any prime divisor $Z\subset Y$ in the support of some $E_{i}$. Then the \emph{pullback} of $\mathfrak{D}$
under $\varphi$ is the proper polyhedral divisor $\varphi^{\star}(\mathfrak{D}) =  \sum_{i = 1}^{r}\mathfrak{D}_{E_{i}}\cdot \varphi^{\star}(E_{i})$ over $Y'$. 
\end{definition}

\subsection{Toric varieties} (cf. \cite{CLS11})\label{Section-Cox}
A \emph{fan} in $N_{\mathbb{Q}}$ is a non-empty set of strictly convex polyhedral cones of $N_{\mathbb{Q}}$ stable by face relation and such that the intersection of any two elements is a mutual face of both.  For a fan $\Sigma$ of $N_{\mathbb{Q}}$ (respectively, a polyhedron $Q$)  and $a\in \mathbb{Z}_{\geq 0}$
we denote by $X_{\Sigma}$ the associated toric variety, by $\Sigma(a)$ (respectively, $Q(a)$) the set of its $a$-dimensional cones of $\Sigma$ (respectively, $a$-dimensional faces of $Q$), and by $\Sigma_{\rm max}$ the set 
of maximal cones. Elements of  $\Sigma(1)$ are called \emph{rays} of $\Sigma$. For a ray $\rho\in \Sigma(1)$, we write $v_{\rho}$ for the associated primitive generator. 
Each cone $\sigma\in \Sigma$ corresponds to a $\mathbb{T}$-orbit $O(\sigma)$ of $X_{\Sigma}$ of  dimension ${\rm dim}\, N_{\mathbb{Q}} - {\rm dim}\, \sigma$. We write $Z_{\rho} =  \overline{O(\rho)}\subset X_{\Sigma}$ for the associated toric divisor. 
Let $X_{\sigma}$ be the affine toric variety of the cone $\sigma\in \Sigma$. 
If $D$ is a torus invariant Cartier divisor, then we call \emph{trivialization}
of the line bundle $\mathcal{O}_{X_{\Sigma}}(D)$ the data $(X_{\sigma}, \chi^{m_{\sigma}})_{\sigma\in \Sigma_{\rm max}}$ such 
that $D_{|X_{\sigma}}  =  {\rm div}\, \chi^{m_{\sigma}}.$ 
\\

We write 
$|\Sigma| :=  \bigcup_{\sigma \in \Sigma}\sigma\subset N_{\mathbb{Q}}$
for the \emph{support} of $\Sigma$. Note that $X_{\Sigma}$ is complete if and only if $|\Sigma| =  N_{\mathbb{Q}}$. We say 
that $\Sigma$ is \emph{simplicial} if each cone $\sigma\in\Sigma$ is generated by a subset of a basis of $N_{\mathbb{Q}}$. 
Given any two fans $\Sigma, \Sigma'$ with respect to lattices $N, N'$, we call \emph{fan morphism} between $\Sigma$ and $\Sigma'$
a linear map $\phi: N\rightarrow N'$ such that for any $\sigma\in \Sigma$ there is $\sigma'\in \Sigma'$ such that $\phi(\sigma\cap N)\subset  \sigma'$. Note that fan morphisms from $\Sigma$ to $\Sigma'$ corresponds \emph{toric morphisms} from $X_{\Sigma}$ to $X_{\Sigma'}$, i.e. morphisms that are equivariant. 
\\

 Homogeneous coordinate theory for a toric variety $X_{\Sigma}$ (without torus factor)
is a presentation 
$X_{\Sigma}=  (\mathbb{A}_{\mathbb{C}}^{r}\setminus Z_{\Sigma})/\!\!/\mathbb{G}_{\Sigma},$
where $\mathbb{G}_{\Sigma}$ is a diagonalizable group and $Z_{\Sigma}$ is a $\mathbb{G}_{\Sigma}$-invariant closed subset.
Example to have in mind is the projective space $\mathbb{P}^{n}_{\mathbb{C}} =  (\mathbb{A}^{n+1}_{\mathbb{C}}\setminus \{0\})/\mathbb{C}^{\star}$. More precisely,
set $r:= \Sigma(1)$ and consider the polynomial ring $R:= \mathbb{C}[X_{\rho}\,|\, \rho\in\Sigma(1)]$.
For $\sigma\in \Sigma$ write 
$\varpi_{\sigma}:= \prod_{\rho\in \Sigma(1)\setminus \sigma(1)} X_{\rho}.$ We define the \emph{irrelevant ideal} as $B(\Sigma):= (\varpi_{\sigma}\,|\, \sigma\in \Sigma)\subset R$ and set $Z_{\Sigma}:= \mathbb{V}(B(\Sigma))\subset {\rm Spec}\, R$.  Note that $\{\varpi_{\sigma}\, |\, \sigma\in \Sigma_{\rm max}\}$ is a set of minimal generators of $B(\Sigma)$. 
The diagonalizable group $\mathbb{G}_{\Sigma}$ is defined as the set of functions $g:\Sigma(1)\rightarrow \mathbb{C}^{\star}$ satisfying
$$\prod_{\rho\in \Sigma(1)} g(\rho)^{\langle m , v_{\rho}\rangle} =  1\text{ for any } m\in M,$$
together with the pointwise multiplication law $(g\cdot g')(\rho):= g(\rho)\cdot g'(\rho)$ for all $g,g'\in \mathbb{G}_{\Sigma}$
and $\rho\in \Sigma(1)$. The group $\mathbb{G}_{\Sigma}$ \emph{diagonally} acts  on the space ${\rm Spec}\, R$, meaning that the action on $R$ is defined as $g\cdot X_{\rho}:= g(\rho)X_{\rho}$ for any $g\in \mathbb{G}_{\Sigma}$. The ring $R$ together with the graduation induced by the $\mathbb{G}_{\Sigma}$-action is the \emph{homogeneous coordinate ring} of $X_{\Sigma}$. The following is due to Cox \cite{Cox95}. 
\begin{theorem}\label{theo-Cox1} \cite[Theorem 5.1.11]{CLS11}
Consider
the fan $\Sigma_{\rm Cox}$ of $\mathbb{Q}^{r} =  \bigoplus_{\rho\in \Sigma(1)}\mathbb{Q}\cdot e_{\rho}$ generated by the cones
$$ \sigma_{\rm Cox} =  \sum_{\rho\in \sigma(1)} \mathbb{Q}_{\geq 0}\cdot e_{\rho}.$$
Then $X_{\Sigma_{\rm Cox}} =  \mathbb{A}^{r}_{\mathbb{C}}\setminus Z_{\Sigma}$ and the morphism $\varepsilon:  \mathbb{A}^{r}_{\mathbb{C}}\setminus Z_{\Sigma}\rightarrow X_{\Sigma},$ inducing by the linear map $e_{\rho}\mapsto v_{\rho}$,
is constant on the $\mathbb{G}_{\Sigma}$-orbits for the diagonal $\mathbb{G}_{\Sigma}$-action on the total coordinate space 
$\mathbb{A}_{\mathbb{C}}^{r} =  {\rm Spec}\, R$. More precisely, it factors through an isomorphism
$$ (\mathbb{A}_{\mathbb{C}}^{r}\setminus Z_{\Sigma})/\!\!/\mathbb{G}_{\Sigma} = \{\text{closed } \mathbb{G}_{\Sigma}\text{-orbits of }\mathbb{A}_{\mathbb{C}}^{r}\setminus Z_{\Sigma}\}\simeq X_{\Sigma}.$$
Moreover, the fan $\Sigma$ is simplicial if and only if all the $\mathbb{G}_{\Sigma}$-orbits of $\mathbb{A}^{r}_{\mathbb{C}}\setminus Z_{\Sigma}$ are closed. 
\end{theorem}
A consequence of Theorem \ref{theo-Cox1} is the \emph{toric Hilbert's Nullstellensatz}.
\begin{theorem}\cite[Propositions 5.2.4, 5.2.6]{CLS11}
For a homogeneous ideal $I\subset R$, the subset
$$V_{h}(I):= \{y\in X_{\Sigma}\, |\, \text{there is }x\in\varepsilon^{-1}(y)\text{ with } f(x) = 0\text{ for all }f\in I\}$$
is a Zariski closed subset of $X_{\Sigma}$ and all the Zariski closed subsets of $X_{\Sigma}$ arises in this way. Moreover, if $\Sigma$ is simplicial, then
the map $I\mapsto V_{h}(I)$ induces a one-to-one correspondence between the radical homogeneous ideals $I\subset  B(\Sigma)\subset R$ and the Zariski closed subsets of $X_{\Sigma}$. 
\end{theorem}
We say that $x, y\in\mathbb{A}^{r}_{\mathbb{C}}\setminus Z_{\Sigma}$ are equivalent if 
$\varepsilon^{-1}(x) =  \varepsilon^{-1}(y)$. After numbering the rays of $\Sigma$, write $x =  (x_{1}, \ldots, x_{r})$ for a point of $\mathbb{A}^{r}_{\mathbb{C}}\setminus Z_{\Sigma}$ and denote by $[x_{1}:\ldots: x_{r}]_{\Sigma}$ its equivalence class called \emph{the Cox's coordinates}. In this way,  $$\mathbb{V}_{h}(I)=\{ [x_{1}:\ldots: x_{r}]_{\Sigma}\in X_{\Sigma}\,|\, f(x_{1}, \ldots, x_{r}) =  0\text{ for any } f\in I\}.$$
\subsection{Complexity-one case}
Torus actions of complexity one are described by divisorial fans over curves. In this context, the globalization  simplifies since a finite set  $\mathscr{E} =  \{\mathfrak{D}^{i}\,|\, i\in I\}$ of proper polyhedral divisors $\mathfrak{D}^{i}$ defined over dense open subsets $Y_{i}$ of  a smooth projective curve $Y$ and  stable by intersection is a divisorial fan if and only if:
\begin{itemize}
\item[(1)] for all $i,j\in I$ and for any $z\in Y$ the polyhedron $\mathfrak{D}_{z}^{i}\cap \mathfrak{D}_{z}^{j}$
is either empty or a common face of $\mathfrak{D}_{z}^{i}$ and $\mathfrak{D}^{j}_{z}$, and
\item[(2)] we have ${\rm deg}\, \mathfrak{D}^{i}\cap \mathfrak{D}^{j} =  \sigma_{i}\cap \sigma_{j}\cap {\rm deg}\,\mathfrak{D^{j}}$ for all $i,j\in I$, where $\sigma_{i}$ and $\sigma_{j}$ are the strictly convex polyhedral cones associated with
$\mathfrak{D}^{i}$ and $\mathfrak{D}^{j}$. Here for a $\sigma$-polyhedral divisor $\mathfrak{D}$ over a curve $Y_{0}$, we set ${\rm deg}\, \mathfrak{D} := \sum_{z\in Y_{0}}\mathfrak{D}_{z}\subset N_{\mathbb{Q}}$ if $Y_{0}$ is complete and ${\rm deg}\, \mathfrak{D} =  \emptyset$ otherwise.
\end{itemize}
Moreover, using valuative criterion of properness, the $\mathbb{T}$-variety $X(\mathscr{E})$ is complete
if and only if 
$$ \bigcup_{i\in I}\bigcup_{z\in Y_{i}}\{z\}\times \mathfrak{D}_{z}^{i}  =  Y\times N_{\mathbb{Q}}.$$
Note that the  relative spectra ${\rm Spec}_{Y_{i}}\, \mathcal{A}_{\mathfrak{D}^{i}}$ for $\mathfrak{D}^{i}\in \mathscr{E}$, 
where 
$$\mathcal{A}_{\mathfrak{D^{i}}}:=  \bigoplus_{m\in \sigma^{\vee}\cap M}\mathcal{O}_{Y_{i}}(\mathfrak{D}^{i}(m))\chi^{m},$$
 glue together into a $\mathbb{T}$-variety $\widetilde{X}$. The natural map $\pi: \widetilde{X}\rightarrow X$, where $X  =  X(\mathscr{E})$, is the \emph{contraction map.} By \cite[Section 3, Lemma 1]{Vol10} the normalization of the graph of the rational map $\iota: X\dashrightarrow Y$, induced by the inclusion $\mathbb{C}(X)^{\mathbb{T}}\subset \mathbb{C}(X)$ identifies with the contraction space $\widetilde{X}$ and the natural projection to $X$ with $\pi$.
\\

Next definition formulates the passage to contraction spaces via divisorial fans.
\begin{definition}\label{def-contrat-theta5758}
Given a divisorial fan $\mathscr{E} =  \{\mathfrak{D}^{i}\,|\, i\in I\}$ over $(Y, N)$, where $Y$ is a smooth projective curve, 
a \emph{contraction divisorial fan} of $\mathscr{E}$ is a divisorial fan of the form
$$\widetilde{\mathscr{E}} =  \{\mathfrak{D}^{i}_{| U_{j}\cap Y_{i}} :=  \iota_{i,j}^{\star}(\mathfrak{D}^{i})\,|\, i\in I, j\in J\}.$$
Here $(U_{j})_{j\in J}$ is a finite affine open covering of $Y$ and $\iota_{i, j}: Y_{i}\cap U_{j}\rightarrow Y$ is the inclusion. From the 
previous discussion, $X(\widetilde{\mathscr{E}})$ is exactly the contraction space of $X(\mathscr{E})$. 
\end{definition}

We end this section with the description of the integral closure of affine varieties with torus action of complexity one. 
See \cite{Lan15} for a version over arbitrary fields. 
\begin{theorem}\label{theo-norma}\cite[Proposition 3.9]{FZ03},  \cite[Th\'eor\`eme 2.4]{Lan13}
Let $Y$ be a smooth curve and consider the subalgebra $A:= \mathbb{C}[Y][f_{1}\chi^{m_{1}}, \ldots, f_{r}\chi^{m_{r}}]\subset \mathbb{C}(Y)[M]$, where $m_{1}, \ldots, m_{r}\in M$ and $f_{1}, \ldots, f_{r}\in \mathbb{C}(Y)^{\star}$. Set $\xi =  (f_{1}\chi^{m_{1}}, \ldots, f_{r}\chi^{m_{r}})$. Assume that $A$ and $\mathbb{C}(Y)[M]$ have same fraction field. Then ${}^{\xi}\mathfrak{D}$ is proper and the normalization of $A$ identifies with $A[Y, {}^{\xi}\mathfrak{D}]$.

\end{theorem}
\section{Intersection cohomology and actions of finite groups}\label{sec-three}
In this section, we study the pullback on intersection cohomology complexes of quotient maps of finite group actions.
\subsection{Maps preserving intersection cohomology complexes}\label{subsec-pullbackfiniteGaction}
 \begin{definition}\label{def1}
We say that a finite dominant morphism of varieties $f:X\rightarrow Y$ \emph{preserves the intersection cohomology complexes}
if there are proper birational morphisms $\varphi: Y'\rightarrow Y$ and $\bar{\varphi}: X'\rightarrow X$ with $X'$ smooth and a  Cartersian square
 $$\xymatrix{
      X'\ar[r]^{\bar{\varphi}} \ar[d] ^{\bar{f}} & X\ar[d]^{f}\\ Y' \ar[r]^{\varphi} & Y }$$ 
such that $\bar{f}^{\star}IC_{Y'}\simeq IC_{X'}$. We say further that $f$ \emph{strictly preserves the intersection cohomology complexes} if we can choose $\bar{\varphi}$ semi-small.
\end{definition}
\begin{lemma}\label{lem-pure}
If $\mathcal{F}, \mathcal{G}\in D^{b}_{\rm const}(X)$ and $\mathcal{F}\oplus \mathcal{G}$ is semi-simple, then 
so is $\mathcal{F}$. Moreover, if  $\mathcal{F}\oplus \mathcal{G}$ is perverse, then so is $\mathcal{F}$.
\end{lemma}
\begin{proof} First claim is \cite[Lemma 15]{BM99}. Second claim is consequence of  \cite[Proposition 1.8]{Wil17}.
\end{proof}
\begin{proposition}\label{propIC}
Let $f: X\rightarrow Y$ be a finite dominant morphism of varieties. If $f$ preserves the intersection cohomology complexes,
then $f^{\star} IC_{Y}$ is semi-simple. If further $f$ strictly preserves the intersection cohomology complexes, then
$f^{\star} IC_{Y}\simeq IC_{X}$.
\end{proposition}
\begin{proof}
Take the commutative diagram
 $$\xymatrix{
      X'\ar[r]^{\bar{\varphi}} \ar[d] ^{\bar{f}} & X\ar[d]^{f}\\ Y' \ar[r]^{\varphi} & Y }$$ 
 of  Definition \ref{def1}. By base change \cite[Theorem 2.3.26]{Dim04} we have
$f^{\star}\varphi_{\star} IC_{Y'}\simeq \bar{\varphi}_{\star}\bar{f}^{\star} IC_{Y'}\simeq \bar{\varphi}_{\star} IC_{X'}.$
The decomposition theorem gives 
$$\varphi_{\star}IC_{Y'}\simeq IC_{Y}\oplus \bigoplus_{a\in I} (\iota_{a})_{\star}IC_{Y_{a}}(\mathscr{L}_{a})[r_{a}],$$
where $\iota_{a}: Y_{a}\rightarrow Y'$ are inclusions of closed subvarieties and $\mathscr{L}_{a}$ are 
local systems defined on the smooth locus. Set
$$\mathcal{F}:= f^{\star} IC_{Y}\text{ and } \mathcal{G}:=f^{\star}\left( \bigoplus_{a\in I} (\iota_{a})_{\star}IC_{Y_{a}}(\mathscr{L}_{a})[r_{a}]\right).$$
 Then
$\mathcal{F}\oplus \mathcal{G}\simeq \bar{\varphi}_{\star} IC_{X'}$
is semi-simple. So by Lemma \ref{lem-pure}, $\mathcal{F}$ is semi-simple. Assume further that $\bar{\varphi}$ is semi-small.
Then the examination of the supports gives
$$\mathcal{F}\simeq IC_{X}\oplus \bigoplus_{\beta\in J}(\iota_{\beta})_{\star}IC_{W_{\beta}}(\mathscr{M}_{\beta})^{\oplus s_{\beta}},$$
where $\iota_{\beta}:W_{\beta}\rightarrow X$ are strict inclusions of  closed subvarieties. 
Assume, toward a contradiction, that each $\mathscr{M}_{\beta}$ is nonzero.
Take a subvariety $W_{\beta_{0}}$ for $\beta_{0}\in J$ of maximal dimension. Definitions of
$IC_{X}$ and $IC_{Y}$ come with algebraic Whitney stratifications
$$X =  \bigcup_{\lambda \in \Lambda}X_{\lambda}\text{ and } Y =  \bigcup_{\alpha \in \Delta}Y_{\alpha}.$$
Let $\lambda_{1}\in \Lambda$ (respectively, $\alpha_{1}\in \Delta$) be the unique index such that the generic point (respectively,
the image by $f$ of the generic point) of $W_{\beta_{0}}$ is contained in the stratum $X_{\lambda_{1}}$ (respectively, $Y_{\alpha_{1}}$).
Since $f$ is finite, we have ${\rm dim}\, W_{\beta_{0}}\leq {\rm dim}\, Y_{\alpha_{1}}$. Note that ${\rm dim}\, W_{\beta_{0}}\leq {\rm dim}\, X_{\lambda_{1}}$. Moreover, the fact that $W_{\beta_{0}}$ is of maximal dimension implies that there is $x\in W_{\beta_{0}}$ such that $x\in X_{\lambda_{1}}$, $f(x) \in Y_{\alpha_{1}}$ and with the condition that $x$ is not contained in any irreducible subvarieties $W_{\beta}$ with $\beta\neq \beta_{0}$. The open stratum and support conditions imply
$$\mathcal{H}^{-{\rm dim}\, W_{\beta_{0}}}(IC_{Y})_{f(x)} =  0 =  \mathcal{H}^{-{\rm dim}\, W_{\beta_{0}}}(IC_{X})_{x} .$$
Therefore
$$\mathscr{M}_{\beta_{0}, x}^{\oplus s_{\beta_{0}}} =  \mathcal{H}^{-{\rm dim}\, W_{\beta_{0}}}(IC_{X})_{x}\oplus \bigoplus_{\beta\in J} \mathcal{H}^{-{\rm dim}\, W_{\beta_{0}}}((\iota_{\beta})_{\star}IC_{W_{\beta}}(\mathscr{M}_{\beta})^{\oplus s_{\beta}})_{x} =  \mathcal{H}^{-{\rm dim}\, W_{\beta_{0}}}(\mathcal{F})_{x} $$
$$ =  \mathcal{H}^{-{\rm dim}\, W_{\beta_{0}}}(f^{\star}IC_{Y})_{x}  =  \mathcal{H}^{-{\rm dim}\, W_{\beta_{0}}}(IC_{Y})_{f(x)} =  0.$$
So $\mathscr{M}_{\beta_{0}} = 0$, yielding a contradiction. We conclude that $f^{\star} IC_{Y} \simeq \mathcal{F}\simeq IC_{X}$, as required. 
\end{proof}
\begin{lemma}\label{lem-G-smooth}
Let $X$ be a smooth quasi-projective variety with action of a finite group $G$. Set $Y =  X/G$.
 Then
$IC_{Y}\simeq \mathbb{Q}_{Y}[\dim\, Y]$.
\end{lemma}
\begin{proof}
See \cite[Section 5, Proposition 3]{GS93}.
\end{proof}
\begin{lemma}\label{lem-GactIC}
Let $X$ be a quasi-projective variety with action of a finite group $G$. Then the quotient map $f: X\rightarrow Y$ preserves 
the intersection cohomology complexes. 
\end{lemma}
\begin{proof}
By \cite{AW97} there is a $G$-equivariant resolution of singularities $\bar{\varphi}: X'\rightarrow X$. Denote
by $\bar{f}: X'\rightarrow Y'$ the quotient map. Then we have a Cartesian square
 $$\xymatrix{
      X'\ar[r]^{\bar{\varphi}} \ar[d] ^{\bar{f}} & X\ar[d]^{f}\\ Y' \ar[r]^{\varphi} & Y }$$ 
with $\bar{\varphi}$ proper since $\varphi$ is. Moreover, if $U\subset X$ is a dense Zariski
open subset such that  
$\bar{\varphi}_{|\bar{\varphi}^{-1}(U)}: \bar{\varphi}^{-1}(U)\rightarrow U$
is an isomorphism, then
$\varphi_{|\varphi^{-1}(V)}: \varphi^{-1}(V)\rightarrow V$
is an isomorphism, where $W  =  \bigcup_{g\in G} g\cdot U$ and $V = W/G$. Thus $\varphi$ is birational. 
Finally, by Lemma \ref{lem-G-smooth}, 
$$\bar{f}^{\star} IC_{Y'}\simeq \bar{f}^{\star}\mathbb{Q}_{Y'}[\dim\, Y']\simeq \mathbb{Q}_{X'}[\dim\, X']\simeq IC_{X'},$$
proving that $f$ preserves the intersection cohomology complexes.
\end{proof}
\begin{proposition}\label{prop-IC-G-X}
Let $X$ be a quasi-projective variety with action of a finite group $G$.  Denote by $f: X\rightarrow Y$ the quotient map. 
Then $f^{\star}IC_{Y}$ is semi-simple.  Assume further that $X$ has an equivariant semi-small resolution
of singularities.
Then $f^{\star} IC_{Y}\simeq IC_{X}$. 
\end{proposition}
\begin{proof}
This follows from Proposition \ref{propIC} and Lemma \ref{lem-GactIC}.
\end{proof}
\begin{example}
For a dominant finite morphism $f:X\rightarrow Y$, $f^{\star}IC_{Y}$ is generally not isomorphic to $IC_{X}$. Indeed, let $X$ be an affine cone\footnote{Example based on an idea by Williamson communicated on MathOverflow.} over an elliptic curve
and denote by $x$ its vertex. Let $f: X\rightarrow  Y =  \mathbb{A}_{\mathbb{C}}^{2}$ be the finite morphism
given by Noether normalization. Then from
\cite[Lemma 2.1]{Fie91}, \cite[Lemma 6.5]{DL91}, the stalk at $x$ is
$$\mathcal{H}^{-2}(IC_{X})_{x}  =  \mathbb{Q},\, \mathcal{H}^{-1}(IC_{X})_{x} =  \mathbb{Q}^{2}, \text{ and }\mathcal{H}^{j}(IC_{X})_{x}= 0 \text{ for } j\neq -2,-1,$$
while one has 
$$\mathcal{H}^{-2}(\mathbb{Q}_{X}[2])_{x}  =  \mathbb{Q}\text{ and } \mathcal{H}^{j}(\mathbb{Q}_{X}[2])_{x} = 0 \text{ for }j\neq -2.$$
\end{example}

\subsection{Finite group actions on toric varieties}
Our aim is to prove the following proposition.
\begin{proposition}\label{theo-finitegrouptoric}
Let $X$ be a toric variety for the torus $\mathbb{T}$, let $G\subset \mathbb{T}$ be a finite subgroup and
denote by $f: X\rightarrow Y =  X/G$ the quotient map for the natural $G$-action. Then $f^{\star}IC_{Y}\simeq IC_{X}$.
\end{proposition}
Proof of Proposition \ref{theo-finitegrouptoric} uses the toric decomposition theorem \cite{CMM18}. 
\begin{theorem}\cite[Theorem D]{CMM18}\label{decomptoric}
Let $\Sigma$ be a fan, let $\Sigma'$ be a fan subdivision of $\Sigma$ and let $q: X_{\Sigma'}\rightarrow X_{\Sigma}$
be the corresponding toric morphism. Then we have 
$$q_{\star}IC_{X_{\Sigma'}} \simeq \bigoplus_{b\in \mathbb{Z}}\bigoplus_{\tau \in \Sigma}(\iota_{\tau})_{\star}IC_{V(\tau)}^{\oplus s_{b, \tau}}[-b],$$
where $\iota_{\tau}: V(\tau)\rightarrow X_{\Sigma}$ is the inclusion of the closure of the orbit $O(\tau)\subset X_{\Sigma}$.
\end{theorem}

\begin{lemma}\label{lem-toricmultN0unic}
Let $\Sigma$ be a fan of $N_{\mathbb{Q}}$ and let $\Sigma'$ be a simplicial fan subdivising $\Sigma$. Consider a lattice $N_{0}\subset N_{\mathbb{Q}}$ containing $N$ with $[N_{0}: N]$ finite. Denote by $X_{\Sigma, N_{0}}, X_{\Sigma', N_{0}}$
the associated toric vareties with respect to the lattice $N_{0}$. If $q: X_{\Sigma'}\rightarrow X_{\Sigma}$ and $q_{0}: X_{\Sigma', N_{0}}\rightarrow X_{\Sigma, N_{0}}$ are the toric modifications and 
$$q_{\star}IC_{X_{\Sigma'}} \simeq \bigoplus_{b\in \mathbb{Z}}\bigoplus_{\tau \in \Sigma}(\iota_{\tau})_{\star}IC_{V(\tau)}^{\oplus s_{b, \tau}}[-b] \text{ and } (q_{0})_{\star}IC_{X_{\Sigma', N_{0}}} \simeq \bigoplus_{b\in \mathbb{Z}}\bigoplus_{\tau \in \Sigma}(\iota_{\tau})_{\star}IC_{V(\tau)_{N_{0}}}^{\oplus s_{b, \tau, N_{0}}}[-b]$$
are the decompositions of Theorem \ref{decomptoric}. Then $s_{b, \tau} =  s_{b, \tau, N_{0}}$ for all $b, \tau$. 
In other words, the multiplicities of the toric decomposition theorem do depend on the ambient lattice $N$. 
\end{lemma}
\begin{proof}
Consequence of \cite[Corollary 7.5]{CMM18} and the fact that the polynomials $R_{\tau, \sigma}(T) =  \sum_{j\in \mathbb{Z}}{\rm dim}\, \mathcal{H}^{j}(IC_{V(\tau)})_{x_{\sigma}}T^{j}$ for $x_{\sigma}\in O(\sigma)$
do not depend on the lattice $N$ (see \cite[Section 1, Theorem 1.2]{Fie91}).
\end{proof}

\begin{proof}[Proof of Proposition \ref{theo-finitegrouptoric}] We proceed by induction on  the dimension of $X$. 
For dimensions less than or equal to two, every resolution of singularities is semi-small; so, by Proposition \ref{prop-IC-G-X}, the statement holds in this case.
Assume now that the claim is true for all varieties of dimension strictly smaller than the dimension of 
$X$
and write $X  =  X_{\Sigma}$ for a fan $\Sigma$  of $N_{\mathbb{Q}}$. Note that $Y = X_{\Sigma, N_{0}}$, where $N_{0}\subset N_{\mathbb{Q}}$ is a lattice containing $N$
with $[N_{0}: N]$ finite. Consider $q_{0}: X_{\Sigma', N_{0}}\rightarrow X_{\Sigma, N_{0}}$
the morphism induced by a fan subdivision $\Sigma'$ with $X_{\Sigma'}$ smooth. We have a Cartesian commutative diagram
 $$\xymatrix{
      X_{\Sigma'}\ar[r]^{q} \ar[d] ^{f_{0}} & X_{\Sigma}\ar[d]^{f}\\ X_{\Sigma', N_{0}} \ar[r]^{q_{0}} & X_{\Sigma, N_{0}} .}$$ 
By base change \cite[Theorem 2.3.26]{Dim04},
$$ f^{\star}(q_{0})_{\star} IC_{X_{\Sigma', N_{0}}}\simeq q_{\star} f^{\star}_{0} IC_{X_{\Sigma', N_{0}}}.$$
Write $\gamma_{\tau}: V(\tau)/G\rightarrow X_{\Sigma, N_{0}}$ for the inclusion and let $f_{\tau}: V(\tau)\rightarrow V(\tau)/G$ be the quotient. On one hand Theorem \ref{decomptoric} yields 
$$ f^{\star}(q_{0})_{\star} IC_{X_{\Sigma', N_{0}}}\simeq f^{\star}\left( \bigoplus_{b\in \mathbb{Z}}\bigoplus_{\tau\in \Sigma}(\gamma_{\tau})_{\star}IC_{V(\tau)/G}^{\oplus s_{b, \tau}}[-b]\right).$$
Note that the commutative diagram 
 $$\xymatrix{
      V(\tau)\ar[r]^{\iota_{\tau}} \ar[d] ^{f_{\tau}} & X_{\Sigma}\ar[d]^{f}\\ V(\tau)/G \ar[r]^{\gamma_{\tau}} & X_{\Sigma, N_{0}} }$$ 
is Cartesian.
By base change and induction we have  
$$f^{\star} (\gamma_{\tau})_{\star} IC_{V(\tau)/G}\simeq (\iota_{\tau})_{\star}f_{\tau}^{\star}IC_{V(\tau)/G}\simeq (\iota_{\tau})_{\star} IC_{V(\tau)}$$
for any nonzero $\tau$ and hence
$$ f^{\star}(q_{0})_{\star} IC_{X_{\Sigma', N_{0}}}\simeq f^{\star} IC_{X_{\Sigma, N_{0}}}\oplus \bigoplus_{b\in\mathbb{Z}}\bigoplus_{\tau\in \Sigma\setminus \{\{0\}\}}(\iota_{\tau})_{\star}IC_{V(\tau)}^{\oplus s_{b, \tau}}[-b].$$
On the other hand, by Lemma \ref{lem-G-smooth},  $f_{0}^{\star} IC_{X_{\Sigma', N_{0}}}\simeq IC_{X_{\Sigma'}}$ and thus by Theorem \ref{decomptoric} and Lemma \ref{lem-toricmultN0unic},
\begin{equation}\label{Eq:qf}
    q_{\star} f_{0}^{\star} IC_{X_{\Sigma', N_{0}}}\simeq  q_{\star} IC_{X_{\Sigma'}}\simeq \bigoplus_{b\in \mathbb{Z}}\bigoplus_{\tau \in \Sigma}(\iota_{\tau})_{\star}IC_{V(\tau)}^{\oplus s_{b, \tau}}[-b].
\end{equation}
This implies that $$f^{\star} IC_{X_{\Sigma, N_{0}}} \simeq \bigoplus_{b\in \mathbb{Z}}
\bigoplus_{\tau \in \Sigma}(\iota_{\tau})_{\star}IC_{V(\tau)}^{\oplus d_{b, \tau}}[-b]$$
 and so by Equation (\ref{Eq:qf})
$$\bigoplus_{b\in \mathbb{Z}}\bigoplus_{\tau \in \Sigma}(\iota_{\tau})_{\star}IC_{V(\tau)}^{\oplus s_{b, \tau}}[-b]\simeq q_{\star} IC_{X_{\Sigma'}}$$
$$\simeq \bigoplus_{b\in \mathbb{Z}}
\bigoplus_{\tau \in \Sigma}(\iota_{\tau})_{\star}IC_{V(\tau)}^{\oplus d_{b, \tau}}[-b] \oplus \bigoplus_{b\in \mathbb{Z}}\bigoplus_{\tau\in \Sigma\setminus \{\{0\}\}}(\iota_{\tau})_{\star}IC_{V(\tau)}^{\oplus s_{b, \tau}}[-b].$$
The uniqueness of the decomposition gives
$$d_{b, \tau} = \begin{cases}
 0 & ~\text{if }~ b\in \mathbb{Z}\setminus\{0\}~ \text{ or } ~\tau\in \Sigma\setminus \{0\},\\
1 & ~\text{if } b=0 ~\text{ and }~ \tau=0.
\end{cases}$$
So $f^{\star} IC_{X_{\Sigma, N_{0}}}\simeq IC_{X_{\Sigma}}$, as required. 
\end{proof}

\subsection{Seifert torus bundles} In this section, we prove that the local systems in the decomposition theorem for contraction maps
of torus actions of complexity one are trivial.
\begin{definition}\cite[Section 2.3]{AL21}
A \emph{Seifert torus bundle} with fiber $\mathbb{T}:= (\mathbb{C}^{\star})^{n}$ is  a homogeneous fiber space 
$E:= \mathbb{T}\times^{G} X$, where $G\subset \mathbb{T}$ 
is a finite group and $X$ is a quasi-projective $G$-variety, together with the projection $\varepsilon: E\rightarrow B = X/G$.
\end{definition}
For a Seifert torus bundle 
$$ \varepsilon: E:= \mathbb{T}\times^{G} X\rightarrow B := X/G,$$
the $G$-action on $\mathbb{T}\times X$ by translation on the first factor $\mathbb{T}$ induces a $G$-action 
on $E$. We have $E/G\simeq \mathbb{T}\times X/G$ and a commutative
diagram 
$$
\xymatrix{
E \ar[r]^{f} \ar[dr]_{\varepsilon} & E/G \ar[d]^{p} \\
 & B,
}
$$
where $f: E\rightarrow E/G$ is the quotient and $p: E/G \rightarrow B$ is, after identification, 
the projection $\mathbb{T}\times B\rightarrow B$ on the second factor \cite[Lemma 2.7]{AL21}. 
\\

 Recall the notation $X$ from Section \ref{sec-two}, that is $X$
a normal variety with an effective complexity-one $\mathbb{T}$-action.
Let $\pi:\widetilde{X}\rightarrow X$ be the contraction map. Assume that $\pi$ is not an isomorphism. 
Let $q: \widetilde{X}\rightarrow C$
be the global quotient onto a smooth projective curve $C$. Let $E\subset X$ be the image of the exceptional locus of $\pi$. We fix an orbit $O\subset E$.
For any $z\in C$ let $O_{z}$ be the unique orbit contained  in the subset $\pi^{-1}(O)\cap q^{-1}(\{z\})$. The following is a relative version of Luna's slice theorem. 
\begin{lemma}\label{l-diagram}\cite[Lemma 4.4]{AL21}
Using the notation as above, we set
$$\widetilde{X}_{O}:= \{ x\in \widetilde{X}\,|\, O_{q(x)}\subset \overline{\mathbb{T}\cdot x}\}.$$
Then $\widetilde{X}_{O}$ is a Zariski open subset containing $\pi^{-1}(O)$, which is affine over $C$. The image $X_{O} =  \pi(\widetilde{X}_{O})$ is a Zariski affine open subset of $X$ in which all the orbit closures contain $O$ as closed orbit. Moreover, the map $\pi$ induces
a Cartesian square
 $$\xymatrix{
      \widetilde{X}_O \simeq \mathbb{T}_O\times^{G}\widetilde{X}_1\,\,\,\ar[r]^{\varepsilon} \ar[d] ^{\pi} & \widetilde{X}_{1}/G \ar[d]^{\pi_{1}}\\X_O \simeq \mathbb{T}_O\times^{G}X_1\,\, \ar[r]^{\varepsilon_{1}} & X_{1}/G,}$$ 
where the horizontal arrows are Seifert torus bundles and the vertical ones are proper morphisms. The torus $\mathbb{T}_{O}$ is the 
quotient of $\mathbb{T}$ by the neutral connected component of the stabilizer $\mathbb{T}_{x}$ at a point $x\in O$, and the group $G$
is the group of connected components of $\mathbb{T}_{x}$. Finally, $X_{1}/G$ has a unique fixed point under the action of $\mathbb{T}/\mathbb{T}_{x}$,
whose preimage under the map $X_{O}\rightarrow X_{1}/G$ is $O$. 
\end{lemma}
Next we discuss on the toroidal structure of  contraction spaces. 
\begin{remark}\label{rem-keyetale-struct-tor}
Any $\mathbb{T}$-variety $X(\mathfrak{D})$ associated with a $\sigma$-polyhedral divisor $\mathfrak{D}$ over a smooth affine curve $Y$ is locally toric for the \'etale topology. Indeed, take $y\in Y$ and  a uniformizer $\varpi_{y}$ of the local ring  $\mathcal{O}_{Y, y}$. 
Then
$\psi: V\rightarrow \mathbb{A}^{1}_{\mathbb{C}}$, $z\mapsto \varpi_{y}(z)$ is \'etale on a Zariski open subset 
$V$ of $Y$ containing $y$. Considering
$$\mathfrak{D}_{0}: =  \sum_{z\in V} \mathfrak{D}_{z}\cdot [\psi(z)]\text{ and } \mathfrak{D}_{|V} :=  \sum_{z\in V}\mathfrak{D}_{z}\cdot z$$
and assuming $\psi^{-1}(0) =  \{y\}$ and $\{z\in V\,|\, \mathfrak{D}_{z}\neq \sigma\} =  \{y\}$, the map $\psi$ induces an isomorphism $X(\mathfrak{D}_{|V})\simeq V\times_{\mathbb{A}^{1}_{\mathbb{C}}}X(\mathfrak{D}_{0})$. Hence the composition of the projection $ V\times_{\mathbb{A}^{1}_{\mathbb{C}}}X(\mathfrak{D}_{0})\rightarrow X(\mathfrak{D}_{0})$ and the inclusion $ X(\mathfrak{D}_{0})\hookrightarrow X_{\sigma_{y}}$,
where $X_{\sigma_{y}}$ is the affine toric variety associated with $\sigma_{y} =  {\rm Cone}((\sigma\times\{0\})\cup (\mathfrak{D}_{y}\times\{1\}))$, gives an \'etale map $\phi: X(\mathfrak{D}_{|V})\rightarrow X_{\sigma_{y}}$ and the desired toroidal structure.
\end{remark}
\begin{lemma}\label{l-Seifert-pullback}
With the notation as in Lemma \ref{l-diagram}, we have $\varepsilon^{\star}IC_{\widetilde{X}_{1}/G}[r]\simeq IC_{\widetilde{X}_{O}}$, where $r  =  {\rm dim}\, O$.
\end{lemma}
\begin{proof}
By Remark \ref{rem-keyetale-struct-tor}, the variety $\widetilde{X}_{O}$ has an \'etale open covering $(\phi_{i}:U_{i}\rightarrow X_{\sigma_{i}})_{1\leq i\leq s}$.
Consider the $G$-action on $\widetilde{X}_{O}\simeq \mathbb{T}_{O}\times^{G} \widetilde{X}_{1}$ given by translation on the factor $\mathbb{T}_{O}$. Then this action 
is obtained via an inclusion $G\subset \mathbb{T}$ \cite[Lemma 4.2]{AL21}. From the inclusions $G\subset \{1\}\times\mathbb{T}\subset \mathbb{G}_{m}\times \mathbb{T}$, the group $G$ also acts on $X_{\sigma_{i}}$ and $\phi_{i}$ is $G$-equivariant. To sum-up we have a Cartesian square
$$\xymatrix{U_{i}\ar[r]^{\phi_{i}} \ar[d] ^{g_{i}} & X_{\sigma_{i}}\ar[d]^{f_{i}}\\ U_{i}/G \ar[r]^{\varphi_{i}} & X_{\sigma_{i}}/G,}$$ 
for $i = 1,2, \ldots, s$, where the vertical arrows are quotient maps. It follows from the description in Remark \ref{rem-keyetale-struct-tor} of $\phi_{i}$ that 
$\varphi_{i}$ is \'etale.
Using Lemma \ref{lemma-smooth}  and Proposition \ref{theo-finitegrouptoric} we deduce that
$$g_{i}^{\star} IC_{U_{i}/G}\simeq g_{i}^{\star} \varphi_{i}^{\star} IC_{X_{\sigma_{i}}/G}\simeq (\varphi_{i}\circ g_{i})^{\star}IC_{X_{\sigma_{i}}/G}\simeq (f_{i}\circ\phi_{i})^{\star} IC_{X_{\sigma_{i}}/G}$$ $$\simeq \phi_{i}^{\star} f_{i}^{\star} IC_{X_{\sigma_{i}/G}}\simeq \phi_{i}^{\star} IC_{X_{\sigma_{i}}}\simeq IC_{U_{i}}.$$
So $(f^{\star}IC_{\widetilde{X}_{O}/G})_{|U_{i}}\simeq (IC_{\widetilde{X}_{O}})_{|U_{i}}$ for any $i\in \{1,\ldots, s\}$,
where $f: \widetilde{X}_{O}\rightarrow \widetilde{X}_{O}/G$ is the quotient map. As $f^{\star}IC_{\widetilde{X}_{O}/G}$
is semi-simple (see Proposition \ref{prop-IC-G-X}), we have $f^{\star}IC_{\widetilde{X}_{O}/G}
\simeq IC_{\widetilde{X}_{O}}$. Finally consider the commutative diagram
$$
\xymatrix{
\widetilde{X}_{O} \ar[r]^{f} \ar[dr]_{\varepsilon} & \widetilde{X}_{O}/G \ar[d]^{p} \\
 & \widetilde{X}_{1}/G
}
$$
given by the Seifert torus bundle $\varepsilon$,
where $p: \widetilde{X}_{O}/G\simeq \mathbb{T}_{O}\times \widetilde{X}_{1}/G\rightarrow \widetilde{X}_{1}/G$ is the projection
on the second factor. Since $p$ is smooth, by Lemma \ref{lemma-smooth} we have
$$\varepsilon^{\star}IC_{\widetilde{X}_{1}/G}[r]\simeq (p\circ f)^{\star}  IC_{\widetilde{X}_{1}/G}[r]\simeq f^{\star} p^{\star} IC_{\widetilde{X}_{1}/G}[r]\simeq f^{\star} IC_{\widetilde{X}_{O}/G}\simeq  IC_{\widetilde{X}_{O}},$$
proving the lemma.
\end{proof}
The following result  improves \cite[Theorem 1.1]{AL21}.
\begin{proposition}\label{theo-decomp1}
Let $X$ be a normal $\mathbb{T}$-variety of complexity one
and let $\pi:\widetilde{X}\rightarrow X$ be the contraction map. Then we have
$$\pi_{\star} IC_{\widetilde{X}}\simeq IC_{X} \oplus \bigoplus_{O\in {\rm Orb}(E)}\bigoplus_{b\in \mathbb{Z}} IC_{\bar{O}}^{\oplus s_{b, O}}[-b],$$
where ${\rm Orb}(E)$ is the set of orbits of $E$ and the $s_{b, O}$ are nonnegative integers.
\end{proposition}
\begin{proof} Following the proof of  \cite[Theorem 1.1 (ii)]{AL21}, we need to have
$\mathcal{H}^{j}(\pi_{\star}IC_{\widetilde{X}})_{|O}$ constant for any orbit $O\subset  E$
and any $j\in \mathbb{Z}$. Using Lemma \ref{l-Seifert-pullback} and base change \cite[Theorem 2.3.26]{Dim04} from the diagram of Lemma \ref{l-diagram},
$$\mathcal{H}^{j}(\pi_{\star} IC_{\widetilde{X}})_{|O}\simeq \mathcal{H}^{j}(\pi_{\star} \varepsilon^{\star}IC_{\widetilde{X}_{1}/G}[r])_{|O}\simeq  \mathcal{H}^{j}(\varepsilon^{\star}_{1}(\pi_{1})_{\star} IC_{\widetilde{X}_{1}/G}[r])_{|O}\simeq \mathcal{H}^{j}((\pi_{1})_{\star} IC_{\widetilde{X}_{1}/G})_{x_{0}}\otimes \mathbb{Q}_{O},$$
where $r  =  \dim\, O$ and $x_{0}$ is the unique fixed point of $X_{1}/G$. This shows the proposition.
\end{proof}

\section{Linear torus actions of complexity one}\label{sec-four}
This section develops weight package theory in order to describe linear torus actions of complexity one on possibly non-normal varieties.
\subsection{Linear torus actions: the affine case}
Consider a torus $\mathbb{G} =  (\mathbb{C}^{\star})^{\ell}$ with character and 
one-parameter subgroup lattices $\overline{M}$ and $\overline{N}$, and endow
$\mathbb{P}_{\mathbb{C}}^{\ell}$ with the toric $\mathbb{G}$-action
$$(\lambda_{1}, \ldots, \lambda_{\ell})\cdot [x_{0}: \ldots: x_{\ell}] = [x_{0}:\lambda_{1} x_{1}: \ldots :\lambda_{\ell}x_{\ell}],\text{ where }(\lambda_{1}, \ldots, \lambda_{\ell})\in \mathbb{G}.$$ Let $\mathbb{T}\subset \mathbb{G}$ be a subtorus.
\begin{definition}\label{def-suvbvarinv}
A subvariety of $\mathbb{P}_{\mathbb{C}}^{\ell}$ with \emph{linear $\mathbb{T}$-action} is an irreducible Zariski closed subset of $\mathbb{P}_{\mathbb{C}}^{\ell}$ intersecting the open $\mathbb{G}$-orbit and stable by $\mathbb{T}$-action.
A subvariety of $\mathbb{A}_{\mathbb{C}}^{\ell}$ with linear $\mathbb{T}$-action is an  irreducible Zariski closed subset $X\subset \mathbb{A}_{\mathbb{C}}^{\ell}$ such that there is subvariety $Z$ of  $\mathbb{P}_{\mathbb{C}}^{\ell}$ with linear $\mathbb{T}$-torus action with $X  =  Z\setminus H$, where $H$ is the coordinate hyperplane $\mathbb{V}(x_{0})\subset \mathbb{P}_{\mathbb{C}}^{\ell}$. 
\end{definition}
The following characterizes linear torus actions when the singularities are mild. This is a direct consequence of \cite[Proposition 2.4]{KKLV89}.
\begin{proposition} 
Any projective normal variety with faithful torus action can be equivariantly embedded as a subvariety of a projective space with linear torus action.
\end{proposition}
Let $X$ be a subvariety of $\mathbb{A}^{\ell}_{\mathbb{C}}$ with linear $\mathbb{T}$-action and set $n =  {\rm dim}\, \mathbb{T}$.
Let $E =  (e_{1}, \ldots, e_{\ell})$ be the canonical basis of $\overline{N} =  \mathbb{Z}^{\ell}$ and fix a basis $B$ of $N$. 
\begin{definition}\label{def-weightpack}
The \emph{weight matrix} of  $X$ is the matrix
${\rm Mat}_{E, B}(F)\in {\rm Mat}_{\ell\times n}(\mathbb{C})$
for the linear map $F: N\rightarrow \overline{N}$ corresponding to the inclusion $\mathbb{T}\hookrightarrow \mathbb{G}$. The \emph{matrix factorization} is the short exact sequence 
$$0\rightarrow N\xrightarrow{F} \overline{N}\xrightarrow{P} {\rm Coker}(F) =  \overline{N}/F(N)\rightarrow 0 $$
together with a splitting $S:\overline{N}\rightarrow N$, that is, $S$ is linear and $S\circ F = {\rm id}_{N}$ . Note that $F(N)$ must be a satured, i.e.
$$\overline{N}\cap F(N)_{\mathbb{Q}}\subset \overline{N}_{\mathbb{Q}}$$ 
is equal to $F(N)$.
Moreover, the map
\begin{equation}\label{Eq:downgraded}
    \phi:\overline{N}\rightarrow N\oplus {\rm Coker}(F),\,\, v\mapsto (S(v), P(v))
\end{equation}
is a $\mathbb{Z}$-module isomorphism. The \emph{quotient fan} $\Sigma$ of $X$ is the coarsest fan of ${\rm Coker}(F)_{\mathbb{Q}}$ containing all the strictly convex polyhedral cones $P(\delta_0)$, where $\delta_{0}$ runs over the face of the 
first quadrant 
$$ \delta:= \left\{ \sum_{i = 1}^{\ell}\alpha_{i} e_{i}\, \, |\, \alpha_{i}\in \mathbb{Q}_{\geq 0}, \,\, 1\leq i\leq \ell\right\}\subset \overline{N}_{\mathbb{Q}}.$$
The \emph{weight package} of $X$ is the data $\theta =  (\overline{N}, N, F, S, \Sigma)$.
\end{definition}

The following result is known \cite[Section 4]{AIPSV12}. For the convenience of the reader, we give a proof.
\begin{lemma}\label{lem-poltheta}
Let $X\subset \mathbb{A}^{\ell}_{\mathbb{C}}$ be a subvariety with linear $\mathbb{T}$-action  and  let $\theta =  (\overline{N}, N, F, S, \Sigma)$ be its weight package. Consider the strictly convex polyhedral cone $\sigma_{\theta}:= S(\delta\cap F(N_{\mathbb{Q}}))$ and the $\sigma_{\theta}$-polyhedral divisor 
$$\mathfrak{D}_{\theta}:= \sum_{Z\subset  X_{\Sigma}} \mathfrak{D}_{\theta, Z}\cdot [Z]$$
defined as follows. If $Z_{\rho} =  \overline{O(\rho)}$ is the divisor corresponding to $\rho\in \Sigma(1)$, then set
$$\mathfrak{D}_{\theta, Z_{\rho}} =  S(\delta\cap P^{-1}(v_{\rho})).$$
 Otherwise, set 
$\mathfrak{D}_{\theta, Z} =  \sigma_{\theta}$. Then  $\mathfrak{D}_{\theta}$ is proper and  $X(X_{\Sigma}, \mathfrak{D}_{\theta})$ is $\mathbb{T}$-isomorphic to $\mathbb{A}_{\mathbb{C}}^{\ell}$.
\end{lemma}
\begin{proof}
 Identify $\overline{N}$ with $ N\oplus {\rm Coker}(F)$.
Denote by ${\bf x}^{w}$ the monomial associated with $w\in {\rm Coker}(F)^{\vee}:= {\rm Hom}_{\mathbb{Z}}({\rm Coker}(F), \mathbb{Z})$ and  let $m\in M$. Observe that  ${\bf x}^{w}\chi^{m}\in A[X_{\Sigma}, \mathfrak{D}_{\theta}]$ if and only if 
$$\langle w, v_{\rho} \rangle_{{\rm Coker}(F)}  + \langle m, S(v)\rangle_{N}\geq 0\text{ and }\langle m, v_{0}\rangle_{N}\geq 0$$
for all  $\rho\in\Sigma(1)$, $v\in\delta$ such that $P(v)  =  v_{\rho}$ and $v_{0}\in\delta\cap N$. This is equivalent to  $\langle (w, m), e \rangle_{\overline{N}}\geq 0$ for any $e\in \delta$ as $\phi$ in Equation (\ref{Eq:downgraded}) is an isomorphism. Thus, it follows that $X(X_{\Sigma}, \mathfrak{D}_{\theta})$ is $\mathbb{T}$-isomorphic to the downgraded toric variety $X_{\delta}$ associated to $\delta$. Moreover, $X_{\delta} \simeq \mathbb{A}_{\mathbb{C}}^{\ell}$.

Let us prove that $\mathfrak{D}_{\theta}$ is proper. Denote by $(m_{i}, w_{i})$ ($1\leq i\leq \ell$) the lattice generators of the rays of the dual cone $\delta^{\vee} \subset  \overline{M}_{\mathbb{Q}}$. In particular,
the elements ${\bf x}^{w_{i}}\chi^{m_{i}}$ generate the algebra $A[X_{\Sigma}, \mathfrak{D}_{\theta}]$ and
$$\mathfrak{D}_{\theta, Z_{\rho}} =  \{ v\in N_{\mathbb{Q}}\,\,|\,\, \langle m_{i}, v\rangle_{N} +  \langle w_{i}, v_{\rho}\rangle_{{\rm Coker}(F)}\geq 0 \text{ for } i= 1, \ldots, \ell\} .$$
Fix a primitive lattice vector $m\in\sigma^{\vee}_{\theta}\cap M \setminus \{0\}$, 
let $L:= \mathbb{Q}_{\geq 0}\cdot m$ and denote by $\mathscr{H}_{L}$ the Hilbert basis 
in $\mathbb{Z}^{\ell}$ of the polyhedral cone 
$$\gamma^{-1}(L)\cap \mathbb{Q}_{\geq 0}^{\ell}, \text{ where } \gamma: \mathbb{Q}^{\ell}\rightarrow M_{\mathbb{Q}}$$
is the linear map sending the canonical basis to $(m_{1}, \ldots, m_{\ell})$. Set
$$\mathscr{H}_{L}^{\star}:= \left\{(s_{1}, \ldots, s_{\ell})\in \mathscr{H}_{L}\,\,|\,\, \sum_{i = 1}^{\ell} s_{i} m_{i}\neq 0\right\}.$$
Note that for any $s = (s_{1}, \ldots, s_{\ell})\in \mathscr{H}_{L}^{\star}$ there is a unique $\lambda(s)\in \mathbb{Z}_{>0}$ such that
$\sum_{i = 1}^{\ell} s_{i} m_{i}  =  \lambda(s)\cdot m.$
By \cite[Lemma 2.11]{Lan15},
$$\min_{v\in \mathfrak{D}_{\theta, Z_{\rho}}}\langle m, v\rangle  =  -\min_{s = (s_{1}, \ldots, s_{\ell})\in \mathscr{H}_{L}^{\star}}\frac{\sum_{i = 1}^{\ell} s_{i} \langle w_{i}, v_{\rho}\rangle}{\lambda(s)}, \text{ that is } \mathfrak{D}_{\theta}(m) =  -\min_{s = (s_{1}, \ldots, s_{\ell})\in \mathscr{H}_{L}^{\star}}\frac{{\rm div}\, f_{s}}{\lambda(s)},$$
where 
$ f_{s}:= \prod_{i = 1}^{\ell} ({\bf x}^{w_{i}})^{s_{i}}.$
Let $\Sigma_{0}$ be the subfan of $\Sigma$ generated by the rays of $\Sigma$.  Note that the complement of $X_{\Sigma_{0}} \subset X_{\Sigma}$ is of codimension $2$. By the previous equality, $(\mathfrak{D}_{\theta}(m))_{|X_{\Sigma_{0}}}$
is semi-ample. Let $A  =  \bigoplus_{r\in\mathbb{Z}_{\geq 0}}A_{r}$ be the graded algebra, where  $$A_{r} = H^{0}(X_{\Sigma}, \mathcal{O}_{X_{\Sigma}}(\mathfrak{D}_{\theta}(rm))),$$ and let $V:= {\rm Proj}(A)$. Since $X_{\Sigma}$ is the Chow quotient \cite[Section 4]{KSZ91}, there is a surjective projective morphism $p: X_{\Sigma}\rightarrow V$. Let $d\in\mathbb{Z}_{>0}$
such that the $d$-th Veronese subalgebra 
$$A_{(d)} = \bigoplus_{d\geq 0} A_{dr}$$ of 
$A$ is generated
by its degree-one elements. Let $E$ be the $\mathbb{Q}$-divisor 
on $X_{\Sigma}$ such that $dE$ is the pullback by $p$ of the $\mathcal{O}(1)$ of $A_{(d)}$. In this way, 
$$\Gamma(X_{\Sigma_{0}}, \mathcal{O}_{X_{\Sigma}}(rE)) = \Gamma(X_{\Sigma_{0}}, \mathcal{O}_{X_{\Sigma}}(\mathfrak{D}_{\theta}(rm)))\text{ for any  }r\in\mathbb{Z}_{\geq 0}.$$ By Lemma \ref{lem1-Qdivisor}, $\mathfrak{D}_{\theta}(m) = E$ is semi-ample. 

Assume that $m$ is in the relative interior of $\sigma_{\theta}^{\vee}$. Since the Chow quotient is the normalization 
of the canonical component of the projective limit of  GIT quotients of the form $V$, the map $p$ is birational. Hence choosing $f\in A_{r}\setminus\{0\}$ for some $r\in \mathbb{Z}_{>0}$ such that 
$p^{-1}(D_{+}(f))\rightarrow D_{+}(f), x\mapsto p(x) $
is an isomorphism,  we have $$(X_{\Sigma})_{f, \mathfrak{D}_{\theta}(rm)} =  p^{-1}(D_{+}(f))$$
affine. Thus $\mathfrak{D}_{\theta}(m)$ is big, proving the lemma. 
\end{proof}
\begin{definition}
The polyhedral divisor $\mathfrak{D}_{\theta}$ is the \emph{associated polyhedral divisor} of $\theta$. 
\end{definition}
From now on assume that the $\mathbb{T}$-action on $X\subset \mathbb{A}^{\ell}_{\mathbb{C}}$ is of complexity one.
We define the curve $\widehat{C_{\theta}}$
as follows.
\begin{itemize}
\item[$(1)$] If $\mathbb{C}[X]^{\mathbb{T}}\neq \mathbb{C}$, then $\widehat{C_{\theta}}$ is the normalization of ${\rm Spec}\, \mathbb{C}[X]^{\mathbb{T}}$;
\item[$(2)$] Otherwise $\widehat{C_{\theta}}$ is the smooth projective curve associated with the one-variable function field $\mathbb{C}(X)^{\mathbb{T}}/\mathbb{C}$. 
\end{itemize}
The next result relates the curve $\widehat{C_{\theta}}$ with the closed  embedding $X\hookrightarrow \mathbb{A}^{\ell}_{\mathbb{C}}$. 
\begin{lemma}\label{lem-Ctheta-Ctheta}
With the notation of Definition \ref{def-weightpack}, the rational quotient $\psi: \mathbb{A}_{\mathbb{C}}^{\ell}\dashrightarrow X_{\Sigma}$
for the $\mathbb{T}$-action is regular on the open $\mathbb{G}$-orbit. Furthermore, $\widehat{C_{\theta}}$ is the normalization of the Zariski closure $C_{\theta}$ of $\psi(X\cap \mathbb{G})$ in $X_{\Sigma}$. 
\end{lemma}
\begin{proof}
First claim is clear since $\psi_{|\mathbb{G}}: \mathbb{G}\rightarrow \mathbb{G}/\mathbb{T}$ is the quotient. The restriction
$\psi_{|X\cap \mathbb{G}}: X\cap \mathbb{G}\rightarrow \psi(X\cap \mathbb{G})$ is thus a geometric quotient, and the rational map $\psi_{|X}: X\dashrightarrow C_{\theta}$  induces an isomorphism
$\mathbb{C}(X)^{\mathbb{T}}\simeq \mathbb{C}(C_{\theta})$. Let 
$$q: \mathbb{A}_{\mathbb{C}}^{\ell}\rightarrow Y_{0}:= \mathbb{A}_{\mathbb{C}}^{\ell}/\!\!/\mathbb{T}$$
be the morphism induced by the inclusion $\mathbb{C}[\mathbb{A}^{\ell}_{\mathbb{C}}]^{\mathbb{T}}\subset \mathbb{C}[\mathbb{A}^{\ell}_{\mathbb{C}}]$. Since $X_{\Sigma}$ is the Chow quotient \cite{KSZ91}, there is a projective surjective morphism $r: X_{\Sigma}\rightarrow Y_{0}$ such that $r\circ \psi =  q$. This induces a commutative diagram 
$$
\xymatrix{
X \ar@{-->}[r]^{\psi_{|X}} \ar[dr]_{q_{|X}} & C_{\theta}\ar[d]^{r_{|C_{\theta}}} \\
 & X/\!\!/\mathbb{T}.
}
$$
Consequently, $r_{|C_{\theta}}: C_{\theta}\rightarrow X/\!\!/\mathbb{T}$ is dominant projective. If $A_{0}:= \mathbb{C}[X]^{\mathbb{T}}\neq \mathbb{C}$, then $r_{|C_{\theta}}$ is finite birational. So $\widehat{C_{\theta}}$ is the normalization of $C_{\theta}$. 
\end{proof}
\begin{definition}\label{def-pullbackpol}
The \emph{pullback polyhedral divisor of $(X, \theta)$} is 
$\bar{\mathfrak{D}}_{\theta}:= \kappa^{\star}(\mathfrak{D}_{\theta})$, where $\kappa: \widehat{C_{\theta}}\rightarrow X_{\Sigma}$ is the composition of the inclusion $C_{\theta}\subset X_{\Sigma}$ (see Lemma \ref{lem-Ctheta-Ctheta}) and the normalization. 
\end{definition}
\begin{lemma}\label{lem-tectec}
Let $\mathfrak{D}$ be a proper $\sigma$-polyhedral divisor over a semi-projective normal variety $Y$. Let $\varphi: S\rightarrow Y$ be a morphism, where $S$ is a smooth curve and $\varphi$ is the composition of a finite 
birational morphism $S\rightarrow S'$ and a closed immersion $S'\rightarrow Y$. Assume that the image of $\varphi$ is not
contained in the union  of the prime divisors of 
$\{Z\subset Y\,|\, \mathfrak{D}_{Z}\neq \sigma\}.$
Let $(g_{1}\chi^{m_{1}}, \ldots, g_{r}\chi^{m_{r}})$ be a system of generators of 
$A[Y, \mathfrak{D}]$ with $g_{i}\in\mathbb{C}(Y)^{\star}$
and such that $f_{i}  =  g_{i}\circ \varphi \in\mathbb{C}(S)^{\star}$  for $1\leq i\leq r$. 
Then the pullback  $\varphi^{\star}(\mathfrak{D})$ is given by 
$$ \varphi^{\star}(\mathfrak{D})_{z} =  \{v\in N_{\mathbb{Q}}\,|\, \langle m_{i}, v\rangle \geq -{\rm ord}_{z}(f_{i})\text{ for }i=1, \ldots, r\} \text{ for any } z\in S.$$ 
\end{lemma}
\begin{proof}
Let $B$ be the integral closure of $\mathbb{C}[S][f_{1}\chi^{m_{1}}, \ldots, f_{r}\chi^{m_{r}}]$ in
$\mathbb{C}(S)[M]$. By Theorem \ref{theo-norma},  $B =  A[S, {}^{\xi}\mathfrak{D}]$, where $\xi =  (f_{1}\chi^{m_{1}}, \ldots, f_{r}\chi^{m_{r}})$. Fix 
a primitive vector $m\in M$ in  the dual of ${}^{\xi}\sigma =  \sigma$ and set $L  =  \mathbb{Q}_{\geq 0}m$. Denote by $\mathscr{H}_{L}$ the Hilbert basis in $\mathbb{Z}^{r}$ of 
$\gamma^{-1}(L)\cap \mathbb{Q}_{\geq 0}^{r}$, where $\gamma: \mathbb{Q}^{r}\rightarrow M_{\mathbb{Q}}$ is the linear map sending the canonical basis to $(m_{1}, \ldots, m_{r})$. Consider 
$\mathscr{H}_{L}^{\star}$ and $\lambda(s)\in\mathbb{Z}_{>0}$ for $s\in \mathscr{H}_{L}^{\star}$  as in the proof of Lemma \ref{lem-poltheta}. 
Using  \cite[Lemma 2.11]{Lan15} we have
$${}^{\xi}\mathfrak{D}(m) =  -\min_{s = (s_{1}, \ldots, s_{r})\in \mathscr{H}_{L}^{\star}}\frac{{\rm div}(f_{s})}{\lambda(s)},
\text{ where } f_{s}:= \prod_{i = 1}^{r}f_{i}^{s_{i}}.$$
As $A[Y, \mathfrak{D}]  =  \mathbb{C}[Y][g_{1}\chi^{m_{1}}, \ldots, g_{r}\chi^{m_{r}}]$,  Lemma \ref{lem0-gen} implies  
${}^{\eta}\mathfrak{D} =  \mathfrak{D}$, where $\eta  =  (g_{1}\chi^{m_{1}}, \ldots, g_{r}\chi^{m_{r}})$, and 
$$\mathfrak{D}(m) =  {}^{\eta}\mathfrak{D}(m) =  -\min_{s = (s_{1}, \ldots, s_{r})\in \mathscr{H}_{L}^{\star}}\frac{{\rm div}(g_{s})}{\lambda(s)}\text{ for }
g_{s}:= \prod_{i = 1}^{r}g_{i}^{s_{i}}.$$
 For any prime divisor $Z\subset Y$, let $U_{Z}\subset Y$
 be an open subset such that ${}^{\eta}\mathfrak{D}(m)_{|U_{Z}} =  a_{Z}\cdot [Z\cap U_{Z}]$
 for some $a_{Z}\in\mathbb{Q}$, $U_{Z}\cap Z \neq \emptyset$. Moreover, we ask these $U_{Z}$ cover $Y$. Note that for any $Z$,
 there is $s(Z)\in \mathscr{H}_{L}^{\star}$ such that
 $$\mathfrak{D}(m)_{|U_{Z}} =  -\left(\frac{{\rm div}(g_{s(Z)})}{\lambda(s(Z))}\right)_{|U_{Z}}.$$
So
 $$\varphi^{\star}(\mathfrak{D}(m)_{|U_{Z}}) =  -\left(\frac{{\rm div}(\varphi^{\star}g_{s(Z)})}{\lambda(s(Z))}\right)_{|\varphi^{-1}(U_{Z})} =  -\left(\frac{{\rm div}(f_{s(Z)})}{\lambda(s(Z))}\right)_{|\varphi^{-1}(U_{Z})} \leq  {}^{\xi}\mathfrak{D}(m)_{|\varphi^{-1}(U_{Z})} $$
for any $Z$,
 that is $  \varphi^{\star}(\mathfrak{D})(w) \leq {}^{\xi}\mathfrak{D}(w)$ for any $w\in \sigma^{\vee}$.
Finally, since ${\rm div}(f_i) + \varphi^{\star}(\mathfrak{D})(m_{i})\geq 0$ for $i = 1, \ldots, r$ we have $\varphi^{\star}(\mathfrak{D})_{z}\subset {}^{\xi}\mathfrak{D}_{z}$ for any $z\in S$, and therefore $\varphi^{\star}(\mathfrak{D}) =   {}^{\xi}\mathfrak{D}$ as required. 
\end{proof}

The following describes normalization of affine subvarieties with linear torus action of complexity one (see \cite[Section 4.2]{AIPSV12} for the normal case).
\begin{theorem}\label{theo-aff-impl}
Let $X\subset \mathbb{A}^{\ell}_{\mathbb{C}}$ be a subvariety with linear torus action of complexity one
and with weight package $\theta =  (\overline{N}, N, F, S, \Sigma)$. Then 
the normalization of $X$ is equivariantly isomorphic to $X(\widehat{C_{\theta}}, \bar{\mathfrak{D}}_{\theta})$, where $\bar{\mathfrak{D}}_{\theta}$ is the pullback polyhedral divisor of $\theta$.
\end{theorem}
\begin{proof}
Set $\mathbb{T} =  \mathbb{T}_{N}$  and
let $x_{1}, \ldots, x_{\ell}$ be the coordinate functions  of $\mathbb{A}^{\ell}_{\mathbb{C}}$ in which $\mathbb{G}$ diagonally acts.
From the decomposition $\overline{M}  =  M\oplus {\rm Coker}(F)^{\vee}$ induced by the weight package $\theta$, we
 have
$x_{i} =  g_{i} \chi^{m_{i}}$, where $m_{i}\in M$ and $g_{i}\in \mathbb{C}(X_{\Sigma})^{\star}$.
Write $x_{i|X} =  f_{i}\chi^{m_{i}}$.  Since $X$ intersects the open $\mathbb{G}$-orbit, each $f_{i} =  \kappa^{\star}(g_{i}) \in \mathbb{C}(X)^{\mathbb{T}}  = \mathbb{C}(\widehat{C_{\theta}})$ is nonzero.
 The coordinate ring of the normalization $X'$ of $X$ is the integral closure of 
$\mathbb{C}[\widehat{C_{\theta}}][f_{1}\chi^{m_{1}}, \ldots, f_{\ell} \chi^{m_{\ell}}]$
in $\mathbb{C}(\widehat{C_{\theta}})[M]$. By Theorem \ref{theo-norma}, we have $\mathbb{C}[X'] =  A[\widehat{C_{\theta}}, {}^{\xi}\mathfrak{D}]$, where $\xi =  (f_{1}\chi^{m_{1}}, \ldots, f_{\ell} \chi^{m_{\ell}})$ and $\mathbb{C}[X']$ is seen as a subring of $\mathbb{C}(\widehat{C_{\theta}})[M]$. Using Lemma \ref{lem-tectec} we conclude that $\bar{\mathfrak{D}}_{\theta} = \kappa^{\star}(\mathfrak{D}_{\theta}) =  {}^{\xi}\mathfrak{D}$, as required. 
\end{proof}
\begin{example}
Let 
$P(z) =  \prod_{i =1}^{s} (z-z_{i})^{m_{i}}\in \mathbb{C}[z]$
be a polynomial, where the $z_{i}\in\mathbb{C}$ are distinct and $m_{i}\in \mathbb{Z}_{>0}$. Consider the $\mathbb{C}^{\star}$-surface 
$$X =  \mathbb{V}(x^{d}- yP(z))\subset \mathbb{A}_{\mathbb{C}}^{3}$$
with action 
$\lambda\cdot (x, y, z)  =  (\lambda x , \lambda^{d} y, z)$ for $\lambda\in \mathbb{C}^{\star}.$ The weight package $\theta$ of $X$ 
is described as follows. 
The  matrix factorization is 
$$P=
\begin{pmatrix}
-d & 1 & 0\\
0 & 0 & 1
\end{pmatrix},$$ for $S$ take $\begin{pmatrix}
1 & 0 & 0
\end{pmatrix}$ and $X_{\Sigma} = \mathbb{P}_{\mathbb{C}}^{1}\times \mathbb{A}_{\mathbb{C}}^{1}$. Moreover, denoting
by $Z_{(-1, 0)}$ the divisor of $X_{\Sigma}$ corresponding to the ray $\mathbb{Q}_{\geq 0}(-1,0)$, we have
$\mathfrak{D}_{\theta} =  (\frac{1}{d} + \mathbb{Q}_{\geq 0})\cdot Z_{(-1, 0)}$. Also
$$\widehat{C_{\theta}}  = \{([x_{0}: x_{1}], t)\in \mathbb{P}_{\mathbb{C}}^{1}\times \mathbb{A}_{\mathbb{C}}^{1}\, |\, x_{1} P(t) -x_{0} =  0\}\simeq \mathbb{A}^{1}_{\mathbb{C}}.$$
Since
$\kappa^{\star}([Z_{(-1, 0)}]) =  \sum_{i = 1}^{s}m_{i}[z_{i}]$, we deduce that
$\bar{\mathfrak{D}}_{\theta}  =  \sum_{i = 1}^{s} \left(\frac{m_{i}}{d} +\mathbb{Q}_{\geq 0}\right)\cdot [z_{i}]$ over $\mathbb{A}^{1}_{\mathbb{C}}$
describes the normalization of $X$, recovering \cite[Example 3.10]{FZ03}.
\end{example}
Finally, we study the pullback divisor $\bar{\mathfrak{D}}_{\theta}$ under subdivision of the fan $\Sigma$. 
\begin{proposition}\label{prop-sub-kappa}
Let $\Sigma'$ be a fan subdivision of $\Sigma$ such that $X_{\Sigma'}$ is semi-projective and let $f: X_{\Sigma'}\rightarrow X_{\Sigma}$ be 
the corresponding toric modification. Let $C_{\theta}'\subset X_{\Sigma'}$ be the proper transform of the curve $C_{\theta}\subset X_{\Sigma}$ under $f$. Let $\eta: \widehat{C_{\theta}'}\rightarrow X_{\Sigma'}$ be the composition of the normalization $\eta_{0}:\widehat{C_{\theta}'}\rightarrow C_{\theta}'$ and
the closed immersion $C_{\theta}'\rightarrow X_{\Sigma'}$. Then $X(\widehat{C_{\theta}'}, \eta^{\star}f^{\star} \mathfrak{D}_{\theta})$ is equivariantly isomorphic to $X(\widehat{C_{\theta}}, \bar{\mathfrak{D}}_{\theta})$.
\end{proposition}
\begin{proof}
Let $\bar{f}: C_{\theta}'\rightarrow C_{\theta}$
be the morphism induced by $f$. Then we have commutative 
square
 $$\xymatrix{ \widehat{C_{\theta}'}  \ar[r]^{\eta_{0}} \ar[d] ^{\gamma} & C_{\theta}' \ar[d]^{\bar{f}}\\   \widehat{C}_{\theta} \ar[r]^{\nu} & C_{\theta}},$$ 
 where $\nu$ is the normalization. Here $\gamma$ is obtained from the composition
 of the rational maps $\eta_{0}$, $\bar{f}$ and $\nu^{-1}$. As $\widehat{C_{\theta}'}$ and $\widehat{C_{\theta}}$ are smooth,
 $\gamma$ is an open immersion. Moreover, $\gamma$ is bijective from the diagram. So $\gamma$ is an isomorphism. 
 To sum-up we have a commutative square
 $$\xymatrix{ \widehat{C_{\theta}'}  \ar[r]^{\eta} \ar[d] ^{{\gamma}} & X_{\Sigma'} \ar[d]^{f}\\   \widehat{C}_{\theta} \ar[r]^{\kappa} & X_{\Sigma}}$$ 
 and therefore $\gamma^{\star}\bar{\mathfrak{D}}_{\theta} =  \gamma^{\star}\kappa^{\star}\mathfrak{D}_{\theta} =  \eta^{\star} f^{\star} \mathfrak{D}_{\theta}$. This proves the proposition. 
\end{proof}

\subsection{Curves on toric surfaces}\label{SEC-CurveToricsURF}
The next two Subsections \ref{SEC-CurveToricsURF} and \ref{sec-exex} study examples illustrating the
theory. Computing the pullback of polyhedral divisors of weight packages $\theta =  (\overline{N}, N, F, S, \Sigma)$
of hypersurfaces with complexity-one torus action involves considering the following 
problem. 
\\

Assume that $X_{\Sigma}$ is a surface, take
$C :=  \{(x, y)\in \mathbb{T}_{{\rm Coker}(F)} =  (\mathbb{C}^{\star})^{2} \,|\, f(x, y) =  0\}$
for some $f\in \mathbb{C}[x, x^{-1}, y, y^{-1}]$ irreducible, 
let $C_{\theta}$ be the  Zariski closure of $C$ in $X_{\Sigma}$ and  let $\kappa$ be the composition of the normalization $ \widehat{C_{\theta}}\rightarrow C_{\theta}$ and the inclusion
$ C_{\theta}\subset X_{\Sigma}$. \\

\noindent{\bf Question:} \emph{Given a
toric divisor $Z_{\rho}\subset X_{\Sigma}$, how do we compute $\kappa^{\star}([Z_{\rho}])$ in terms of $\Sigma$ and $f$?}
\\

Composing with toric modifications, we assume $X_{\Sigma}$ smooth.

\begin{example}\label{refexamp}
Consider the  fan $\Sigma$ of $\mathbb{Q}^{2}$ with maximal cones

\begin{center}
 \begin{tabular}{|c| c |} 
 \hline
  & Generators \\ [1ex] 
 \hline\hline
 $\sigma_{1} $ & $(1, 0), (1,1)$ \\ [2ex] 
 \hline
 $\sigma_{2} $ & $(1, 1), (0,1)$ \\  [2ex] 
 \hline
 $\sigma_{3} $ & $(0, 1),(-1,-2)$ \\ [2ex] 
 \hline
 $\sigma_{4} $ & $(-1, -2) , (0,-1)$ \\ [2ex] 
 \hline
 $\sigma_{5} $ & $(0, -1), (1,0)$ \\  [2ex] 
 \hline
\end{tabular}
\end{center}
\

and the curve
$C =  \{(x, y)\in (\mathbb{C}^{\star})^{2}\,|\, y + x + x^{2} =  0\}.$
\end{example}
We now provide a three-step calculation method for calculating the pullbacks $\kappa^{\star}([Z_{\rho}])$.
\\

\emph{Step 1.} 
Determine  Laurent monomials $x_{\sigma}, y_{\sigma}$ such that $\mathbb{C}[\sigma^{\vee}\cap M] =  \mathbb{C}[x_{\sigma}, y_{\sigma}]$ for any $\sigma\in \Sigma_{\rm max}$. Then 
write $x  =  x_{\sigma}^{m(\sigma)_{1}} y_{\sigma}^{n(\sigma)_{1}}$ and  $y  =  x_{\sigma}^{m(\sigma)_{2}} y_{\sigma}^{n(\sigma)_{2}}$ with $n_{i}(\sigma), m_{j}(\sigma)\in \mathbb{Z}$ and substitute: $f(x, y) =  f (x_{\sigma}^{m(\sigma)_{1}} y_{\sigma}^{n(\sigma)_{1}},  x_{\sigma}^{m(\sigma)_{2}} y_{\sigma}^{n(\sigma)_{2}})$. Multiply by a Laurent monomial $x_{\sigma}^{s(\sigma)} y_{\sigma}^{t(\sigma)}$ in order to have 
$f_{\sigma}(x_{\sigma}, y_{\sigma}):= x_{\sigma}^{s(\sigma)} y_{\sigma}^{t(\sigma)}f(x,y)\in \mathbb{C}[\sigma^{\vee}\cap M]$
irreducible.
In conclusion $C_{\theta}\cap X_{\sigma} =  {\rm Spec}\, \mathbb{C}[x_{\sigma}, y_{\sigma}]/(f_{\sigma})$. 
\begin{example}\label{refexamp01}
Returning to Example \ref{refexamp}, the local equations are:
\begin{center}
 \begin{tabular}{|c| c |} 
 \hline
  & Equation $f_{\sigma_{i}}$ \\ [1ex] 
 \hline\hline
 $C_{\theta}\cap X_{\sigma_{1}}$ & $1  + x_{1}^{2}y_{1} + y_{1}$ \\ [2ex] 
 \hline
 $C_{\theta}\cap X_{\sigma_{2}}$ & $1 + x_{2} + y_{2}$ \\  [2ex] 
 \hline
 $C_{\theta}\cap X_{\sigma_{3}}$ & $1 + x_{3} + y_{3}$ \\ [2ex] 
 \hline
 $C_{\theta}\cap X_{\sigma_{4}}$ & $1 + x_{4}y_{4}  + y_{4}$ \\ [2ex] 
 \hline
 $C_{\theta}\cap X_{\sigma_{5}}$ & $1 + x_{5}y_{5}  + x_{5}^{2}y_{5}$ \\  [2ex] 
 \hline
\end{tabular}
\end{center}
\end{example}
\emph{Step 2.} We compute the trivializations of  $\mathcal{O}_{X_{\Sigma}}(Z_{\rho})$ for $\rho\in \Sigma(1)$. Let $\rho^{-}, \rho^{+}\in \Sigma(1)\setminus \{\rho\}$ distinct such that 
$\sigma_{\rho}^{-}: =  \rho^{-} + \rho$,  $\sigma_{\rho}^{+} :=  \rho^{+}+\rho\in\Sigma_{\rm max}$. Consider $m_{\rho}^{-}, m_{\rho}^{+}\in M$ satisfying
$\langle m_{\rho^{-}}, v_{\rho}\rangle  = 1, \langle m_{\rho^{-}}, v_{\rho^{-}}\rangle  =  0,\langle m_{\rho^{+}}, v_{\rho}\rangle  = 1, \langle m_{\rho^{+}}, v_{\rho^{+}}\rangle  =  0.$
Then the trivialization of $Z_{\rho}$ is given by
\begin{itemize}
\item $(X_{\sigma}, 1)$ for any $\sigma \in \Sigma_{\rm max}\setminus \{\sigma_{\rho}^{-}, \sigma_{\rho}^{+}\}$;
\item $(X_{\sigma_{\rho}^{-}}, \chi^{m_{\rho^{-}}})$ and $(X_{\sigma_{\rho}^{+}}, \chi^{m_{\rho^{+}}}).$ 
\end{itemize}
\begin{example}
Non-trivial trivializations for Example  \ref{refexamp}
are given by:
\begin{center}
 \begin{tabular}{|c| c |} 
 \hline
  & Trivialization of $Z_{\rho}$ \\ [1ex] 
 \hline\hline
 $Z_{(1, 0)}$ & $(X_{\sigma_{1}}, \chi^{(1, -1)}), (X_{\sigma_{5}}, \chi^{(1, 0)})$ \\ [2ex] 
 \hline
 $Z_{(1,1)}$ & $(X_{\sigma_{1}}, \chi^{(0, 1)}), (X_{\sigma_{2}}, \chi^{(1, 0)})$ \\  [2ex] 
 \hline
 $Z_{(0,1)}$ & $(X_{\sigma_{2}}, \chi^{(-1, 1)}), (X_{\sigma_{3}}, \chi^{(-2, 1)})$  \\ [2ex] 
 \hline
 $Z_{(-1,-2)}$ & $(X_{\sigma_{3}}, \chi^{(-1, 0)}), (X_{\sigma_{4}}, \chi^{(-1, 0)})$ \\ [2ex] 
 \hline
 $Z_{(0,-1)}$ & $(X_{\sigma_{4}}, \chi^{(2, -1)}), (X_{\sigma_{5}}, \chi^{(0, -1)})$  \\  [2ex] 
 \hline
\end{tabular}
\end{center}
\end{example}

\emph{Step 3.}
Let $\varphi: \widehat{C_{\theta}}\rightarrow C_{\theta}$ be the desingularization. With the notation of Step 2, set $\Omega^{\pm}_{\rho} : = 
\varphi^{-1}(X_{\sigma_{\rho}^{\pm}}\cap C_{\theta})$ and consider the sheaf of ideals $\mathscr{J}_{\rho}\subset \mathcal{O}_{C_{\theta}}$ generated by the restricted trivializations, i.e.
\begin{itemize}
\item[(i)] $\mathscr{J}_{\rho}(\Omega_{\rho}^{\pm}) =  \chi^{m_{\rho^{\pm}}}_{|C_{\theta}}\cdot \mathbb{C}[\Omega^{\pm}_{\rho}]$;
\item[(ii)] $\mathscr{J}_{\rho}(\varphi^{-1}(X_{\sigma}\cap C_{\theta})) =  \mathbb{C}[\varphi^{-1}(X_{\sigma}\cap C_{\theta})]$ for any 
$\sigma\in \Sigma_{\rm max}\setminus\{\sigma_{\rho}^{-}, \sigma_{\rho}^{+}\}.$
\end{itemize}
Denote by $\kappa_{\rho}^{\pm}$ the restriction $\kappa_{|X_{\sigma_{\rho}^{\pm}}}$. There is a unique decomposition 
$$\mathscr{J}_{\rho}(\Omega_{\rho}^{\pm}) =  \prod_{j = 1}^{s_{\rho}^{\pm}}\mathfrak{M}_{\rho^{\pm}, j}^{s_{\rho, j}^{\pm}},\text{ where }\mathfrak{M}_{\rho^{\pm}, j}\subset \mathbb{C}[\Omega_{\rho}^{\pm}]$$
are maximal ideals. 
We compute $\kappa^{\star}([Z_{\rho}])$ via 
$\kappa^{\star}_{\rho^{\pm}}([Z_{\rho}|_{X_{\sigma_{\rho}^{\pm}}}]) =  \sum_{j = 1}^{s_{\rho}^{\pm}}s_{\rho, j}^{\pm}\cdot [\zeta_{\rho^{\pm}, j}],$
where $\zeta_{\rho^{\pm}, j}\in \widehat{C_{\theta}}$ corresponds to $\mathfrak{M}_{\rho^{\pm}, j}$.
\begin{example}\label{refexamp03}
The computation for Example \ref{refexamp} of the pullbacks is given by:
\begin{center}
 \begin{tabular}{|c| c |} 
 \hline
  & $\kappa^{\star}$ \\ [1ex] 
 \hline\hline
 $Z_{(1, 0)}$ & $\kappa^{\star}([Z_{(1, 0)}]) =  0$ \\ [2ex] 
 \hline
 $Z_{(1, 1)}$ & $\kappa^{\star}([Z_{(1, 1)}]) =  [\zeta_{(1,1)}]$ \\  [2ex] 
 \hline
 $Z_{(0, 1)}$ & $\kappa^{\star}([Z_{(0, 1)}]) =  [\zeta_{(0,1)}]$ \\ [2ex] 
 \hline
 $Z_{(1, 0)}$ & $\kappa^{\star}([Z_{(-1, -2)}]) =  [\zeta_{(-1, -2)}]$ \\ [2ex] 
 \hline
 $Z_{(0, -1)}$ & $\kappa^{\star}([Z_{(0, -1)}]) =  0$ \\  [2ex] 
 \hline
\end{tabular}
\end{center}
\end{example}
\subsection{The $\mathbb{C}^{\star}$-surface $z_{1}z_{2}^{2} + z_{2}z_{0}^{2} + z_{3}^{3} = 0$}\label{sec-exex}

We now study an example of a projective variety with linear torus action of complexity one. 
\\

\begin{example} 
\emph{Consider the toric variety $\mathbb{P}^{3}_{\mathbb{C}}$.} We are in the case $\ell = 3$. 
Let $e_{1} =  (1, 0,0)$, $e_{2} =  (0, 1,0)$, $e_{3} =  (0, 0, 1)$ and $e =  -e_{1}-e_{2}-e_{3}$. 
The fan of $\mathbb{P}^{3}_{\mathbb{C}}$ is generated
by the cones 
$\delta =  \delta^{(0)} =  {\rm Cone}( e_{1}, e_{2}, e_{3})
\text{ and } \delta^{(i)} =  {\rm Cone}(e , e_{j}\,|\, j\neq i),\,\, i = 1,2,3.$
\\

Now consider the $\mathbb{C}^{\star}$-surface 
$$X:= \mathbb{V}(z_{1}z_{2}^{2} + z_{2}z_{0}^{2} + z_{3}^{3}) \subset \mathbb{P}_{\mathbb{C}}^{3}
\text{ with action }\lambda\cdot [z_{0}: \ldots : z_{3}]  =  [z_{0}: \lambda^{-3}z_{1} : \lambda^{3}z_{2}: \lambda z_{3}].$$
Our goal is constructing  a divisorial fan for $X$ from weight packages $\theta^{(i)}$ of the  charts $X^{(i)}:= X_{\delta^{(i)}}\cap X$. 
\\

\emph{The chart $X^{(0)}$.}
Let us describe  $\theta^{(0)}=  (\mathbb{Z}^{3}, \mathbb{Z}, F^{(0)}, S^{(0)}, \Sigma^{(0)})$. 
The weight matrix is $F^{(0)} =   {}^{t}\begin{pmatrix} -3 & 3 & 1
\end{pmatrix}$, the $P$-matrix is  
$ P^{(0)} =  \begin{pmatrix}
1 & 1 & 0\\
1 & 0 & 3
\end{pmatrix},$
for $S^{(0)}$ we choose $\begin{pmatrix}
0 & 1 & -2
\end{pmatrix}$
and  the quotient fan $\Sigma^{(0)}$ is generated by the cones
${\rm Cone}((1,0), (1,1)), {\rm Cone}((0,1),(1,1)).$
\\

\emph{The chart $X^{(1)}$.} Next, we determine $\theta^{(1)} =  (\mathbb{Z}^{3}, \mathbb{Z}, F^{(1)}, S^{(1)}, \Sigma^{(1)})$ by enhancing $F^{(0)}$ to $\widehat{F} :=   {}^{t}\begin{pmatrix}
 0 & -3 & 3 & 1
\end{pmatrix}$ and doing $\widehat{F}  \mapsto \widehat{F} -(-3)  {}^{t}\begin{pmatrix}
 1 & 1 & 1 & 1
\end{pmatrix},$
where $-3$ is the weight of $z_{1}$. Removing the zero in the second entry we get
$F^{(1)} =   {}^{t}\begin{pmatrix} 3 & 6 & 4
\end{pmatrix}.$ Furthermore, we enhance the matrix factorization 
$$0\rightarrow \mathbb{Z}\xrightarrow{\widehat{F}} \mathbb{Z}^{4}\xrightarrow{\widehat{P}}\mathbb{Z}^{2}\rightarrow 0,
\text{ where } \widehat{P} = 
\begin{pmatrix}
-2 & 1 & 1 & 0\\
-4 & 1 & 0 & 3
\end{pmatrix}.$$ Observe that we add a new column in the first position of $P^{(0)}$ in order that the lines of $\widehat{P}$ are orthogonal to $(1,1,1,1)$ and ${}^{t}\widehat{F}$. Now  $P^{(1)}$ is obtained  by removing the second column of $\widehat{P}$. Similarly, we enhance $S^{(0)}$ to $\widehat{S}:= \begin{pmatrix} 1 & 0 & 1 & -2
\end{pmatrix}$ so that $S^{(1)} =   \begin{pmatrix} 1  & 1 & -2
\end{pmatrix}$. The quotient fan $\Sigma^{(1)}$ is the complete fan with rays
$\mathbb{Q}_{\geq 0}(-1,-2)$, $\mathbb{Q}_{\geq 0}(1,0)$, $\mathbb{Q}_{\geq 0}(0,1)$. It can be obtained from two manners:
 as the fan generated by $P^{(1)}(\delta_{0})$, where $\delta_{0}$ runs over the faces of $\delta = \delta^{(0)}$, or as the fan generated by $P^{(0)}(\delta)$, where 
$\delta$ runs over the faces of $\delta^{(1)}$.
\\

\emph{The other charts.} We construct the other weight packages $\theta^{(i)}$ by removing
the $(i+1)$-th column of $\widehat{P}$ and $\widehat{S}$.
To sum-up we obtain:
\begin{center}
 \begin{tabular}{|c| c |} 
 \hline
  &  Weight package $\theta^{(i)}$ \\ [1ex] 
 \hline\hline
 $\theta^{(0)}$ & $\left(\mathbb{Z}^{3}, \mathbb{Z}, {}^{t}\begin{pmatrix}
 -3 & 3 & 1
\end{pmatrix},  \begin{pmatrix} 0  & 1 & -2
\end{pmatrix}, \Sigma^{(0)}\right)$ \\ [2ex] 
 \hline
  $\theta^{(1)}$ & $\left(\mathbb{Z}^{3}, \mathbb{Z}, {}^{t}\begin{pmatrix}
3 & 6 & 4
\end{pmatrix},  \begin{pmatrix} 1  & 1 & -2
\end{pmatrix}, \Sigma^{(1)}\right)$  \\  [2ex] 
 \hline
  $\theta^{(2)}$ & $\left(\mathbb{Z}^{3}, \mathbb{Z}, {}^{t}\begin{pmatrix}
 -3 & -6 & -2
\end{pmatrix},  \begin{pmatrix} 1  & 0 & -2
\end{pmatrix}, \Sigma^{(2)}\right)$  \\ [2ex] 
 \hline
  $\theta^{(3)}$ &$\left(\mathbb{Z}^{3}, \mathbb{Z}, {}^{t}\begin{pmatrix}
 -1 & -4 & 2
\end{pmatrix},  \begin{pmatrix} 1  & 0 & 1
\end{pmatrix}, \Sigma^{(3)}\right)$  \\ [2ex] 
 \hline
\end{tabular}
\end{center}
\begin{center}
 \begin{tabular}{|c| c |} 
 \hline
  &  Quotient fan $\Sigma^{(i)}$ \\ [1ex] 
 \hline\hline
 $\Sigma^{(0)}_{\rm max}$ &  ${\rm Cone}((1,0), (1,1)),{\rm Cone}((0,1),(1,1))$ \\ [2ex] 
 \hline
  $\Sigma^{(1)}_{\rm max}$ & ${\rm Cone}((1,0),(0,1)), {\rm Cone}((1,0),(-1,-2)), {\rm Cone}((0,1),(-1,-2))$  \\  [2ex] 
 \hline
  $\Sigma^{(2)}_{\rm max}$ & ${\rm Cone}((1,1), (0,1)), {\rm Cone}((1,1),(-1,-2)), {\rm Cone}((0,1),(-1,-2))$  \\ [2ex] 
 \hline
  $\Sigma^{(3)}_{\rm max}$ &${\rm Cone}((1,1), (1,0)), {\rm Cone}((1,0), (-1,-2))$  \\ [2ex] 
 \hline
\end{tabular}
\end{center}
\emph{Constructing the divisorial fan.}
Let $\Sigma$ be the fan of Example \ref{refexamp}. For each $i$ there is an open subset 
$X_{\overline{\Sigma}^{(i)}}\subset X_{\Sigma}$ and a modification $f^{(i)}: X_{\overline{\Sigma}^{(i)}}\rightarrow X_{\Sigma^{(i)}}$.
The proper transform $C_{\theta^{(i)}}'\subset  X_{\overline{\Sigma}^{(i)}}$ of $C_{\theta^{(i)}}$ have
local equations as in Example \ref{refexamp01}. Denote by $\kappa^{(i)}:\widehat{C_{\theta}}\rightarrow X_{\Sigma^{(i)}}$ the morphism obtained by composing the normalization of $C_{\theta^{(i)}}'$, the closed immersion $C_{\theta^{(i)}}'\hookrightarrow X_{\overline{\Sigma}^{(i)}}$ and the modification $f^{(i)}$. Set $\bar{\mathfrak{D}}^{(i)}_{\theta} =  \kappa^{(i)\star}\mathfrak{D}_{\theta^{(i)}}$ and denote by $\bar{C}$ the smooth completion of the affine plane curve $\mathbb{V}(y+x + x^{2})$.
Using Example \ref{refexamp03},  the divisorial fan $\mathscr{E}_{\theta}$  of the $\mathbb{C}^{\star}$-surface $X$ is:
\begin{center}
 \begin{tabular}{|c| c |} 
 \hline
  &  Divisorial fan $\mathscr{E}_{\theta}$ \\ [1ex] 
 \hline\hline
 $\bar{\mathfrak{D}}^{(0)}_{\theta}$ & $[0, 1/3]\cdot [\zeta_{1,1)}] + \{-2/3\}\cdot [\zeta_{(0,1)}]\text{ over }\bar{C}\setminus\{\zeta_{(-1,-2)}\}$ \\ [2ex] 
 \hline
 $\bar{\mathfrak{D}}^{(1)}_{\theta}$ & $ \mathbb{Q}_{\geq 1/3}\cdot [\zeta_{(1,1)}] + \mathbb{Q}_{\geq -2/3}\cdot [\zeta_{(0,1)}]+\mathbb{Q}_{\geq 1/2}\cdot [\zeta_{(-1,-2)}] \text{ over }\bar{C}$ \\  [2ex] 
 \hline
 $\bar{\mathfrak{D}}^{(2)}_{\theta}$ & $ \mathbb{Q}_{\leq -2/3}\cdot [\zeta_{(0,1)}]+\mathbb{Q}_{\leq 1/2}\cdot [\zeta_{(-1,-2)}]\text{ over }\bar{C}$ \\ [2ex] 
 \hline
 $\bar{\mathfrak{D}}^{(3)}_{\theta}$ & $\{1/2\}\cdot [\zeta_{(-1,-2)}]\text{ over } \bar{C}\setminus\{\zeta_{(0,1)}\}$ \\ [2ex] 
 \hline
\end{tabular}
\end{center}
\end{example}
\subsection{Linear torus actions: the projective case}\label{sec-projproj}
Inspired by Example \ref{sec-exex}, we now treat the projective case. 
 Main result of this section, Theorem \ref{theo-main-nonnorm}, is an extension of Theorem \ref{theo-aff-impl}. 
\begin{definition}\label{def-enhance}
Let $\theta =  (\overline{N}\simeq \mathbb{Z}^{\ell}, N \simeq \mathbb{Z}^{n}, F, S, \Sigma)$ be an \emph{abstract weight package}. This means that $F: N\rightarrow \overline{N}$ is an injective morphism, $S: \overline{N}\rightarrow N$ is a section. The symbol $\Sigma$ stands for the fan of ${\rm Coker}(F)_{\mathbb{Q}}$ generated by the cones $P(\delta_{0})$, where $\delta_{0}$ is a face of the first quadrant $\delta =  \mathbb{Q}_{\geq 0}^{\ell}\subset \overline{N}_{\mathbb{Q}}$ and $P:\overline{N}\rightarrow {\rm Coker}(F)$ is the quotient. We define the sequence 
$$\underline{\theta}:= (\theta^{(0)}, \theta^{(1)}, \ldots, \theta^{(\ell)})$$ of weight packages by \emph{enhancing} $\theta$ into a weight package $\widehat{\theta}$. Fix  basis 
so that we identify $N$ and $\overline{N}$  with $\mathbb{Z}^{n}, \mathbb{Z}^{\ell}$, see $F$ as an 
$\ell\times n$-matrix $(a_{i,j})$ and $P$ as an $s\times \ell$-matrix $(b_{i,j})$, where $s$ is the rank of ${\rm Coker}(F)$. The weight package $\widehat{\theta}$
is 
$(\mathbb{Z}\oplus \overline{N} =  \mathbb{Z}^{\ell +1}, N  =  \mathbb{Z}^{n}, \widehat{F}, \widehat{S}, \widehat{\Sigma}),$ where:
\begin{itemize}
\item[(1)]
$$\widehat{F} = \begin{bmatrix} 
  0 & \dots & 0 \\
    a_{1,1} & a_{1,2} & \dots \\
    \vdots & \ddots & \\
    a_{\ell,1} &        & a_{\ell, n} 
    \end{bmatrix} \text{ and }
\widehat{P} =   \begin{bmatrix} 
    b_{1, 0} & b_{1,1} & \dots & b_{1, \ell}\\
    \vdots &  \vdots &     & \vdots  \\
    b_{s, 0} & b_{s,1} & \dots & b_{s, \ell}
    \end{bmatrix}\in {\rm Mat}_{s\times (\ell + 1)}(\mathbb{Z})$$
with the condition $\sum_{j = 0}^{\ell}b_{i,j} = 0$ for any $i\in \{1, \ldots, s\};$ In this way, we get the matrix factorization 
$$0\rightarrow N\xrightarrow{\widehat{F}}\mathbb{Z}\oplus \overline{N}\xrightarrow{\widehat{P}} {\rm Coker}(F) =  {\rm Coker}(\widehat{F})\rightarrow 0.$$
\item[(2)] We define the fan $\widehat{\Sigma}$ as the fan generated by the cones $\widehat{P}(\delta_{0})$, where $\delta_{0}$ runs over the faces of the first quadrant $\widehat{\delta} =  \mathbb{Q}^{\ell+1}_{\geq 0}\subset\mathbb{Q}\oplus \overline{N}_{\mathbb{Q}}$. 
\item[(3)] Similarly, the section $\widehat{S}$ is the matrix 
$$\begin{bmatrix} 
    s_{1,0} & s_{1,1} & \dots & s_{1, \ell}\\
    \vdots &  \vdots &     & \vdots  \\
    s_{n,0} & s_{n,1} & \dots & s_{n, \ell}
    \end{bmatrix}\in {\rm Mat}_{n\times (\ell + 1)}(\mathbb{Z}), \text{ where } S =  (s_{i,j})$$
and with the condition $\sum_{j = 0}^{\ell}s_{i,j} =  0$ for any $i\in \{1, \ldots, n\}$.
\end{itemize}
Now define $\theta^{(0)}$ as $\theta$ and for $v\in \{1, \ldots, \ell\}$,
$$\theta^{(v)} =  (\overline{N}, N, F^{(v)}, S^{(v)}, \Sigma^{(v)}) \text{ where: }$$
\begin{itemize}
\item[(I)] $F^{(v)}$ is obtained by omitting the line of index $v$ of the matrix 
$$\widehat{F}  -  \begin{bmatrix} 
     a_{v,1} & \dots & a_{v, n}\\
      \vdots &     & \vdots  \\
     a_{v,1} & \dots & a_{v, n}
    \end{bmatrix}\in {\rm Mat}_{(\ell+1)\times n}(\mathbb{Z});$$
\item[(II)] The linear map $F^{(v)}$ induces the matrix factorization 
$$0\rightarrow N \xrightarrow{F^{(v)}} \overline{N}\xrightarrow{P^{(v)}}{\rm Coker}(F^{(v)})\rightarrow 0,$$
where $P^{(v)}$ is obtained by omitting the column of index $v$ of the matrix $\widehat{P}$. The fan of $\Sigma^{(v)}$ of ${\rm Coker}(F^{(v)})_{\mathbb{Q}}$ is fan generated by the cones $P^{(v)}(\delta_{0})$, where $\delta_{0}$ runs over the set of faces of $\delta =  \mathbb{Q}_{\geq 0}^{\ell}\subset\overline{N}_{\mathbb{Q}}$.
\item[(III)] The section $S^{(v)}$ is obtained by omitting the column of index $v$ of the matrix $\widehat{S}$.
\end{itemize}
\end{definition}
The following builds affine subvarieties with linear torus action from weight packages. 
\begin{lemma}\label{lem-equa-tvar}
Let $\theta =  (\overline{N}=  \mathbb{Z}^{\ell}, N  =  \mathbb{Z}^{n}, F, S,\Sigma)$ be a weight package.
Let $C\subset \mathbb{T}_{{\rm Coker}(F)}$ be an irreducible curve with defining ideal $I_{C} =  (f_{i}(x_{1}, \ldots, x_{s})\,|\, 1\leq i\leq a)$. Let $C_{\theta}$ be Zariski closure of $C$ in $X_{\Sigma}$. Consider the matrix 
$$P =   \begin{bmatrix} 
     b_{1,1} & \dots & b_{1, \ell}\\
    \vdots &     & \vdots  \\
     b_{s,1} & \dots & b_{s, \ell}
    \end{bmatrix}\in {\rm Mat}_{s\times \ell}(\mathbb{Z})$$
coming from the matrix factorization of $\theta$.
Set 
$$g_{i}(T_{1}, \ldots, T_{\ell}):= f_{i}\left(\prod_{j = 1}^{\ell}T_{j}^{b_{1, j}}, \ldots, \prod_{j = 1}^{\ell}T_{j}^{b_{s, j}}\right)\text{ for } 1\leq i\leq a,$$
and assume that there are Laurent monomials $u_{i}\in \mathbb{C}[T_{1}, T_{1}^{-1}, \ldots, T_{\ell}, T_{\ell}^{-1}]$ such
that $$X: = \mathbb{V}(u_{1}g_{1}, \ldots, u_{a}g_{a})\subset \mathbb{A}^{\ell}_{\mathbb{C}}$$ is irreducible. 
Then $X\subset \mathbb{A}^{\ell}_{\mathbb{C}}$ is a subvariety with linear torus action of complexity one and its normalization is equivariantly isomorphic to $X(\widehat{C_{\theta}}, \bar{\mathfrak{D}}_{\theta})$, where $\widehat{C_{\theta}}$ is the normalization of $C_{\theta}$ and $\bar{\mathfrak{D}}_{\theta}$ is pullback polyhedral divisor as in Definition \ref{def-pullbackpol}.
\end{lemma}
\begin{proof}
Set $F =  (a_{i, j})_{1\leq i\leq \ell, 1\leq j\leq n}$. Then the action on the coordinates is given by 
$$\mu\cdot (x_{1}, \ldots, x_{n}) = \left(\left(\prod_{j = 1}^{n}\mu_{j}^{a_{1, j}}\right) \cdot x_{1}, \ldots, \left(\prod_{j = 1}^{n}\mu_{j}^{a_{\ell, j}}\right) \cdot x_{\ell}\right)$$
for any $\mu =  (\mu_{1}, \ldots, \mu_{n})\in \mathbb{T}_{N} =  (\mathbb{C}^{\star})^{n}$.
Remarking that $g_{i}$ is a $\mathbb{T}_{N}$-invariant function (since $P\circ F =  0$) and $u_{i}$ is a $\mathbb{T}_{N}$-eigenfunction,  the subset $X\subset \mathbb{A}^{\ell}_{\mathbb{C}}$ is stable $\mathbb{T}_{N}$-action. Let  $\psi: \mathbb{G}\rightarrow \mathbb{T}_{{\rm Coker}(F)}$ be the quotient. Note that $X$ intersects $\mathbb{G}$ because an element of $C$ lifts via $\psi$ to solution  
of the equations $g_{i}$. Also the map $X\cap \mathbb{G}\rightarrow C$, $x\mapsto \psi(x)$ is a geometric quotient for the $\mathbb{T}_{N}$-action. So $X\subset \mathbb{A}^{\ell}_{\mathbb{C}}$ is
a subvariety with linear torus action of complexity one and 
we conclude by Theorem \ref{theo-aff-impl}.
\end{proof}
\begin{example}
 Lemma \ref{lem-equa-tvar}  for $C$ smooth does not give, in general, normal varieties.
Since for the weight package $\theta =  (\overline{N}, N, F, S, \Sigma)$ with
$\overline{N} =  \mathbb{Z}^{4}$, $N =  \mathbb{Z}^{2}$,
$$F =  {}^{t}\begin{pmatrix}
  4 & 3  &  0 & 12 \\
  0& 0  &  1 & -1

\end{pmatrix}, \,\, S =  \begin{pmatrix}
 1 & -1  &  0 & 0\\
 0 & 0  &  1 & 0
\end{pmatrix}, $$
the fan $\Sigma$ of $\mathbb{P}^{2}_{\mathbb{C}}$, and the elliptic curve
$C = \mathbb{V}(z_{0}z_{2}^{2} + z_{1}^{3} + z_{0}^{2}z_{1}) \subset \mathbb{P}^{2}_{\mathbb{C}},$
we get the non-normal hypersurface
$$X  =  \mathbb{V}(x_{1}^{8}x_{3}x_{4} + x_{2}^{12} + x_{2}^{4}x_{3}^{2}x_{4}^{2})\subset \mathbb{A}^{4}_{\mathbb{C}}.$$
\end{example}
\begin{proposition}\label{prop-Construction2-irredhyper}
Let
$\theta =   (\overline{N}=  \mathbb{Z}^{\ell}, N  =  \mathbb{Z}^{n}, F, S,\Sigma)$ be a weight package. Consider
a Laurent polynomial of the form
$$g_{1}(T_{1}, \ldots, T_{\ell}):= f_{1}\left(\prod_{j = 1}^{\ell}T_{j}^{b_{1, j}}, \ldots, \prod_{j = 1}^{\ell}T_{j}^{b_{s, j}}\right),$$
where  $f_{1}\in\mathbb{C}[\mathbb{T}_{{\rm Coker}(F)}] =  \mathbb{C}[x_{1}, x_{1}^{-1}, \ldots, x_{s}, x_{s}^{-1}]$ is irreducible and the $b_{i, j}$ are the coefficients of the $P$-matrix of $\theta$. 
 Then there is a Laurent monomial $u_{1}\in \mathbb{C}[T_{1}, T_{1}^{-1}, \ldots, T_{\ell}, T_{\ell}^{-1}]$
such that $u_{1} g_{1}\in \mathbb{C}[T_{1}, \ldots, T_{\ell}]$ is irreducible. 
\end{proposition}
\begin{proof}
Let $v_{1}, \ldots, v_{s}$ be the vectors of $\mathbb{Z}^{\ell}$ corresponding to the first $s$ lines of $P$. We may find $v_{s+1}, \ldots, v_{\ell}$ such that $(v_{1}, \ldots, v_{\ell})$ is a basis of $\mathbb{Z}^{\ell}$. Denote 
by $v_{i} =  (v_{i, 1}, \ldots, v_{i, \ell})$ the coordinates of $v_{i}$ and set 
$X_{i}  := \prod_{j =  1}^{\ell} T_{j}^{v_{i, j}}.$
Then
$$\mathbb{C}[T_{1}, T_{1}^{-1}, \ldots, T_{\ell}, T_{\ell}^{-1}] =  \mathbb{C}[X_{1}, X_{1}^{-1}, \ldots, X_{\ell}, X_{\ell}^{-1}]
\text{ and }g_{1}(T_{1}, \ldots, T_{\ell}) =  f_{1}(X_{1},\ldots, X_{s}).$$  Set $R =  \mathbb{C}[X_{1}, \ldots, X_{\ell}]$. Let $S, V\in R$ such that $f_{1}(X_{1}, \ldots, X_{s})  =  SV$. Since 
${\rm deg}_{X_{i}}(S)  + {\rm deg}_{X_{i}}(V)  =  {\rm deg}_{X_{i}}(f_{1}(X_{1},\ldots, X_{s}) =  0$ for $i =  s+1, \ldots, \ell$, we have $S, V\in \mathbb{C}[X_{1},\ldots, X_{s}]$. But $f_{1}(X_{1},\ldots, X_{s})$ is irreducible in $\mathbb{C}[X_{1}, X_{1}^{-1}, \ldots, X_{s}, X_{s}^{-1}]$ by assumption. So we may 
assume that $S$ is a monomial and $V$ is a product of a monomial and an irreducible element of $R$,  whence the irreducibility of $g_{1}$ in $\mathbb{C}[T_{1}, T_{1}^{-1}, \ldots, T_{\ell}, T_{\ell}^{-1}]$. Thus, the result follows.
\end{proof}
\begin{corollary}\label{xoro-expliteqvarT}
Any subvariety $X\subset \mathbb{A}^{\ell}_{\mathbb{C}}$ with linear torus action of complexity one has a
decomposition $X  =  \mathbb{V}(u_{1}g_{1}, \ldots, u_{a}g_{a})$ as in Lemma \ref{lem-equa-tvar}.
\end{corollary}
\begin{proof}
The $\mathbb{T}_{N}$-action on $\mathbb{A}^{\ell}_{\mathbb{C}}$ induces an $M$-grading on the ring $R:= \mathbb{C}[x_{1}, \ldots, x_{\ell}]$. Declare an element of $R$ homogeneous if it is homogeneous with respect to the $M$-grading.
The vanishing ideal of $X$ in $R$ is generated by non-constant homogeneous polynomials $h_{1}, \ldots, h_{a}\in R$. Set $X_{i}:= \mathbb{V}(h_{i})\subset \mathbb{A}^{\ell}_{\mathbb{C}}$ and consider the decomposition
$$ X_{i} =  \bigcup_{j =  1}^{s_{i}}X_{i, j}, \,\, 1\leq i\leq a$$
in irreducible components. Since $X_{i}$ is $\mathbb{T}_{N}$-stable, each component $X_{i, j}$ is of the form $\mathbb{V}(h_{i, j})\subset \mathbb{A}^{\ell}_{\mathbb{C}}$ for some homogeneous irreducible $h_{i, j}\in R$. Let $E$ be the set of pairs $(i, j)$ with $1\leq i\leq a$ and $1\leq j\leq s_{i}$
such that $h_{i, j}$ is not a scalar multiplication of a coordinate $x_{k}$. By 
Proposition \ref{prop-Construction2-irredhyper} we have  for $(i, j)\in E$ a decomposition $X_{i, j}  =  \mathbb{V}(u_{i, j} g_{i, j})$
obtained by pulling back $\psi(\mathbb{G}\cap X_{i, j}) = \mathbb{V}(f_{i, j})\subset \mathbb{T}_{{\rm Coker}(F)}$ by the quotient $\psi: \mathbb{G}\rightarrow \mathbb{T}_{{\rm Coker}(F)}$, where $u_{i, j}$ is a monomial. 
For $(i, j)\not\in E$ with $1\leq i\leq a$ and $1\leq j\leq s_{i}$ we set $u_{i, j} =  h_{i, j}$ and $g_{i, j} =  1$. Furthermore, for $1\leq i\leq a$
write 
$$ u_{i} =  \prod_{j = 1}^{s_{i}} u_{i, j},\,\, g_{i} =  \prod_{j =  1}^{s_{i}}g_{i, j} \text{ and }f_{i}  =  \prod_{j =  1}^{s_{i}}f_{i, j}.$$
Then $X  =  \mathbb{V}(u_{1}g_{1}, \ldots, u_{r}g_{a})$. Moreover, the zero locus $\mathbb{V}(f_{1}, \ldots, f_{r})\subset \mathbb{T}_{{\rm Coker}(F)}$ is the image $\psi(X\cap \mathbb{G})$, whence the result. 
\end{proof}
\begin{lemma}\label{lem-P-matrix}
Let $\theta =  (\overline{N}, N, F, S, \Sigma)$ be a weight package and consider the sequence 
$\theta^{(i)} =  (\overline{N}, N, F^{(i)}, S^{(i)}, \Sigma^{(i)})$ that comes from the enhanced structure (see Definition \ref{def-enhance}).
Let 
$$\delta =  \delta^{(0)} =  \sum_{i =  0}^{\ell}\mathbb{Q}_{\geq 0}e_{i},\,\, \delta^{(j)} =  \mathbb{Q}_{\geq 0}e +\sum_{i\neq j}\mathbb{Q}_{\geq 0}e_{i} \text{ for }j = 1,2, \ldots, \ell,$$
be the maximal cones generated the fan of $\mathbb{P}_{\mathbb{C}}^{\ell}$, where $(e_{1}, \ldots, e_{\ell})$ is the standard
basis of $\overline{N}_{\mathbb{Q}} =  \mathbb{Q}^{\ell}$ and $e  = -\sum_{i =1}^{\ell}e_{i}$. Then $\Sigma^{(i)}$
is the fan generated by the cones $P(\delta')$, where $\delta'$ runs over the faces of $\delta^{(i)}$.
Moreover, we have the identity 
$$S^{(i)}(P^{(i)-1}(v)\cap \delta^{(0)})  =  S(P^{-1}(v)\cap \delta^{(i)})$$
for all $i\in\{1, 2, \ldots, \ell\}$ and $v\in |\Sigma^{(i)}|$. 
\end{lemma}
\begin{proof}
First claim is consequence of the identities
$$P^{(a)}(e_{1}) =  P(e);\, P^{(a)}(e_{i + 1}) =  P(e_{i}) \text{ for }1\leq i\leq a;$$
$$P^{(a)}(e_{j}) =  P(e_{j}) \text{ for } a<j\leq \ell.$$
Similarly, consider the map 
$$f: \delta^{(i)}\rightarrow \delta^{(0)},\,\, \lambda e  + \sum_{j\neq i}\lambda_{j}e_{j}\mapsto \lambda e_{1} + \sum_{1\leq j\leq i}\lambda_{j} e_{j +1} + \sum_{i< j\leq \ell}\lambda_{j}e_{j}\,\, (\lambda_{j}, \lambda\in\mathbb{Q}_{\geq 0}).$$
Then
$$f(P^{-1}(v)\cap \delta^{(i)}) =  P^{(i)-1}(v)\cap \delta^{(0)}.$$
As $S^{(i)}(f(w)) =  S(w)$ for any $w\in \delta^{(i)}$, we conclude that 
$$S^{(i)}(P^{(i)-1}(v)\cap\delta^{(0)}) =  S^{(i)}(f(P^{-1}(v)\cap \delta^{(i)})) =  S(P^{-1}(v)\cap \delta^{(i)}),$$
finishing the proof of the lemma. 
\end{proof}

\begin{theorem}\label{theo-main-nonnorm}
\begin{itemize}
\item[(1)] Let $X\subset \mathbb{P}^{\ell}_{\mathbb{C}}$ be a subvariety with linear torus action of complexity one. Let $\theta$ be the weight package of $X$ and consider the sequence 
$$\theta^{(i)} =  (\overline{N}, N, F^{(i)}, S^{(i)}, \Sigma^{(i)}),\,\, 0\leq i\leq \ell,$$ that comes from the enhanced structure (see Definition  \ref{def-enhance}).
Let $\overline{\Sigma}$ be a fan with support $\bigcup_{i =  0}^{\ell}|\Sigma^{(i)}|$ such that for $0\leq i\leq \ell$
the set
$\overline{\Sigma}^{(i)} =  \{\sigma \in \overline{\Sigma}\,|\, \sigma \subset |\Sigma^{(i)}|\}$
is a projective fan subdivision of $\Sigma^{(i)}$. Denote by $\kappa^{(i)}: \widehat{C_{\theta^{(i)}}}\rightarrow X_{\Sigma^{(i)}}$ the map, which is the composition of the projective modification $f^{(i)}: X_{\overline{\Sigma}^{(i)}}\rightarrow X_{\Sigma^{(i)}}$, the inclusion $C_{\theta^{(i)}}'\rightarrow X_{\overline{\Sigma}^{(i)}},$ where $C_{\theta^{(i)}}'$ is the proper transform of $C_{\theta^{(i)}}$ under $f^{(i)}$, and the normalization 
$\widehat{C_{\theta^{(i)}}}\rightarrow C_{\theta^{(i)}}'$. 
Then the set
$$\{\bar{\mathfrak{D}}_{\theta}^{(i)} =  \kappa^{(i) \star}\mathfrak{D}_{\theta^{(i)}}\,\,|\,\, i =  0,1, \ldots, \ell\}$$
generates  a divisorial fan $\mathscr{E}_{\theta}$ over a smooth projective curve $\overline{C}$. Moreover the normalization of $X$ is equivariantly isomorphic to $X(\overline{C},\mathscr{E}_{\theta})$.
\item[(2)] Conservely, let $\theta =  (\overline{N}, N, F, S, \Sigma)$ be a weight package. Let $C\subset \mathbb{T}_{{\rm Coker}(F)} =  (\mathbb{C}^{\star})^{s}$ be an irreducible curve with defining ideal
$$I_{C} =  (f_{i}(x_{1}, \ldots, x_{s})\,|\, 1\leq i\leq a).$$
Consider the enhanced $P$-matrix 
$$\widehat{P} =   \begin{bmatrix} 
    b_{1,0} & b_{1,1} & \dots & b_{1, \ell}\\
    \vdots &  \vdots &     & \vdots  \\
    b_{s, 0} & b_{s,1} & \dots & b_{s, \ell}
    \end{bmatrix}\in {\rm Mat}_{s\times \ell + 1}(\mathbb{Z})$$
as in Definition \ref{def-enhance}. Set 
$$g_{i}(T_{0},T_{1}, \ldots, T_{\ell}):= f_{i}\left(\prod_{j = 0}^{\ell}T_{j}^{b_{1, j}}, \ldots, \prod_{j = 0}^{\ell}T_{j}^{b_{s, j}}\right)\text{ for } 1\leq i\leq a,$$
and assume that there exist Laurent monomials $u_{i}\in \mathbb{C}[T_{0}, T_{0}^{-1}, \ldots, T_{\ell}, T_{\ell}^{-1}]$ 
such that $X:= \mathbb{V}(u_{1}g_{1}, \ldots, u_{a}g_{a})\subset \mathbb{P}^{\ell}_{\mathbb{C}}$ is irreducible. Then $X$ is a subvariety with linear torus action of complexity one and its normalization is described by the divisorial fan $\mathscr{E}_{\theta}$ obtained from Contruction $(1)$. 
\item[$(3)$] If $X \subset \mathbb{P}^{\ell}_{\mathbb{C}}$ is a projective subvariety with linear torus of complexity one with weight package $\theta$, then $X$ admits a decomposition $X  =  \mathbb{V}(u_{1}g_{1}, \ldots, u_{a}g_{a})$ as in Assertion $(2)$. 
\end{itemize}
\end{theorem}
\begin{proof}
$(1)$ Lemma \ref{lem-P-matrix} implies that the  $X_{\overline{\Sigma}^{(i)}}$ define a Zariski open covering 
of $X_{\overline{\Sigma}}$, while Lemma \ref{lem-poltheta} gives that $X(\mathfrak{D}_{\theta^{(i)}})$ is $\mathbb{T}_{N}$-isomorphic to $X_{\delta^{(i)}}$ for $0\leq i\leq \ell$. So it follows from \cite[Proposition 4.3]{AHS08} (see the proof argument of \cite[Theorem 5.6]{AHS08})
that the set 
$$\{f^{(i)\star}(\mathfrak{D}_{\theta^{(i)}})\,\,|\,\, i = 0,1, \ldots, \ell\}$$
generates a divisorial fan $\overline{\mathscr{E}}$ over $(X_{\overline{\Sigma}}, N)$ describing the $\mathbb{T}_{N}$-variety $\mathbb{P}_{\mathbb{C}}^{\ell}$. Let $\eta: \widehat{X}\rightarrow X(\overline{\mathscr{E}})\simeq \mathbb{P}^{\ell}_{\mathbb{C}}$ be the composition of the the normalization $\widehat{X}\rightarrow X$ with the closed immersion $X\hookrightarrow \mathbb{P}^{\ell}_{\mathbb{C}}$. Note that the polyhedral divisorial $\bar{\mathfrak{D}}_{\theta}^{(i)}$
encodes
the normalization of the affine $\mathbb{T}_{N}$-variety
$X^{(i)}: =  X_{\delta^{(i)}}\cap X$ for $0\leq i\leq \ell$ by virtue of  Proposition \ref{prop-sub-kappa} and Theorem \ref{theo-aff-impl}. So for all $i, j\in \{0, 1, \ldots, \ell\}$ we have the natural identifications
$$X(\mathfrak{D}_{\theta^{(i)}})\cap X(\mathfrak{D}_{\theta^{(j)}}) \simeq \eta^{-1}(X(f^{(i)\star}(\mathfrak{D}_{\theta^{(i)}}) \cap f^{(i)\star}(\mathfrak{D}_{\theta^{(j)}}))) \simeq  X(  \kappa^{(i) \star}\mathfrak{D}_{\theta_{(i)}}\cap   \kappa^{(j) \star}\mathfrak{D}_{\theta_{(j)}}),$$
showing that $\mathscr{E}_{\theta}$ generates a divisorial fan.  We conclude that  the normalization of $X$ is equivariantly isomorphic to $X(\mathscr{E}_{\theta})$. 
\\

$(2)$ Consequence of Lemma \ref{lem-equa-tvar} applied to each chart $X^{(i)} =  X\cap X_{\delta^{(i)}}$ and Construction $(1)$.
\\

$(3)$ Follows from the affine case (see Corollary \ref{xoro-expliteqvarT}) via homogeneization. 
\end{proof}
\subsection{Geometry of the contraction space}
From the suggestion in \cite[Section 4.2, Remark 15]{AIPSV12},
we now study contraction spaces of torus actions of complexity one via weight packages. 
\begin{lemma}\label{lemma-Delta}
Let $\theta =  (\overline{N} =  \mathbb{Z}^{\ell}, N, S, F, \Sigma)$ be a weight package. Let $\Delta$ be the smallest fan of $\overline{N}_{ \mathbb{Q}} =  \mathbb{Q}^{\ell}$ with support $\delta = \mathbb{Q}_{\geq 0}^{\ell}$ such that the map
$$\gamma:\Delta\rightarrow \Sigma, \,\, \delta_{0}\mapsto P(\delta_{0}),$$
induced by the matrix factorization of $\theta$, is a fan morphism. Then the toric $\mathbb{G}$-variety $X_{\Delta}$ is
 the normalization $Z$ of the Zariski closure of the graph of the $\mathbb{T}_{N}$-invariant rational map $\mathbb{A}_{\mathbb{C}}^{\ell}\dashrightarrow X_{\Sigma}$ in $\mathbb{A}_{\mathbb{C}}^{\ell}\times X_{\Sigma}$. 
\end{lemma}
\begin{proof}
Existence and unicity of $\Delta$ follow from that $\Delta$ must be the fan of $\overline{N}_{\mathbb{Q}}$
generated by the cone $P^{-1}(\tau)\cap \delta$, where $\tau$ runs over $\Sigma$. The map $\gamma$ induces a toric morphism
$p: X_{\Delta}\rightarrow X_{\Sigma}$. We see that for any $\sigma\in\Sigma_{\rm max}$,
we have $p^{-1}(X_{\sigma}) =  X_{P^{-1}(\sigma)\cap \delta}$.  So $p$ is affine. The existence of two morphisms $p:X_{\Delta}\rightarrow X_{\Sigma}$
and $\pi_{\Delta}: X_{\Delta}\rightarrow \mathbb{A}_{\mathbb{C}}^{\ell}$ induces a morphism $\varphi: X_{\Delta}\rightarrow Z$, which is birational. Since the composition $\pi_{\Delta}$ of $\varphi$ with the natural morphism $Z\rightarrow \mathbb{A}_{\mathbb{C}}^{\ell}$ is proper, $\varphi$ is proper.
Similarly, as the quotients $p:X_{\Delta}\rightarrow X_{\Sigma}$ and $Z\rightarrow X_{\Sigma}$ are affine, the morphism $\varphi$ is affine. So $\varphi$ is birational finite and by Zariski's Main Theorem, $\varphi$ is an isomorphism, proving the lemma. 
\end{proof}

\begin{definition}
The fan $\Delta$ in Lemma \ref{lemma-Delta} is the \emph{lifting fan} of the weight package $\theta$. 
\end{definition}
\begin{corollary}\label{Cor-Delta1}
Let $X\subset \mathbb{A}_{\mathbb{C}}^{\ell}$ be a subvariety with linear torus action of complexity one  and let $\theta$ be its weight package.
With the notation of Lemma \ref{lemma-Delta}, consider the toric map $\pi_{\Delta}: X_{\Delta}\rightarrow \mathbb{A}_{\mathbb{C}}^{\ell}$, where $\Delta$ is the lifting fan of $\theta$. Then the normalization of the proper transform of $X$ 
under $\pi_{\Delta}$ is the contraction space $\widetilde{X}$ of $X$ and the morphism  $\pi': \widetilde{X}\rightarrow X$ induced by restriction of $\pi_{\Delta}$ is the contraction morphism $\pi: \widetilde{X}\rightarrow X$. 
\end{corollary}
\begin{proof}
Let $S$ be the Zariski closure of the graph of $\mathbb{A}_{\mathbb{C}}^{\ell}\dashrightarrow X_{\Sigma}$ in $\mathbb{A}_{\mathbb{C}}^{\ell}\times X_{\Sigma}$. The proper transform $X_{1}$ of $X$ under the natural map $S\rightarrow \mathbb{A}_{\mathbb{C}}^{\ell}$ is the closure
of 
$\{(x, \psi(x))\,|\, x\in X\cap \mathbb{G}\}\subset S,$ where $\psi:\mathbb{G}\rightarrow \mathbb{T}_{{\rm Coker}(F)}$ is the quotient. Thus the normalization of $X_{1}$ is $\widetilde{X}$.
\end{proof}
\begin{corollary}
With the notation of Corollary \ref{Cor-Delta1}, assume $X_{\Sigma}$ affine. Then the rational $\mathbb{T}_{N}$-invariant map 
$\mathbb{A}_{\mathbb{C}}^{\ell}\dashrightarrow X_{\Sigma}$ is the global quotient $\mathbb{A}_{\mathbb{C}}^{\ell}\rightarrow \mathbb{A}_{\mathbb{C}}^{\ell} /\!\!/\mathbb{T}_{N}$. In particular, the contraction morphism $\pi:\widetilde{X}\rightarrow X$
is  the normalization of $X$.
\end{corollary}
\begin{proof}
Indeed, as $p:X_{\Delta}\rightarrow X_{\Sigma}$ is affine,  $X_{\Delta}$ is affine. Moreover, the support of $\Delta$ is $\delta$. So $\pi_{\Delta}$ is the identity and by Corollary \ref{Cor-Delta1}, $\pi$ is the normalization. 
\end{proof}
Let us give a concrete example.
\begin{example}\label{ex-Delta}
For $\lambda\in \mathbb{C}^{\star}$ take $X  =  \mathbb{V}(x_{3}-\lambda)\subset \mathbb{A}_{\mathbb{C}}^{3}$
with weight package 
$$\theta  =  \left(\mathbb{Z}^{3}, \mathbb{Z}, {}^{t}\begin{pmatrix}
 1  & 1 & 0
\end{pmatrix}, \begin{pmatrix}
1  & 0 & 0
\end{pmatrix}, 
\Sigma\right)$$
and set 
$$
v_{1} =  (1,0,0), v_{2} =  (1,1, 0), v_{3} =  (0, 1,0), v_{4} =  (0,0,1).$$
The lifting fan $\Delta$ is  generated by 
${\rm Cone}(v_{1}, v_{2} ,v_{4}), {\rm Cone}(v_{2}, v_{3}, v_{4}).$ Using
homogeneous coordinates (see Subsection \ref{Section-Cox}),
we have $$X_{\Delta} = (\mathbb{A}^{4}_{\mathbb{C}}\setminus \mathbb{V}(t_{1}, t_{3}))/\mathbb{G}_{\Delta},$$
where $\mathbb{G}_{\Delta} =  \mathbb{C}^{\star}$ acts on $\mathbb{A}_{\mathbb{C}}^{4}$ via 
$$\mu\cdot (t_{1}, \ldots,t_{4}) =  (\mu t_{1}, \mu^{-1} t_{2}, \mu t_{3}, t_{4})$$
and the variables $t_{1}, \ldots, t_{4}$ correspond to $v_{1}, \ldots, v_{4}$. 
In order to describe $\pi_{\Delta}$, consider the matrix 
$$\mathscr{Q} =  \begin{pmatrix}
 1& 1  &  0 &  0\\
 0& 1  &  1 &  0\\
 0 & 0 &  0&  1
\end{pmatrix}$$
whose the columns are the vectors $v_{1}, \ldots, v_{4}$. Then
$$ \pi_{\Delta} :(\mathbb{A}^{4}_{\mathbb{C}}\setminus Z(\Delta))/\mathbb{G}_{\Delta}\rightarrow \mathbb{A}_{\mathbb{C}}^{3}, \,\, [t_{1}:\ldots: t_{4}]_{\Delta}\mapsto (t_{1}t_{2}, t_{2}t_{3}, t_{4}),$$
where the monomials correspond to the lines of $\mathscr{Q}$. Moreover, the quotient for the $\mathbb{C}^{\star}$-action is
$$X_{\Delta}\rightarrow \mathbb{P}^{1}_{\mathbb{C}}\times \mathbb{A}^{1}_{\mathbb{C}},\,\, [t_{1}:\ldots:t_{4}]_{\Delta}\mapsto ([t_{1}: t_{3}], t_{4}),$$
and the proper transform of $X$ under $\pi_{\Delta}$ is 
$X_{1} =  \{[t_{1}: \ldots: t_{4}]_{\Delta}\, |\, t_{4} =  \lambda\}\subset X_{\Delta}.$
So $X_{1}  =  \widetilde{X}$ and the contraction map $\pi = \pi_{\Delta|X_{1}}$ is the blowing-up of $X$ at $(0,0, \lambda)$. 
\end{example}
Let us formalize what we saw in Example \ref{ex-Delta}. Recall the notion of the contraction divisorial fan form Definition \ref{def-contrat-theta5758}.
\begin{theorem}\label{lem-Q}
Let $X\subset \mathbb{A}^{\ell}_{\mathbb{C}}$ be a subvariety with linear torus action of complexity one.
Let $\theta  =  (\overline{N}=  \mathbb{Z}^{\ell}, N, F, S, \Sigma)$ be the weight package of $X$ and write $\Delta$ for its lifting fan. Let $(v_{1}, \ldots, v_{r})$ be the list of primitive vectors of $\overline{N}$ generating the rays of the fan $\Delta$. Consider the matrix 
$\mathscr{Q} =  (q_{i,j})\in {\rm Mat}_{\ell\times r}(\mathbb{Z})$ where the columns are the vectors $v_{1}, \ldots, v_{r}$ in the canonical basis of $\overline{N} $. Consider $X$ with presentation $X =  \mathbb{V}(f_{1} =  u_{1}g_{1}, \ldots, f_{a} =  u_{a}g_{a})$ as in Lemma \ref{lem-equa-tvar}. Let $w_{v}$ be a monomial such that
$$h_{v} := w_{v}\cdot f_{v}\left(\prod_{j =  1}^{r}T_{j}^{q_{1, j}}, \ldots, \prod_{j = 1}^{r}T_{j}^{q_{\ell, j}}\right)$$
is homogenous in $\mathbb{C}[T_{1}, \ldots, T_{r}]$ for the ${\rm Cl}(X_{\Delta})$-grading, where $T_{i}$ corresponds 
to ray associated with $v_{i}$ and $1\leq v\leq a$. Assume that the subvariety $X_{1}\subset X_{\Delta}$ defined in homogeneous coordinates by 
$$X_{1} =  \{[z_{1}: \ldots: z_{r}]_{\Delta}\,|\, h_{v}(z_{1}, \ldots, z_{r}) =  0 \text{ for }v=1, \ldots, a\}$$
is irreducible. Then the normalization of $X_{1}$ is $\mathbb{T}_{N}$-isomorphic to the contraction space $\widetilde{X}$ and encoded
by a contraction divisorial fan  of $\{\bar{\mathfrak{D}}_{\theta}\}$.
\end{theorem}
\begin{proof}
By construction $X_{1}\cap \mathbb{G}$ coincides with $X\cap \mathbb{G}$. Therefore $X_{1}$ is the Zariski closure 
of $X\cap \mathbb{G}$ in $X_{\Delta}$, that is $X_{1}$ is the proper transform of $X$ by $\pi_{\Delta}$. We conclude using Theorem \ref{theo-aff-impl}.
\end{proof}
For the projective case one needs an appropriate notion of lifting fan.
\begin{definition}
Let $\theta =  (\overline{N}, N, F, S, \Sigma)$ be a weight package. Let $(\theta^{(0)}, \ldots, \theta^{(\ell)})$ be the 
sequence of weight packages from the enhanced structure $\widehat{\theta}$ of $\theta$ (see Definition \ref{def-enhance}). Each $\theta^{(i)}$ corresponds to a maximal cone $\delta^{(i)}$ of the fan of the projective space $\mathbb{P}_{\mathbb{C}}^{\ell}$ with
here $\delta^{(0)} =  \delta =  \mathbb{Q}_{\geq 0}^{\ell}$. We denote by $\Delta^{(i)}$ the smallest fan of $\overline{N}_{\mathbb{Q}}$ with support $\delta^{(i)}$ such that the map
$$\Delta^{(i)}\rightarrow \Sigma^{(i)},\,\, \delta_{0}\mapsto P(\delta_{0})$$
is a fan morphism, where $P: \overline{N}_{ \mathbb{Q}}\rightarrow {\rm Coker}(F)_{\mathbb{Q}}$ is induced 
by the matrix factorization of $\theta$. The \emph{enhanced lifting fan} $\bar{\Delta}$ of $\theta$ is the fan of $\overline{N}_{\mathbb{Q}}$ generated by $\Delta^{(0)}, \ldots, \Delta^{(\ell)}$. 
\end{definition}
The following describes up to normalization of contraction spaces of  projective
varieties with linear torus action of complexity one. 
\begin{theorem}\label{theo-main-contr-theorDelta}
Let $X\subset \mathbb{P}_{\mathbb{C}}^{\ell}$ be a projective variety with linear torus action of complexity one and let $\theta = (\overline {N}, N, F, S,\Sigma)$ be a weight package of $X$. Let $\bar{\Delta}$ be the enhanced lifting fan of $\theta$ and 
let $(v_{1}, \ldots, v_{r})$ be the list of primitive vectors 
generating the rays of the fan $\bar{\Delta}$. Denote by
$\bar{\mathscr{Q}} =  (\bar{q}_{i,j})\in {\rm Mat}_{\ell \times r}(\mathbb{Z})$ the matrix where the $i$-th columns is $v_{i}$.
 Assume that $X =  \mathbb{V}(f_{1} =  u_{1}g_{1}, \ldots, f_{a} =  u_{a}g_{a})$ arises from the construction of Theorem \ref{theo-main-nonnorm}. Let $w_{v}$ be a monomial such that 
$$h_{v} := w_{v}\cdot f_{v}\left(1,\prod_{j =  1}^{r}T_{j}^{\bar{q}_{1, j}}, \ldots, \prod_{j = 1}^{r}T_{j}^{\bar{q}_{\ell, j}}\right)$$
is homogeneous in $\mathbb{C}[T_{1}, \ldots, T_{r}]$ for the ${\rm Cl}(X_{\bar{\Delta}})$-grading, where $T_{i}$ corresponds to the ray associated with $v_{i}$ and for $1\leq v\leq a$. Assume that the subvariety $X_{1}\subset X_{\bar{\Delta}}$ defined in homogeneous coordinates by 
$$X_{1} =  \{[z_{1}: \ldots: z_{r}]_{\bar{\Delta}}\,|\, h_{v}(z_{1}, \ldots, z_{r}) =  0 \text{ for }v=1, \ldots, a\}$$
is irreducible. Then, the normalization of $X_{1}$ is $\mathbb{T}_{N}$-isomorphic to the contraction space $\widetilde{X}$ and  given
by a contraction divisorial fan of the divisorial fan $\mathscr{E}_{\theta}$ associated to the weight package $\theta$ as in Theorem \ref{theo-main-nonnorm}.
\end{theorem}

\begin{proof}
Same proof as Theorem \ref{lem-Q}. 
\end{proof}
\begin{example}
Consider the cubic $\mathbb{C}^{\star}$-surface $X  =  \mathbb{V}(z_{1}z_{2}^{2} + z_{2}z_{0}^{2}  + z_{3}^{3})\subset \mathbb{P}^{3}_{\mathbb{C}}$ of Section \ref{sec-exex}. Let $(e_{1},e_{2},e_{3})$ be the canonical basis of $\mathbb{Q}^{3}$, set $e =  -e_{1}-e_{2}-e_{3}$,
$v_{0} =  (0,3,1), v_{1} =  (-3, 3, 1), v_{2} =  (3,-3,-1), v_{3} =  (3, -1,-1).$
The enhanced lifting fan is 
\begin{center}
 \begin{tabular}{|c| c |} 
 \hline
 $\Delta^{(0)}_{\rm max}$ &  ${\rm Cone}(e_{1} , e_{2} ,v_{0}), {\rm Cone}(e_{1}, e_{3},v_{0})$ \\ [2ex] 
 \hline
  $\Delta^{(1)}_{\rm max}$ & ${\rm Cone}(e, e_{2}, v_{1}),{\rm Cone}(e_{2}, e_{3},v_{1}),  {\rm Cone}(e, e_{3}, v_{1})$  \\  [2ex] 
 \hline
  $\Delta^{(2)}_{\rm max}$ & ${\rm Cone}(e_{1}, e, v_{2}), {\rm Cone}(e_{1}, e_{3}, v_{2}), {\rm Cone}(e, e_{3}, v_{2})$  \\ [2ex] 
 \hline
  $\Delta^{(3)}_{\rm max}$ &${\rm Cone}(e_{1}, e_{2}, v_{3}), {\rm Cone}(e_{2}, e, v_{3})$  \\ [2ex] 
 \hline
\end{tabular}
\end{center}
Note that $\mathbb{G}_{\bar{\Delta}} =  (\mathbb{C}^{\star})^{5}$ acts on homogeneous coordinate space $\mathbb{A}^{8}_{\mathbb{C}}$ of $X_{\bar{\Delta}}$ via 
$$(\alpha, \beta, \gamma, \lambda, \mu)\cdot (t_{1}, \ldots, t_{8}) =  (\alpha\gamma^{3}\lambda^{-3}\mu^{-3}t_{1}, \alpha\beta^{-3}\gamma^{-3}\lambda ^{3}\mu t_{2}, \alpha \beta^{-1}\gamma^{-1}\lambda \mu t_{3}, \alpha t_{4}, \beta t_{5}, \gamma t_{6}, \lambda t_{7}, \mu t_{8})$$
where the variables $t_{1}, \ldots, t_{8}$ correspond to $e_{1},e_{2}, e_{3}, e, v_{0},\ldots, v_{3}$. 
Finally, the contraction space up to normalization is the variety
$$X_{1} =  \{ [x_{1}, \ldots, x_{8}]_{\bar{\Delta}}\,|\, x_{1}x_{2}^{2}x_{5}^{3}x_{8}^{4} +  x_{2}x_{4}^{2}x_{8}^{2} + x_{3}^{3} = 0\}.$$
\end{example}

\section{Vanishing of odd dimensional intersection cohomology}\label{sec-five}
In this section, we prove Theorem \ref{theo-even}, which
 gives a precise form of the decomposition theorem for 
 contraction maps of torus actions of complexity one. 
\subsection{Combinatorics of $h$-vectors}
\label{sec: h-poly23}We start by recalling notations for intersection cohomology of toric varieties. 
Take any strictly convex polyhedral cone $\sigma\subset N_{\mathbb{Q}}$. Let $N(\sigma)_{\mathbb{Q}}$ be the subspace of $N_{\mathbb{Q}}$ generated by $\sigma$ and denote
by $M(\sigma)_{\mathbb{Q}}$ its dual, seen as quotient of $M_{\mathbb{Q}}$. The dual cone $\omega(\sigma)$ 
of $\sigma$ in $M(\sigma)_{\mathbb{Q}}$ is strictly convex. We define the polytope $Q(\sigma)\subset  M(\sigma)_{\mathbb{Q}}$ as  intersection of $\omega(\sigma)$ with an affine hyperplane of $M(\sigma)_{\mathbb{Q}}$ cutting all one-dimensional faces of $\omega(\sigma)$. Actually, $Q(\sigma)$ is defined up to combinatorial equivalence, meaning that another choice of a hyperplane  gives a polytope with  same poset of faces than $Q(\sigma)$ and 
 same dimension function on each face. We denote by $\Sigma^{\star}(\sigma)$ the normal fan of $N(\sigma)_{\mathbb{Q}}$
of the polytope $Q(\sigma)$.
\\

For a complete fan $\Sigma$ and a strictly convex polyhedral cone $\sigma\subset N_{\mathbb{Q}}$ we define the \emph{$h$-polynomial} $h(\Sigma; t^{2})$ and the \emph{$g$-polynomial} $g(\sigma; t^{2})$
as 
$$h(\Sigma; t^{2}) =  P_{X_{\Sigma}}(t) =  \sum_{j\in\mathbb{Z}}{\rm dim}\, IH^{j}(X_{\Sigma}; \mathbb{Q})t^{j};$$
$$ g(\sigma; t^{2}) =  \sum_{j\in\mathbb{Z}}\mathcal{H}^{j-n}(IC_{X_{\sigma}})_{x}t^{j},$$
where $n =  {\rm dim}\, N_{\mathbb{Q}}$ and $x\in  O(\sigma) \subset X_{\sigma}$. 
The following computes $h(\Sigma; t^{2})$ and $g(\sigma; t^{2})$ via a double induction. 

\begin{theorem}\cite[Theorem 1.1 and Theorem 1.2]{Fie91}, \cite[Section 6]{DL91}, \cite[Section 5]{BBFK02} \label{Theo-hpolgpol}
The polynomial $g(\sigma; t^{2})$ does not depend on $x\in O(\sigma)$, and for $\Sigma$ a complete fan we have
$$h(\Sigma; t^{2}) =  \sum_{\sigma\in\Sigma}(t^{2}-1)^{n-{\rm dim}\, \sigma}g(\sigma; t^{2});$$
$$g(\sigma; t^{2}) = \begin{cases} \tau_{\leq d-1}((1-t^{2})h(\Sigma^{\star}(\sigma); t^{2}))\text{  if } d\geq 3 \\   1 \text{ if } d\leq 2,  \end{cases}$$
where $n =  {\rm dim}\, N_{\mathbb{Q}}$, $d  =  {\rm dim}\, N(\sigma)_{\mathbb{Q}}$, and $\tau_{\leq d-1}$ is the truncation of polynomials to degrees $\leq d-1$. 
\end{theorem}
Theorem \ref{Theo-hpolgpol} implies that $g(\sigma; t^{2})$ and $h(\Sigma; t^{2})$ are even polynomials. We define the \emph{$g$-number} $g_{j}(\sigma)$ via
$g(\sigma; t^{2}) =  \sum_{0\leq j\leq 2n}g_{j}(\sigma)t^{j}.$
\begin{remark}\label{rem-hpolexplicit234} \cite[Section 1, Remark (iii)]{Fie91} Let $\sigma$ be a full-dimensional strictly convex polyhedral cone and let $\Sigma$ be a complete fan of  $N_{\mathbb{Q}}$. Here are some computations of $h(\Sigma; t^{2})$, $g(\sigma; t^{2})$ in small dimensions.
\begin{center}
 \begin{tabular}{ |c|c| c|}
 \hline
  ${\rm dim}\, N_{\mathbb{Q}}$ & $h(\Sigma; t^{2})$  & $g(\sigma; t^{2})$\\ 
 \hline\hline
  2 &  $1  + ( |\Sigma(1)| - 2)t^{2} +  t^{4}$  & 1\\ 
 \hline
  3   & $1  + ( |\Sigma(1)| - 3)t^{2} + (|\Sigma(1)| - 3)t^{4} +  t^{6}$ & $1  + ( |\sigma(1)| - 3)t^{2}$\\   
 \hline
  4  & $\sum_{i = 0}^{4}\sum_{\sigma \in \Sigma(i)}(1 + (|\sigma(1)| - i)t^{2}) (t-1)^{4-i}$ & $1  + ( |\sigma(1)| - 4)t^{2}$\\   
 \hline
\end{tabular}
\end{center}
\end{remark}

Notions of $h$-polynomials extend for divisorial fans. 

\begin{definition}
Let $\sigma\subset N_{\mathbb{Q}}$ be a strictly convex polyhedral cone and let $\Lambda\in {\rm Pol}_{\sigma}(N_{\mathbb{Q}})$. The \emph{Cayley cone} ${\rm Cay}(\Lambda)\subset N_{\mathbb{Q}}\times \mathbb{Q}$  of $\Lambda$ is the cone generated by
$$(\sigma\times\{0\})\cup(\Lambda\times\{1\})\subset N_{\mathbb{Q}}\times \mathbb{Q}.$$
Let $\mathfrak{D}$ be a $\sigma$-polyhedral divisor over a smooth curve $C$. The \emph{$g$-polynomial} $g_{\mathfrak{D}}(t)$ 
of $\mathfrak{D}$ is defined as 
follows. Let $\bar{C}$ be the smooth compactification of $C$ and let $\rho_{g}(\bar{C})$ be its genus. If $C = \bar{C}$, then we set
$$g_{\mathfrak{D}}(t) =  (t^{2} + 2\rho_{g}(\bar{C})t  + 1-a)g(\sigma; t^{2}) + \sum_{z\in {\rm Supp}(\mathfrak{D})}g({\rm Cay}(\mathfrak{D}_{z}); t^{2}),$$
where $a$ is the cardinality of 
${\rm Supp}(\mathfrak{D}):=\{z\in C\,|\, \mathfrak{D}_{z}\neq \sigma\}.$ Otherwise we set 
$$g_{\mathfrak{D}}(t) =  ((2\rho_{g}(\bar{C}) + b -1)t  + 1-a)g(\sigma; t^{2}) + \sum_{z\in {\rm Supp}(\mathfrak{D})}g({\rm Cay}(\mathfrak{D}_{z}); t^{2}),$$
where $b$ is the cardinality of $\bar{C}\setminus C$. Let $\mathscr{E} =  \{\mathfrak{D}^{i}\,|\, i\in I\}$ be
a contraction divisorial fan over $\bar{C}$ of a complete complexity-one $\mathbb{T}$-variety. For any $z\in\bar{C}$, we write $\Sigma_{z}(\mathscr{E})$ for the complete fan generated  
$${\rm Cay}(\mathfrak{D}_{z}^{i})\text{ and } {\rm Cone}((\sigma_{i}\times\{0\})\cup(\sigma_{i}\times\{-1\})) \text{ for any }i\in I,$$ where $\mathfrak{D}^{i}\in\mathscr{E}$ is a $\sigma_{i}$-polyhedral divisor. The \emph{$h$-polynomial} of $\mathscr{E}$
is the polynomial 
$$h_{\mathscr{E}}(t) =  ((1-c)t^{2} + 2\rho_{g}(\bar{C})t + 1-c)h(\Sigma(\mathscr{E});t^{2}) + \sum_{z\in{\rm Supp}(\mathscr{E})}h(\Sigma_{z}(\mathscr{E}); t^{2}),$$
where $c$ is the cardinality of 
$\{z\in \bar{C}\,|\, \mathfrak{D}_{z}^{i}\neq \sigma\text{ for some }i\in I\}$
and $\Sigma(\mathscr{E})$ is the fan generated by the cones $\sigma_{i}$, $i\in I$.
\end{definition}
The following justifies the terminology of $h$-polynomials for divisorial fans. Note that calculation of the Poincar\'e polynomial of the contraction space of a complete complexity-one $\mathbb{T}$-variety can be obtained from results of \cite{CMS08}. 
\begin{theorem}\cite[Theorem 1.1]{AL18}, \cite[Theorem 5.1, Lemma 5.17]{AL21}\label{theo-gpol-dib} If $\mathscr{E}$ is the contraction divisorial fan 
a complete complexity-one $\mathbb{T}$-variety, then the Poincar\'e polynomial $P_{X(\mathscr{E})}(t)$ of $X(\mathscr{E})$ is $h_{\mathscr{E}}(t)$. Furthermore, if $\mathfrak{D}$ is a proper $\sigma$-polyhedral divisor over a smooth curve and $\sigma$ is full-dimensional, then the Poincar\'e polynomial $P_{\widetilde{X}(\mathfrak{D})}(t)$ of the contraction space of $X(\mathfrak{D})$ is  $g_{\mathfrak{D}}(t)$. 
\end{theorem}

\subsection{Topology of the contraction map}\label{sec-top-contractm}
Our goal is to prove the following theorem.
\begin{theorem}\label{theo-even}
Let $X$ be a normal variety with effective torus action of complexity one. Denote by $E$ the image of the exceptional
locus of the contraction map $\pi: \widetilde{X}\rightarrow X$, and let ${\rm Orb}_{\rm even}(E)$ be the set of even codimensional torus orbits  of $X$ contained in $E$. Then we have 
$$\pi_{\star} IC_{\widetilde{X}} \simeq IC_{X}\oplus \bigoplus_{O \in {\rm Orb}_{\rm even}(E)} (\iota_{O})_{\star} IC_{\bar{O}}^{\oplus s_{O}},$$
where $\iota_{O} : \bar{O}\rightarrow X$ is the inclusion and $s_{O}\in \mathbb{Z}_{\geq 0}$ for any $O\in {\rm Orb}_{\rm even}(E)$.
\end{theorem}
Let $X$ be as in Theorem \ref{theo-even}. By Proposition \ref{theo-decomp1} we have
\begin{equation}\label{Eq:decomposition}
\pi_{\star}IC_{\widetilde{X}}\simeq IC_{X}\oplus \bigoplus_{O\in {\rm Orb}(E)}\bigoplus_{b\in\mathbb{Z}}(\iota_{O})_{\star}IC_{\bar{O}}^{\oplus s_{b, O}}[-b],
\end{equation}
where ${\rm Orb}(E)$ is the set of orbits of $E$ and $s_{b, O}\in\mathbb{Z}_{\geq 0}$. Set 
$S_{O}(t) =  S_{X, O}(t):= \sum_{b\in\mathbb{Z}}s_{b, O}t^{b}\in\mathbb{Z}[t, t^{-1}].$
From Identity (\ref{Eq:decomposition}),
$$P_{X}(t) =  P_{\widetilde{X}}(t) -\sum_{O\in{\rm Orb}(E)}\widetilde{S}_{O}(t)P_{\bar{O}}(t), \text{ where }\widetilde{S}_{O}(t) =  \widetilde{S}_{X, O}(t) := S_{O}(t)t^{{\rm dim}\, X -{\rm dim}\, O}$$ for any $O\in {\rm Orb}(E)$.
Thus, we need  prove that $\widetilde{S}_{O}(t) =  \lambda t^{{\rm dim}\, X - {\rm dim}\, O},$ where $\lambda\in\mathbb{Z}_{\geq 0}$ and $\lambda =  0$ provided that ${\rm dim}\, X - {\rm dim}\, O$ is odd. 

\begin{remark}\cite[Sections 4.1, 4.2]{AL21}. \label{remarq-XO}
By Lemma \ref{l-diagram} there exist,  for any $O\in {\rm Orb}(E)$, a $\mathbb{T}$-stable Zariski open subset $X_{O}\subset X$ containing $O$ as a unique closed orbit and a fiber product decomposition 
$X_{O}\simeq \mathbb{T}_{O}\times^{G}X_{1},$ where  $X_{1,O} :=  X_{1}$ is a variety with complexity-one torus action having a unique attractive fixed point $x$ and $G_{O} :=  G\subset \mathbb{T}$ is a finite subgroup. 
\end{remark}

Let $\Gamma$ be the  subgroup generated 
by the $G_{O}$ for  $O\in {\rm Orb}(E)$. We will say that $X$ \emph{satisfies Condition $(\star)$} if all the $G_{O}$ are trivial.
\begin{lemma}\label{lem:quo}
    The quotient $X/\Gamma$ satisfies the Condition $(\star)$.
\end{lemma}
\begin{proof} Let $\Gamma'$ be the image of $\Gamma$ inside the torus quotient $\mathbb{T}_{O}$. Then we indeed have the $\mathbb{T}_{O}/\Gamma'$-equivariant isomorphism $\widetilde{X}_{O}/ \Gamma \simeq \mathbb{T}_{O}/ \Gamma' \times \widetilde{X}_{1}/G_{O}$ for any $O\in {\rm Orb}(E)$. Therefore, the left $\mathbb{T}_{O}/\Gamma'$-action on $\widetilde{X}_{O}$ is free for any $O\in{\rm Orb}(E)$, showing that $X/\Gamma$ satisfies $(\star)$.
\end{proof}


\begin{lemma}\label{LemmaStildeL}
With the notation of Lemma \ref{lem:quo}, we have $\widetilde{S}_{X, O}(t) =  \widetilde{S}_{X/\Gamma, O/\Gamma}(t)$ for any $O\in {\rm Orb}(E)$.
\end{lemma}
\begin{proof}
 Set 
$$\gamma(X, O, O_{2})(t) = \sum_{i\in \mathbb{Z}}{\rm dim}\, \mathbb{H}^{i}(O_{2}, \iota_{O_{2}}^{\star}IC_{\bar{O}})t^{i},$$
where $O,O_{2}\in {\rm Orb}(E)$ with $O\prec O_{2}$  and $\iota: O_{2}\rightarrow X$ is the inclusion. The subset 
$$X(O, O_{2}):= \bigcup_{O\prec O_{1}\prec O_{2}}O_{1}\subset \bar{O}$$
is  Zariski open and contains $O_{2}$ as unique closed orbit. Let $\delta: \widehat{X}(O, O_{2})\rightarrow X(O, O_{2})$ be the normalization. Then the toric variety $\widehat{X}(O, O_{2})$ is affine and $\widehat{O}_{2}:= \delta^{-1}(O_{2})$
is its closed orbit. We have a Cartersian square  
 $$\xymatrix{
      \widehat{O}_{2}\ar[r]^{\beta} \ar[d] ^{\iota_{\widehat{O}_{2}}} & O_{2}\ar[d]^{\iota_{O_{2}}}\\ \widehat{X}(O, O_{2}) \ar[r]^{\delta} & X(O, O_{2}), }$$ 
where $\iota_{O_{2}}$ and $\iota_{\widehat{O}_{2}}$ are inclusions. So by base change \cite[Theorem 2.3.26]{Dim04} and Lemma \ref{LelENORMAphhh},
$$\beta_{\star}\iota_{\widehat{O}_{2}}^{\star} IC_{\widehat{X}(O, O_{2})}\simeq \iota_{O_{2}}^{\star}\delta_{\star} IC_{\widehat{X}(O, O_{2})}\simeq \iota_{O_{2}}^{\star} IC_{X(O, O_{2})}.$$
Passing to hypercohomology gives
$$\gamma(X, O, O_{2})(t)  =   \sum_{i\in \mathbb{Z}}{\rm dim}\, \mathbb{H}^{i}(\widehat{O}_{2}, \iota_{\widehat{O}_{2}}^{\star}IC_{\widehat{X}(O, O_{2})})t^{i}.$$
Note that $\widehat{O}_{2}$ is the fixed point set of a non-hyperbolic $\mathbb{C}^{\star}$-action on $\widehat{X}(O, O_{2})$. Let $e\in \mathbb{Z}_{>0}$ such that the $\mathbb{C}^{\star}$-action on $\widehat{X}(O, O_{2})/\mu_{e}(\mathbb{C})$ is free outside $\widehat{O}_{2}/\mu_{e}(\mathbb{C})$. Using Lemmata \ref{LemmaCstarfixepointLemma},  \ref{GtrivGtrivlemma} and \cite[Theorems 10, 11]{BM99},
$$P_{\widehat{X}(O, O_{2})}(t) = P_{\widehat{X}(O, O_{2})/\mu_{e}(\mathbb{C})}(t) =   t^{{\rm dim}\, O} \gamma(X/\mu_{e}(\mathbb{C}), O/\mu_{e}(\mathbb{C}), O_{2}/\mu_{e}(\mathbb{C}))(t) $$
$$=   t^{{\rm dim}\, O}\gamma(X, O, O_{2})(t).$$
Taking a lattice point in the relative interior of the cone $\sigma_{O_{2}}\subset N_{\mathbb{Q}}$,
where $X_{O_{2}}$ is described by a $\sigma_{O_{2}}$-polyhedral divisor, defines  a non-hyperbolic $\mathbb{C}^{\star}$-action on $X_{O_{2}}$ in which  $O_{2}$ is the fixed point set. So restricting, in the decomposition theorem for the contraction maps
of $X$ and $X/\Gamma$, to $O_{2}$ and $O_{2}/\Gamma$ gives
$$P_{\widetilde{X}_{O_{2}}}(t) - P_{X_{O_{2}}}(t) = \sum_{O\prec O_{2}} P_{\widehat{X}(O, O_{2})}(t) \widetilde{S}_{X, O}(t)$$
and 
$$P_{\widetilde{X}_{O_{2}}/\Gamma}(t) - P_{X_{O_{2}}/\Gamma}(t) = \sum_{O\prec O_{2}} P_{\widehat{X}(O, O_{2})/\Gamma}(t) \widetilde{S}_{X/\Gamma, O/\Gamma}(t).$$
Finally, applying Lemma \ref{GtrivGtrivlemma} for the $\Gamma$-actions on $X_{O_{2}}$, $\widetilde{X}_{O_{2}}$, $\widehat{X}(O, O_{2})$ gives
\begin{equation}\label{Eq:decomposition2}
\sum_{O\prec O_{2}} P_{\widehat{X}(O, O_{2})}(t) \widetilde{S}_{X, O}(t) =   \sum_{O\prec O_{2}} P_{\widehat{X}(O, O_{2})}(t) \widetilde{S}_{X/\Gamma, O/\Gamma}(t).
\end{equation}
We prove the lemma by induction on the codimension of $O_{2}\in {\rm Orb}(E)$. 
\\

\emph{Initial step.} Case ${\rm dim}\, X -  {\rm dim}\, O_{2} = 2$ follows from Identity (\ref{Eq:decomposition2}) since $O\prec O_{2}$ implies 
$O =  O_{2}$.
\\

\emph{Induction step.} Assume that this holds for codimension $d\geq 2$ and consider $O_{2}$ with ${\rm dim}\, X - {\rm dim}\, O_{2} =  d+1$. By induction $\widetilde{S}_{X, O}(t) =  \widetilde{S}_{X/\Gamma, O/\Gamma}(t)$ for any $O\in {\rm Orb}(E)$ such that $O\prec O_{2}$ and $O\neq O_{2}$. Therefore  Identity (\ref{Eq:decomposition2}) implies  $\widetilde{S}_{X, O_{2}}(t) =  \widetilde{S}_{X/\Gamma, O_{2}/\Gamma}(t)$, as required. 
\end{proof}

\begin{definition}
For all $O_{1}, O_{2}\in {\rm Orb}(E)$ write $O_{1}\prec O_{2}$ when $O_{2}\subset \bar{O}_{1}$. Set
$$R_{O_{1}, O_{2}}(t)  =  \sum_{i\in\mathbb{Z}}{\rm dim}\, \mathcal{H}^{i}(IC_{\bar{O}_{1}})_{x_{2}}t^{i}\in\mathbb{Z}[t, t^{-1}],$$
where $x_{2}\in O_{2}$ and $O_{1}\prec O_{2}$. Note that $R_{O_{1}, O_{2}}(t)$ does not depend on the choice of $x_{2}\in O_{2}$ \cite[Remark 5.21]{AL21}.
\end{definition}

\begin{lemma}\label{lemma-RT} Assume that $X$ satisfies Condition $(\star)$ and take the notation of Remark \ref{remarq-XO}. Then
$$P_{\widetilde{X}_{1,O_{2}}}(t)- P_{X_{1, O_{2}}}(t) =  \sum_{O_{1}\prec O_{2}}R_{O_{1}, O_{2}}(t)t^{{\rm dim}\, O_{1}}\widetilde{S}_{O_{1}}(t),$$
where $\widetilde{X}_{1, O_{2}}$ is the contraction space of $X_{1, O_{2}}$. 
\end{lemma}
\begin{proof}
Take the stalks at $x_{2}\in O_{2}$ on both sides of Equation (\ref{Eq:decomposition}), where on the left-hand side we  use that
 $X_{O_{2}}\simeq \mathbb{T}_{O_{2}}\times X_{1, O_{2}}$ and $\widetilde{X}_{O_{2}}\simeq \mathbb{T}_{O_{2}}\times \widetilde{X}_{1, O}$
and Kunneth's formula.
\end{proof}

\begin{proof}[Proof of Theorem \ref{theo-even}] 

Changing $X$ by $X/\Gamma$,  we may assume (see Lemma \ref{LemmaStildeL}) that $X$ satisfies Condition $(\star)$ of Remark \ref{remarq-XO}. We show the result by induction on the dimension of $X$.
\\

\emph{Initial step.} The statement holds for surfaces \cite[Example 5.3]{AL21}.
\\

\emph{Induction step.} Assume that this holds for dimension $\leq d_{0}$. Let $X$ be a normal complexity-one $\mathbb{T}$-variety of dimension $d =  d_{0}+1$.
\\

\emph{Claim: The polynomial $\widetilde{S}_{O_{0}}(t)$ is of degree $\leq d$, where $O_{0} =  \{x\}$ and $x$ is a fixed point of  $E$.}
\begin{proof}[Proof of the Claim]
For a subset $F\subset E$ let ${\rm Orb}(F)$ be the set of orbits
of $F$ and let $E_{x} =  E\cap X_{1, O_{0}}$.
By Lemma \ref{lemma-RT}, 
$$P_{\widetilde{X}_{O_{0}}}(t) -  P_{X_{O_{0}}}(t)  =  \widetilde{S}_{O_{0}}(t) + R(t), \text{ where }
R(t) := \sum_{O\in\ {\rm Orb}(E_{x}\setminus\{x\})}R_{O, O_{0}}(t)t^{{\rm dim}\, O} \widetilde{S}_{O}(t).$$

First,  observe that $P_{\widetilde{X}_{O_{0}}}(t) -  P_{X_{O_{0}}}(t)$ is of degree $d$. Indeed, let $\bar{\mathfrak{D}}_{O_{0}}$ be the $\sigma_{O_{0}}$-polyhedral divisor over a smooth projective curve $\bar{C}$ describing the $\mathbb{T}$-variety $X_{O_{0}}$. By Theorem \ref{theo-gpol-dib},
$$P_{\widetilde{X}_{O_{0}}}(t)  =  g_{\bar{\mathfrak{D}}_{O_{0}}}(t) =  (t^{2} + 2\rho_{g}(\bar{C})t + 1-a)\cdot g(\sigma_{O_{0}}; t^{2}) +  \sum_{z\in{\rm Supp}(\bar{\mathfrak{D}}_{O_{0}})}
g({\rm Cay}(\bar{\mathfrak{D}}_{O_{0}, z}); t^{2}),$$
where $a$ is the cardinality of ${\rm Supp}(\bar{\mathfrak{D}}_{O_{0}})$. By Theorem \ref{Theo-hpolgpol}, the polynomials $$g(\sigma_{O_{0}}; t^{2})\text{ and }g({\rm Cay}(\bar{\mathfrak{D}}_{O_{0}, z}); t^{2})\text{ for }z\in  {\rm Supp}(\bar{\mathfrak{D}}_{O_{0}})$$ are of degrees $\leq d-2$ and $\leq d-1$. Hence 
$P_{\widetilde{X}_{O_{0}}}(t) $ is of degree $\leq d$. Also, since $X_{O_{0}}$
has a unique attractive fixed point, $X_{O_{0}}$
is, after taking the quotient by a finite subgroup of $\mathbb{T}$,  an affine cone over a projective variety $V$. So
$$P_{X_{O_{0}}}(t)  =  \tau_{\leq d-1}((1-t^{2})P_{V}(t)),$$
where $\tau_{\leq d-1}$ is the truncation to degrees $\leq d-1$ (see \cite[Lemma 2.1]{Fie91}, \cite[Proposition 5.6]{AL21}), and $P_{\widetilde{X}_{O_{0}}}(t) -  P_{X_{O_{0}}}(t)  $ is of degree $\leq d$.

Next, we estimate the degree of $R(t)$. Let $O\in\mathcal{O}(E_{x}\setminus\{x\})$ and let $\pi_{O}: \widetilde{X}_{O}\rightarrow X_{O}$ be the contraction map. Note 
that, due to Condition $(\star)$,  $$X_{O}\simeq \mathbb{T}_{O}\times X_{1, O}\text{ and }\widetilde{X}_{O}\simeq \mathbb{T}_{O}\times \widetilde{X}_{1, O}.$$  Let $y\in X_{1, O}$ be the unique attractive fixed point and set $O' =  \{y\}$. As $\pi_{O}$ is obtained from the product 
of $\mathbb{T}_{O}$ and the contraction map of $X_{1, O}$, 
$$\widetilde{S}_{X, O}(t) =  \widetilde{S}_{ X_{1, O}, O'}(t).$$
Now remark that ${\rm dim}\, O\geq 1$ implies ${\rm dim} X_{1, O} <d $. Consequently,
by induction, there exists $\lambda_{O}\in\mathbb{Z}_{\geq 0}$ such that 
$$\widetilde{S}_{X, O}(t) =  \lambda_{O}t^{d- {\rm dim}\, O}$$
with $\lambda_{O} = 0$ if $d - {\rm dim}\, O$ is odd. 
In addition, using Theorem \ref{Theo-hpolgpol} and \cite[Lemma 5.20]{AL21}, the degree of $R_{O, O_{0}}(t)t^{{\rm dim}\, O}$ is $< {\rm dim}\, O$. Hence the degree of $R(t)$ is $< d$.  Thus the degree of
$$\widetilde{S}_{O_{0}}(t) =  P_{\widetilde{X}_{O_{0}}}(t) -  P_{X_{O_{0}}}(t)-R(t)$$
is $\leq d$. 
This proves the claim. 
\end{proof}
By induction,
$$Q(t):= \sum_{O_{0}\in {\rm Orb}(E),\, {\rm dim}(O_{0}) =  0}\widetilde{S}_{O_{0}}(t)  =  P_{\widetilde{X}}(t) - P_{X}(t) -  \sum_{O\in {\rm Orb}(E\setminus X^{\mathbb{T}})}\widetilde{S}_{X, O}(t)P_{\bar{O}}(t) $$
$$ =  P_{\widetilde{X}}(t) - P_{X}(t) -  \sum_{O\in{\rm Orb}(E\setminus X^{\mathbb{T}})}\widetilde{S}_{X_{1, O}, O'}(t)P_{\bar{O}}(t) $$
$$=  P_{\widetilde{X}}(t) - P_{X}(t) - \sum_{O\in{\rm Orb}(E\setminus X^{\mathbb{T}})}\lambda_{O}t^{d - {\rm dim}\, O}P_{\bar{O}}(t)$$
is Poincar\'e symmetric, i.e. $Q(t) =  t^{2d}Q(1/t)$.
The claim implies
that $Q(t)$ is of degree $\leq d$. So  for any $O_{0}\in \mathcal{O}(E)$ with ${\rm dim}\, O_{0} = 0$, there exists   $\lambda_{O_{0}}\in \mathbb{Z}_{\geq 0}$ such that $\widetilde{S}_{O_{0}}(t) =  \lambda_{O_{0}}t^{d}$. Since from the proof of the claim the number $\lambda_{O_{0}}$ is the $(d - 2)$-th coefficient 
of $g(\sigma_{O_{0}}; t^{2})$, we have $\lambda_{O_{0}} =  0$ when $d$ is odd, proving the theorem. 
\end{proof}
\subsection{Betti numbers via divisorial fans}\label{Beettidivfanop}
We recall how to describe the image of the exceptional locus of the contraction map of torus actions of 
complexity one in terms of divisorial fans. Let  $\mathscr{E} =  \{\mathfrak{D}^{i}\, |\, i\in I\}$ be a divisorial fan over a smooth curve $Y$, where $\mathfrak{D}^{i}$ is a  $\sigma_{i}$-polyhedral divisor for $i\in I$, and let $\Sigma(\mathscr{E})$ be the fan
generated by the $\sigma_{i}$. We call \emph{degree} of $\mathscr{E}$ the set
$${\rm deg}(\mathscr{E}) =  \bigcup_{i\in I} {\rm deg}(\mathfrak{D}^{i})\subset N_{\mathbb{Q}}.$$
Let $HF(\mathscr{E}) =  \{\tau\in \Sigma(\mathscr{E})\,|\, {\rm deg}(\mathscr{E}) \cap \tau\neq \emptyset\}$ and let $E$ be the image of the exceptional locus of the contraction 
map of $X =  X(\mathscr{E})$. By \cite[Sections 3 and 4]{Tim97}
we have a bijection 
$$HF(\mathscr{E})\rightarrow {\rm Orb}(E),\,\, \tau\mapsto O_{\tau}$$
between $HF(\mathscr{E})$ and the set of orbits of $E$. 
\\

The correspondence $\tau\mapsto O_{\tau}$ is seen as follows. 
Let $\mathfrak{D}\in \mathscr{E}$ be a $\sigma$-polyhedral divisor and let $\tau\in HF(\mathscr{E})$ such that $\tau\cap {\rm deg}(\mathfrak{D})\neq \emptyset$. Then the vanishing of the  ideal 
$$I_{\tau, \mathfrak{D}}= \bigoplus_{m\in \sigma^{\vee}\cap M\setminus \tau^{\perp}}H^{0}(Y, \mathcal{O}_{Y}(\mathfrak{D}(m)))\chi^{m}\subset A[Y, \mathfrak{D}]$$
is  the orbit closure $O_{\tau}$ in $X(\mathfrak{D})$. In particular, the fan
$${\rm Star}(\mathscr{E}, \tau):= \{q_{\tau}(\gamma_{0})\,|\, \gamma_{0}\in HF(\mathscr{E})\text{ and }\tau\text{ is a face of }\gamma_{0}\}$$
describes the normalization of the orbit closure $O_{\tau}$ in $X$, where $N(\tau)\subset N$ is the sublattice generated by $\tau\cap N$ and $q_{\tau}: N_{\mathbb{Q}}\rightarrow (N/N(\tau))_{\mathbb{Q}}$ is the canonical surjection.   
\begin{remark}\label{remark-SigmaO1s}
By Lemma \ref{l-diagram} each orbit $O\in{\rm Orb}(E)$ gives rise to an open
set $X_{O} =  \mathbb{T}_{O}\times^{G}X_{1}\subset X$.
Let $\bar{\mathfrak{D}}_{O}$ be the $\sigma_{O}$-polyhedral divisor describing $X_{1} = X_{1, O}$
Then the map $O\mapsto \sigma_{O}$ is the inverse of $HF(\mathscr{E})\rightarrow {\rm Orb}(E)$, $\tau\mapsto O_{\tau}$. 
\end{remark}

As a consequence of Theorem \ref{theo-even}, one can give a precise formula of the intersection cohomology Betti numbers of any complete normal
variety with torus action of complexity one. 
\begin{corollary}\label{coro-Betti-divfan}
Let $X =  X(\mathscr{E})$ be a complete normal $\mathbb{T}$-variety of complexity one with defining divisorial fan $\mathscr{E}$. 
Denote by $\widetilde{\mathscr{E}}$ the divisorial fan of the contraction space $\widetilde{X}$ of $X$. 
For $\tau\in HF(\mathscr{E})$, set $n(\tau):= {\rm dim}\, \tau -1$ and $c(\tau):= {\rm dim}\, \tau +1$. 
Then the 
Poincar\'e polynomial of $X$ is given by the formula
$$P_{X}(t)  =  h_{\widetilde{\mathscr{E}}}(t) -  \sum_{\tau\in HF(\mathscr{E})}g_{n(\tau)}(\tau)t^{c(\tau)}h({\rm Star}(\mathscr{E}, \tau); t^{2}).$$
\end{corollary}
\begin{proof}
Regarding the proof of Theorem \ref{theo-even}, we observe that $\widetilde{S}_{O}(t) =  \lambda_{O}t^{{\rm dim}\, X - {\rm dim}\, O},$ where $\lambda_{O}$ is leading coefficient of the polynomial  $P_{\widetilde{X}_{1, O}}(t) -  P_{X_{1, O}}(t)$. This latter, is according to Theorem \ref{Theo-hpolgpol} and \cite[Proposition 5.6]{AL21},
the number $g_{{\rm dim}\, \sigma_{O}-1}(\sigma_{O})$, whence the result. 
\end{proof}
Another consequence is the affine case with a unique attractive fixed point. If $\mathscr{E}$ is the divisorial fan $\{\mathfrak{D}\}$, then
we will respectively write  $HF(\mathfrak{D})$ and ${\rm Star}(\mathfrak{D}, \tau)$  instead of $HF(\mathscr{E})$ and ${\rm Star}(\mathscr{E}, \tau)$. 
\begin{corollary}\label{corcor-gD-Ic}
Let $\mathfrak{D}$ be a proper $\sigma$-polyhedral divisor over a smooth projective curve. Assume that
$X  =  X(\mathfrak{D})$ has a unique attractive fixed point, i.e. $\sigma\subset N_{\mathbb{Q}}$ is full-dimensional.
For $\tau\in HF(\mathfrak{D})$, set  $n(\tau):= {\rm dim}\, \tau -1$ and $c(\tau):= {\rm dim}\,\tau +1$. 
Then the 
Poincar\'e polynomial of $X$ is given by the formula
$$P_{X}(t) =  g_{\mathfrak{D}}(t) -  \sum_{\tau\in HF(\mathfrak{D})}g_{n(\tau)}(\tau)t^{c(\tau)}g({\rm Star}(\mathfrak{D}, \tau); t^{2}).$$
\end{corollary}
\begin{proof}
Same proof as Corollary \ref{coro-Betti-divfan}.
\end{proof}

\begin{example}
With the notation of Corollary \ref{coro-Betti-divfan}, note that if $\Sigma(\mathscr{E})$ is simplicial, then 
$$\begin{cases} g_{n(\tau)}(\tau) =  0\text{  if } {\rm dim}(\tau)\neq 1\\   g_{n(\tau)}(\tau) =  1 \text{ if } {\rm dim}(\tau)= 1  \end{cases}$$
for any $\tau\in HF(\mathscr{E})$. 
By Theorem \ref{theo-even} we have 
$$\pi_{\star}IC_{\widetilde{X}}\simeq IC_{X}\oplus \bigoplus_{O\in{\rm Orb}_{2}(E)} (\iota_{O})_{\star}IC_{\bar{O}},$$
where ${\rm Orb}_{2}(E)$ is the set of codimension-two orbits of $X$ in $E$. In the case where $\widetilde{X}$ is rationally smooth, the
fan $\Sigma(\mathscr{E})$ is simplicial. Therefore, we recover \cite[Theorem 1.1(iii)]{AL21}, which was proven via the decomposition theorem for semi-small maps. 
\end{example}

\subsection{Vanishing of odd dimensional intersection cohomology}
Theorem \ref{theo-even} implies the following  rationality criterion.
\begin{theorem}\label{theo-rat-imp-coho}
Let $X$ be any complete variety with torus action of complexity one. Then the following 
are equivalent.
\begin{itemize}
\item[(i)] $X$ is a rational variety.
\item[(ii)] We have $IH^{2j+1}(X; \mathbb{Q}) =  0$ for any $j\in\mathbb{Z}$. 
\end{itemize}
\end{theorem}
\begin{proof}
By Theorem \ref{theo-even} we have 
$$P_{X}(t) =  P_{\widetilde{X}}(t) -  \sum_{O\in{\rm Orb}(E)}\widetilde{S}_{O}(t)P_{\bar{O}}(t)$$
with $\widetilde{S}_{O}(t), P_{\bar{O}}(t)\in\mathbb{Z}[t^{2}]$ for any $O\in{\rm Orb}(E)$, and where $\widetilde{X}$ is the contraction space of $X$. Hence $P_{X}(t)\in\mathbb{Z}[t^{2}]$ if and only if $P_{\widetilde{X}}(t)\in\mathbb{Z}[t^{2}]$, and
by Theorem \ref{theo-gpol-dib}, this is equivalent to that the genus of the smooth projective curve $\bar{C}$ such that $\mathbb{C}(X)^{\mathbb{T}}\simeq \mathbb{C}(\bar{C})$ is $0$. Since $X$ is birationally equivalent to $\bar{C}\times \mathbb{P}^{n}$ \cite[Section 1, Corollary 3]{Tim08}, we conclude by L\"uroth's theorem. 
\end{proof}
\section{Computing intersection cohomology}\label{sec-six}
We now relate the computation of  intersection cohomology with linear torus action
of complexity one from the defining equations and treat the case of trinomial hypersurfaces.
\subsection{First consequences}\label{sec-first-cons}
We start with the affine case having a unique attractive fixed point.
\begin{corollary}[Consequence of Theorem \ref{theo-gpol-dib}, Theorem \ref{theo-aff-impl} and Corollary \ref{corcor-gD-Ic}]\label{cor-compute-ginv} Let $\theta =  (\overline{N}, N, F, S, \Sigma)$ be a weight package and let 
$$X  =  \mathbb{V}(u_{1} f_{1}, \ldots, u_{a} f_{a})\subset \mathbb{A}^{\ell}_{\mathbb{C}}$$
be the subvariety with torus action of complexity one  arising from $\theta$ via Construction \ref{lem-equa-tvar}. Assume that
$\sigma_{\theta}:= S(F(N_{\mathbb{Q}})\cap \delta)$ is full-dimensional, where $\delta =  \mathbb{Q}_{\geq 0}^{\ell}\subset \overline{N}_{\mathbb{Q}} =  \mathbb{Q}^{\ell}$. Then the following hold.   
\begin{itemize}
\item[$(i)$] If the associated curve $C_{\theta}$ of $\theta$ is projective, then the intersection cohomology Betti numbers
of $X$ are described by the formula
$$P_{X}(t) =  g_{\bar{\mathfrak{D}}_{\theta}}(t) -  \sum_{\tau\in HF(\bar{\mathfrak{D}}_{\theta})} g_{n(\tau)}(\tau) t^{c(\tau)} g({\rm Star}(\bar{\mathfrak{D}}_{\theta}, \tau); t^{2}),$$
where $n(\tau) =  {\rm dim}\, \tau - 1$ and $c(\tau)  =  {\rm dim}\, \tau +1$ for any $\tau \in HF(\bar{\mathfrak{D}}_{\theta})$.
\item[$(ii)$] Assume that the associated curve $C_{\theta}$ is projective. Let
$$X_{1} =  \{[z_{1}: \ldots: z_{r}]_{\Delta}\in X_{\Delta}\,|\, h_{v}(z_{1}, \ldots, z_{r}) =  0  \text{ for } v =  1, \ldots, a\}$$
be defined as in Theorem \ref{lem-Q}, where $\Delta$ is the lifting fan of $\theta$. Then the intersection cohomology Betti numbers of $X_{1}$ and of the
contraction space $\widetilde{X}$ of $X$ is given by the formula
$$P_{\widetilde{X}}(t) =  P_{X_{1}}(t)  =  g_{\bar{\mathfrak{D}}_{\theta}}(t).$$
\item[$(iii)$] Assume that $C_{\theta}$ is affine. Then $P_{X}(t) =  g_{\bar{\mathfrak{D}}_{\theta}}(t).$
\end{itemize}
\end{corollary}
\begin{proof}
   Since intersection cohomology Betti numbers are invariant under normalization, to prove the corollary it suffices to apply Theorem~\ref{theo-aff-impl}.\\

\noindent\it{Proof of (i).}
Since $C_{\theta}$ is projective, the claim follows immediately from Corollary~\ref{corcor-gD-Ic}.\\

\noindent\it{Proof of (ii) and (iii).}
By construction, the variety $X_{1}\subset X_{\Delta}$ is a complete complexity-one $\mathbb{T}$-variety whose contraction space coincides with the contraction space $\widetilde{X}$ of $X$.  Hence, (ii) and (iii) follow from Theorem~\ref{theo-gpol-dib}.
\end{proof}

Next we pass to the projective case.
\begin{corollary}[Consequence of Theorem \ref{theo-main-nonnorm} and Corollary \ref{coro-Betti-divfan}]\label{cor-proj2222}
Let $\theta =  (\overline{N}, N, F, S, \Sigma)$ be a weight package and let 
$$X =  \mathbb{V}(u_{1} f_{1}, \ldots, u_{a} f_{a})\subset \mathbb{P}_{\mathbb{C}}^{\ell}$$
be the subvariety with torus action of complexity one obtained from $\theta$ via Construction \ref{theo-main-nonnorm}. Then the following hold. 
\begin{itemize}
\item[$(i)$] The intersection cohomology Betti numbers of $X$ are described by the formula
$$P_{X}(t)  =  h_{\mathscr{E}_{\theta}}(t) -  \sum_{\tau\in HF(\mathscr{E}_{\theta})} g_{n(\tau)}(\tau) t^{c(\tau)} g({\rm Star}(\mathscr{E}_{\theta}, \tau); t^{2}),$$
where $\mathscr{E}_{\theta}$ is the divisorial fan of Theorem \ref{theo-main-nonnorm}, $n(\tau) =  {\rm dim}\, \tau - 1$ and $c(\tau)  =  {\rm dim}\, \tau +1$ for any $\tau \in HF(\mathscr{E}_{\theta})$.
\item[$(ii)$] Let $X_{1}$ be the subvariety 
$$\{[z_{1}:\ldots: z_{r}]_{\bar{\Delta}}\in X_{\bar{\Delta}}\, |\, h_{v}(z_{1}, \ldots, z_{r}) =  0 \text{ for } v =  1, \ldots, a\}$$
arising in Construction \ref{theo-main-contr-theorDelta}, where $\bar{\Delta}$ is the enhanced lifting fan of $\theta$. Then the intersection cohomology Betti numbers of $X_{1}$ and of contraction space $\widetilde{X}$ of $X$  are given by the formula
$$ P_{\widetilde{X}}(t) =  P_{X_{1}}(t)  =  h_{\mathscr{E}_{\theta}}(t).$$
\end{itemize}
\end{corollary}
\subsection{Intersection cohomology of affine trinomial hypersurfaces}\label{sec-trinomial-aff}
By \emph{affine trinomial hypersurface} we mean a hypersurface
$$X  =  \mathbb{V}(T_{1}^{\underline{n}_{1}} + T_{2}^{\underline{n}_{2}} + T_{3}^{\underline{n}_{3}}) \subset \mathbb{A}_{\mathbb{C}}^{\ell}
\text{ such that  }T_{i}^{\underline{n}_{i}} =  \prod_{j = 1}^{r_{i}} T_{i, j}^{n_{i,j}} \text{ for } i = 1,2,3,$$
where $r_{i}, n_{i,j}\in\mathbb{Z}_{> 0}$ and $\ell =  r_{1} + r_{2} + r_{3}$. Here the variables $T_{i,j}$ are independent. 
Note that trinomial hypersurfaces have singularity loci of codimension $\geq 2$ by Jacobian criterion. Hence they are normal.
Set
$$  R = \begin{pmatrix}
 -n_{1,1}& \ldots &  -n_{1, r_{1}} &  n_{2,1} &\ldots & n_{2,r_{2}} & 0& \ldots & 0\\
 -n_{1,1}& \ldots  &  -n_{1, r_{1}} &  0 &\ldots & 0 & n_{3,1} & \ldots & n_{3, r_{3}} 
\end{pmatrix}.$$
Write 
$$ u_{i}:= {\rm gcd}(n_{i,1}, \ldots, n_{i, r_{i}}) \text{ for } i = 1,2,3, \,\, d =  {\rm gcd}(u_{1}, u_{2}, u_{3}),$$
$$ d_{1}  =  {\rm gcd}(u_{2}/d, u_{3}/d), d_{2}  =  {\rm gcd}(u_{1}/d, u_{3}/d) \text{ and } d_{3}  =  {\rm gcd}(u_{1}/d, u_{2}/d).$$
The matrix $R$ induces an exact sequence 
$$0\rightarrow N \xrightarrow{F} \overline{N} =  \mathbb{Z}^{\ell}\xrightarrow{R} {\rm Coker}(F)\rightarrow 0.$$
Here we identify ${\rm Coker}(F)$ with the image of $R$, that is the sublattice of $\mathbb{Z}^{2}$ generated
by $(-u_{1}, -u_{1})$, $(u_{2}, 0),$ $(0, u_{3})$. Let $\Sigma$ be the fan generated by the cones $R(\delta_{0})$, where
$\delta_{0}$ runs over the faces of the first quadrant $\delta =  \mathbb{Q}_{\geq 0}^{\ell}.$ Then 
 (see \cite[Section 4]{Kru19}) $X_{\Sigma}$ is the weighted projective plane $\mathbb{P}(d_{1}, d_{2}, d_{3})$. Let $w_{1}, w_{2}, w_{3}$ be the homogeneous coordinates of $\mathbb{P}(d_{1}, d_{2}, d_{3})$ and consider the smooth curve
$$ C_{d_{1}, d_{2}, d_{3}} := \mathbb{V}( w_{1}^{u/d_{1}} + w_{2}^{u/d_{2}} + w_{3}^{u/d_{3}})\subset \mathbb{P}(d_{1}, d_{2}, d_{3}), $$
where $u =  d d_{1} d_{2} d_{3}$. Furthermore, if $S: \overline{N}\rightarrow N$ is a section of $F$, then the weight package 
$\theta =  (\overline{N}, N, F, S, \Sigma)$ defines a linear torus action of complexity one on $X \subset \mathbb{A}_{\mathbb{C}}^{\ell}$. Kruglov proved the following.
\begin{theorem}\cite[Theorem 3.1]{Kru19}\label{Kruglov-Theorem}
Let $$X  =  \mathbb{V}(T_{1}^{\underline{n}_{1}} + T_{2}^{\underline{n}_{2}} + T_{3}^{\underline{n}_{3}}) \subset \mathbb{A}_{\mathbb{C}}^{\ell}$$
be an affine trinomial hypersurface with its natural weight package $\theta =  (\overline{N}, N, F, S, \Sigma)$.
Then $\bar{\mathfrak{D}}_{\theta}$ is defined over the curve
$C_{d_{1}, d_{2}, d_{3}}\subset \mathbb{P}(d_{1}, d_{2}, d_{3}) =  X_{\Sigma}$ and given by the relation
$$\bar{\mathfrak{D}}_{\theta} =  \bar{\mathfrak{D}}_{\theta, 1}\cdot E_{1} + \bar{\mathfrak{D}}_{\theta, 2}\cdot E_{2} + \bar{\mathfrak{D}}_{\theta, 3}\cdot E_{3}  , \text{ where: }$$
\begin{itemize}
\item[$(i)$] The polyhedron $\bar{\mathfrak{D}}_{\theta, i}$ for $1\leq i\leq 3$ is the Minkowski sum of 
$\sigma_{\theta}:= S(\mathbb{Q}_{\geq 0}^{\ell}\cap F(N_{\mathbb{Q}}))$ and the convex hull of the set
$$\left\{ S\left( \frac{d}{d_{j} n_{i,j}}e_{i,j}\right)\,|\, 1\leq j\leq r_{i}\right\}.$$
Here 
$$(e_{1,1}, \ldots, e_{1, r_{1}}, e_{2, 1}, \ldots, e_{2, r_{2}}, e_{3,1}, \ldots, e_{3, r_{3}})$$
is the canonical basis of $\overline{N} =  \mathbb{Z}^{\ell}.$
\item[$(ii)$] The divisors $E_{1}, E_{2}, E_{3}$ are defined by
$$E_{i}:= \iota^{\star}(H_{i})  =  \sum_{z\in C_{d_{1}, d_{2}, d_{3}}}(H_{i}, C_{d_{1}, d_{2}, d_{3}})_{z}\cdot [z],$$
where $H_{i} =  \mathbb{V}(w_{i})\subset \mathbb{P}(d_{1}, d_{2}, d_{3})$,
$\iota:  C_{d_{1}, d_{2}, d_{3}}\rightarrow  \mathbb{P}(d_{1}, d_{2}, d_{3})$ is the inclusion, and $(H_{i}, C_{d_{1}, d_{2}, d_{3}})_{z}$ is the local intersection number between $H_{i}$ and $C_{d_{1}, d_{2}, d_{3}}$ at $z$. 
\end{itemize}
\end{theorem}
\begin{remark} \cite[Remark 4.1]{Kru19}
Let $\zeta\in \mu_{2u}(\mathbb{C})$ be a primitive root such that $\zeta^{u} = -1$. Then 
$$E_{1} =  \sum_{i = 0}^{dd_{1}-1} [\varsigma([0:1: \zeta \eta_{1}^{i}])],E_{2} =  \sum_{i = 0}^{dd_{2}-1} [\varsigma([1:0: \zeta \eta_{2}^{i}])],E_{3} =  \sum_{i = 0}^{dd_{3}-1} [\varsigma([1: \zeta \eta_{1}^{i}: 0])],$$
where $\eta_{j}\in\mu_{dd_j}(\mathbb{C})$ is a primitive root for $j= 1,2,3$ and $\varsigma: \mathbb{P}^{2}_{\mathbb{C}}\rightarrow \mathbb{P}(d_{1}, d_{2}, d_{3})$ is the quotient.
\end{remark}
\begin{remark} \cite[Remark 4.2]{Kru19}
The genus of the curve 
$C_{d_{1},d_{2}, d_{3}}$ is $$\frac{d}{2}(u - (d_{1} + d_{2} + d_{3})) + 1.$$
\end{remark}
The following corollary describes the intersection cohomology of affine trinomial hypersurfaces. We formulate
the result in a way that we can directly compute from the defining equation.
\begin{corollary}\label{cor-trinome}
Let $$X  =  \mathbb{V}(T_{1}^{\underline{n}_{1}} + T_{2}^{\underline{n}_{2}} + T_{3}^{\underline{n}_{3}}) \subset \mathbb{A}_{\mathbb{C}}^{\ell}$$
be an affine trinomial hypersurface with its natural weight package $\theta =  (\overline{N}, N, F, S, \Sigma)$.
Set $\sigma_{\theta} = S(\mathbb{Q}_{\geq 0}^{\ell}\cap F(N_{\mathbb{Q}}))$, $\gamma =  d(d_{1} + d_{2} + d_{3})$
and $$\Pi_{i} := {\rm Cone}\left( (\sigma_{\theta}\times \{0\})\cup \left(\left\{ S\left( \frac{d}{d_{j} n_{i,j}}e_{i,j}\right)\,|\, 1\leq j\leq r_{i}\right\}\times\{1\}\right)\right)$$
for $i = 1, 2, 3$. Then the intersection cohomology Betti numbers of the contraction space $\widetilde{X}$ of $X$ is given by the formula
$$P_{\widetilde{X}}(t) =  \left(t^{2} + (du - \gamma + 2)t - \gamma + 1\right)\cdot g(\sigma_{\theta}; t^{2}) +  \sum_{i = 1}^{3} dd_{i}\cdot g(\Pi_{i}; t^{2}).$$   
Write 
$$H(\theta, \underline{n}_{1},\underline{n}_{2},\underline{n}_{3}):= \left\{\tau \text{ face of }\sigma_{\theta}\,|\, \tau\cap \left\{\sum_{i =  1}^{3} S\left( \frac{d}{d_{j} n_{i,j_{i}}}e_{i,j_{i}}\right)\,|\, (j_{1}, j_{2}, j_{3})\in \prod_{i = 1}^{3}\{1, \ldots, r_{i}\}\right\} \neq \emptyset\right\}.$$
Then the intersection cohomology Betti numbers of $X$ are given by the formula
$$P_{X}(t)  = P_{\widetilde{X}}(t)  -  \sum_{\tau \in H(\theta, \underline{n}_{1},\underline{n}_{2},\underline{n}_{3})} g_{n(\tau)}(\tau) t^{c(\tau)}, $$
where $n(\tau) =  {\rm dim}\, \tau - 1$ and $c(\tau) =  {\rm dim}\, \tau +1$.    
\end{corollary}
\begin{proof}
Formulae for $P_{\widetilde{X}}(t)$ and $P_{X}(t)$ are consequences of Corollary \ref{cor-compute-ginv} and of Theorem \ref{Kruglov-Theorem} by observing that orbit closures of the image of the exceptional locus of the contraction map are affine spaces.
\end{proof}
\begin{example}
Consider the affine trinomial hypersurface
$$X =  \mathbb{V}(T_{1,1}T_{1,2} + T_{2,1} T_{2,2}+ T_{3,1}^{2})\subset \mathbb{A}_{\mathbb{C}}^{5}$$

with weight matrix and  section
$$F  =  {}^{t}\begin{pmatrix}
 -1 &   1 & 0 & 0 & 0 \\
 2 &  0 &  2 & 0& 1 \\
 0  &  0 &  -1& 1& 0    
\end{pmatrix} \text{ and } S  =  \begin{pmatrix}
 0 &   1 & 0 & 0 & 0 \\
 0 &  0 &  0 & 0& 1 \\
 0  &  0 &  0 & 1& 0    
\end{pmatrix}.$$
 Set 
$v_{1} =  (2,1, 0), v_{2} =  (2,1,2), v_{3} =  (0,1,2), v_{4} =  (0,1,0)$, 
$b_{0} =  (0,0,0,1), \, b_{1} =  (1,0,0,1),\, b_{2} =  (0,0,1,1),\,\, b_{3} =  (0, 1/2, 0,1).$  The table
\\

\begin{center}
 \begin{tabular}{ |c|c|}
 \hline
  &  Generators \\ 
 \hline\hline
 $\sigma_{\theta}$ & $v_{1}, v_{2}, v_{3}, v_{4}$ \\ 
 \hline
  $\Pi_{1}$ & $\sigma_{\theta}\times\{0\}, b_{0} , b_{1}$  \\   
 \hline
  $\Pi_{2}$ & $\sigma_{\theta}\times\{0\}, b_{0}, b_{2}$  \\ 
 \hline
  $\Pi_{3}$ & $\sigma_{\theta}\times\{0\} , b_{3}$  \\
 \hline
\end{tabular}
\end{center}
 implies that
$g(\sigma_{\theta};t^{2}) = g(\Pi_{3}; t^2) =  1 + t^{2}$, and  $g(\Pi_{1}; t^{2}) = g(\Pi_{2}; t^{2}) =  1 + 2t^{2}$. So
$$ P_{\widetilde{X}}(t)  =  (t^{2} - 2) g(\sigma_{\theta};t^{2}) +  \sum_{i =  1}^{3}g(\Pi_{i}; t^{2})  =  t^{4} + 4 t^{2}  + 1.$$
Moreover,
$H(\theta, \underline{n}_{1}, \underline{n}_{2},\underline{n}_{3}) =  \{\text{faces of } \sigma_{\theta}\}.$
Denote by $\rho_{1}, \ldots, \rho_{4}$ the rays of $\sigma_{\theta}$.
Then $c(\rho_{i})  = 2$, $c(\sigma_{\theta})  =  4$,
$g_{n(\rho_{i})}(\rho_{i}) =  g_{n(\sigma_{\theta})}(\sigma_{\theta})  =1 $  and from Corollary \ref{cor-trinome},
$$ P_{X}(t)  =  P_{\widetilde{X}}(t) -  g_{n(\sigma_{\theta})}t^{c(\sigma_{\theta})} - \sum_{i = 1}^{4}g_{n(\rho_{i})}(\rho_{i})t^{c(\rho_{i})} = 1,$$
which is expected since $X$ is an affine cone over a smooth quadric.
\end{example}
\subsection{Intersection cohomology of relevant projective trinomial hypersurfaces}
By \emph{projective trinomial hypersurface} we mean a hypersurface
$$X  =  \mathbb{V}(T_{1}^{\underline{n}_{1}} + T_{2}^{\underline{n}_{2}} + T_{3}^{\underline{n}_{3}})\subset \mathbb{P}^{\ell}_{\mathbb{C}}\text{ such that }T_{i}^{\underline{n}_{i}} =  \prod_{1\leq j\leq r_{i}}T_{i, j}^{n_{i, j}} \text{ for }i = 1,2,3,$$
where $n_{i,j}, r_{i}\in\mathbb{Z}_{>0}$ and $\ell + 1 =  r_{1} + r_{2} + r_{3}$. The defining 
equation of $X$ must be homogeneous, so 
$s  =  \sum_{1\leq j\leq r_{i}} n_{i, j}$
does not depend on $i$. We introduce the following technical condition.
\begin{definition}
The projective trinomial hypersurface $X$ is \emph{relevant}
if each monomial $T_{i}^{\underline{n}_{i}}$ has at least two distinct variables.
\end{definition}

From now on we assume $X$ relevant\footnote{The case of \emph{irrelevant} projective trinomial hypersurfaces will be carried out in the article \cite{ABL}.}. Let  
$$X^{(i,j)} =  \mathbb{V}((T^{\underline{n}_{1}}_{1} + T^{\underline{n}_{2}}_{2} + T^{\underline{n}_{3}}_{3})_{|T_{i,j} = 1})\subset \mathbb{A}^{\ell}_{\mathbb{C}}$$
be the localization with respect to $T_{i,j}$, and let 
$\theta  =  (\overline{N}, N, F, S, \Sigma)$
be the weight package of $X^{(1,1)}$ as in \ref{sec-trinomial-aff}. We denote by
$\underline{\theta} =  (\theta^{(i,j)})_{ 1\leq i\leq 3,\,1\leq j\leq r_{i}}$
the sequence of weight packages coming from the enhanced structure of $\theta$.
Concretely, the $P$-matrix of $\widehat{\theta}$ corresponds to the linear map 
$\mathbb{Z}^{\ell+1}\rightarrow \widehat{R}(\mathbb{Z}^{\ell+1}),$ $w \mapsto \widehat{R}(w)$, where
$$ \widehat{R} =  \begin{pmatrix}
 -n_{1,1}& \ldots &  -n_{1, r_{1}} &  n_{2,1} &\ldots & n_{2,r_{2}} & 0& \ldots & 0\\
  -n_{1,1}& \ldots  &  -n_{1, r_{1}} &  0 &\ldots & 0 & n_{3,1} & \ldots & n_{3, r_{3}} 
\end{pmatrix}.$$
The $P$-matrix $P^{(i,j)}$ of $\theta^{(i,j)}$
is obtained from $\widehat{R}$ by omitting the column corresponding to the variable $T_{i,j}$. Write 
$$\theta^{(i,j)} = (\overline{N}, N, F^{(i,j)}, S^{(i,j)}, \Sigma^{(i,j)}).$$
Since the sum of the components of each line of $\widehat{R}$ is $0$, the subspaces ${\rm Coker}(F^{(i,j)})\subset \mathbb{Z}^{2}$ and the fan $\Sigma^{(i,j)}$ does not depend on $(i,j)$. We will consider the notations $d, d_{1}, d_{2}, d_{3}, u$ defined in Section \ref{sec-trinomial-aff} for $\theta$ so that $X_{\Sigma^{(i,j)}} =  \mathbb{P}(d_{1}, d_{2}, d_{3})$  for any $(i,j)$. 
\\

The next result is a direct consequence of Theorem \ref{Kruglov-Theorem}.
\begin{corollary}\label{cor-div-fan-proj-trinom}
Let 
$$X  =  \mathbb{V}(T_{1}^{\underline{n}_{1}} + T_{2}^{\underline{n}_{2}}  + T_{3}^{\underline{n}_{3}})\subset \mathbb{P}_{\mathbb{C}}^{\ell}$$
be a relevant projective trinomial hypersurface with weight package $\theta$, and denote by $\underline{\theta} =  (\theta^{(i,j)})$ the sequence from the enhanced structure of $\theta$. Then the normal $\mathbb{T}_{N}$-variety $X$ is described by the divisorial fan $\mathscr{E}_{\theta} =  \{\bar{\mathfrak{D}}_{\theta^{(a, b)}}\},$ where each element $\bar{\mathfrak{D}}_{\theta^{(a,b)}}$ is defined over the  curve $C_{d_{1}, d_{2}, d_{3}}$ and is given by the relation
$$\bar{\mathfrak{D}}_{\theta^{(a,b)}}  =  \bar{\mathfrak{D}}_{\theta^{(a, b)}, 1}\cdot E_{1} + \bar{\mathfrak{D}}_{\theta^{(a, b)}, 2}\cdot E_{2} + \bar{\mathfrak{D}}_{\theta^{(a, b)}, 3}\cdot E_{3}.$$
The polyhedral divisor $\bar{\mathfrak{D}}_{\theta^{(a,b)}}$ has the following properties. 
\begin{itemize}
\item[$(i)$] The polyhedron $ \bar{\mathfrak{D}}_{\theta^{(a, b)}, i}$ for $1\leq i\leq 3$ is the Minkowski sum of $\sigma_{\theta^{(a, b)}} =  S^{(a,b)}(\mathbb{Q}_{\geq 0}^{\ell}\cap F^{(a, b)}(N_{\mathbb{Q}}))$ and the convex hull of the set 
$$ \Gamma_{i}^{(a, b)}: =  \left\{ S^{(a,b)}\left( \frac{d}{d_{i} n_{i,j}} e_{i,j}\right)\,|\, (i,j)\neq (a, b), 1\leq j\leq r_{i}\right\}.$$
Here
$$(e_{1,1}, \ldots, e_{1, r_{1}}, e_{2, 1}, \ldots, \widehat{e_{(a, b)}}, \ldots, e_{3, r_{3}})$$
is the canonical basis of $\overline{N} =  \mathbb{Z}^{\ell}$ (the symbol $\widehat{e_{(a, b)}}$ means that we omit the index $(a, b)$).
\item[$(ii)$] The divisors $E_{1}, E_{2}, E_{3}$ satisfy Condition $(ii)$ of Theorem \ref{Kruglov-Theorem}.
\end{itemize}
\end{corollary}

Next gives a formula for the intersection cohomology of relevant projective trinomial hypersurfaces.  We refer 
to \cite{BC94, BB96} for another description using Hodge theory. 
\begin{corollary}
Let 
$$X  =  \mathbb{V}(T_{1}^{\underline{n}_{1}} + T_{2}^{\underline{n}_{2}}  + T_{3}^{\underline{n}_{3}})\subset \mathbb{P}_{\mathbb{C}}^{\ell}$$
be a relevant projective trinomial hypersurface with weight package $\theta$, and denote by $\underline{\theta} =  (\theta^{(i,j)})$ the sequence from the enhanced structure of $\theta$. Set $$\sigma_{\theta^{(a, b)}} =  S^{(a,b)}(\mathbb{Q}_{\geq 0}^{\ell}\cap F^{(a, b)}(N_{\mathbb{Q}})),$$
$$\Pi_{i}^{(a, b)} := {\rm Cone}\left((\sigma_{\theta^{(a, b)}}\times\{0\})\cup (\Gamma_{i}^{(a, b)}\times\{1\})\right)\text{ for }i =1,2,3,
\text{ and }$$
$$\Pi_{-}^{(a, b)}  =  {\rm Cone}\left((\sigma_{\theta^{(a, b)}}\times\{0\})\cup (\sigma_{\theta^{(a, b)}}\times\{-1\})\right),$$
 where the set $\Gamma_{i}^{(a, b)}$ is defined in Corollary \ref{cor-div-fan-proj-trinom}. Write $\gamma =  d(d_{1} + d_{2} + d_{3})$ and let $\Sigma_{i}(\theta)$ be the complete
fan generated by 
$$\left\{\Pi_{i}^{(a, b)}, \Pi_{-}^{(a, b)}\,|\, 1\leq a\leq 3,\, 1\leq b\leq r_{a}\right\}.$$
Then the intersection cohomology Betti numbers of the contraction space $\widetilde{X}$ of $X$ is given by
$$P_{\widetilde{X}}(t) =  \left((1-\gamma)t^{2} + (du - \gamma + 2)t  + 1- \gamma \right)\cdot h(\Sigma(\theta); t^{2}) +  \sum_{i = 1}^{3} dd_{i}\cdot h(\Sigma_{i}(\theta); t^{2}),$$
where $\Sigma(\theta)$ is the fan generated by the cones $\sigma_{\theta^{(a,b)}}$. 

Furthermore, write
$$H(\underline{\theta}, \underline{n}_{1}, \underline{n}_{2}, \underline{n}_{3})  =  \bigcup_{1\leq a \leq 3, 1\leq b \leq r_{a}}
H(\theta^{(a, b)}, \underline{n}_{1}, \underline{n}_{2}, \underline{n}_{3}).$$
Then the intersection cohomology Betti numbers of $X$ are given by the formula 
$$P_{X}(t)  = P_{\widetilde{X}}(t)  -  \sum_{\tau \in H(\underline{\theta}, \underline{n}_{1},\underline{n}_{2},\underline{n}_{3})} g_{n(\tau)}(\tau) t^{c(\tau)} \cdot P_{\mathbb{P}^{n-c(\tau)}_{\mathbb{C}}}(t), $$
where $n =  \ell -2$, $n(\tau) =  {\rm dim}\, \tau - 1$ and $c(\tau) =  {\rm dim}\, \tau +1$.  
\end{corollary}
\begin{proof}
Consequence of Corollary \ref{cor-proj2222}, Corollary \ref{cor-div-fan-proj-trinom} and the fact that orbit closures
of the image of the exceptional locus of the contraction map of $X$ are projective spaces. 
\end{proof}
\begin{example}
Consider the projective trinomial fourfold hypersurface
$$\mathbb{V}(T_{1,1}^{3} T_{1,2}^{9} + T_{2,1}^{6} T_{2,2}^{6} + T_{3,1}^{3} T_{3,2}^{9})\subset \mathbb{P}_{\mathbb{C}}^{5}$$
with weight package $\theta =  (\mathbb{Z}^{5}, \mathbb{Z}^{3}, F, S, \Sigma)$ such that
$$ F  = {}^{t} \begin{pmatrix}
 0 &   1 & -1 & 0& 0  \\
 2 &  0 &  3& 6& 0  \\
 0  &  0 &  0 & -3& 1
\end{pmatrix} \text{ and } S  = \begin{pmatrix}
 0 &   1 & 0 & 0& 0  \\
 -1 &  1 &  1& 0& 0  \\
 0  &  0 &  0 & 0& 1
\end{pmatrix}.$$
Set
$$v_{1} =  (3,1, 0), v_{2} =  (0,1,0), v_{3} =  (3,1,2), v_{4} =  (0,1,2),$$
$$w_{1} =  (-3,-3,-4), w_{2} =  (-1,-1, -2), w_{3} =  (-6, -3, -4), w_{4} =  (-2,-1, -2).$$
Then maximal cones of $\Sigma(\theta)$ are
\begin{center}
 \begin{tabular}{ |c|c|}
 \hline 
  &  Generators \\ 
 \hline\hline
 $\sigma_{\theta^{(1,1)}}$ & $v_{1}, v_{2}, v_{3}, v_{4}$ \\ 
 \hline
  $\sigma_{\theta^{(1,2)}} $ & $w_{1}, w_{2}, w_{3}, w_{4}$  \\   
 \hline
  $\sigma_{\theta^{(2,1)}}$ & $ v_{2}, v_{4},  w_{3},  w_{4}$  \\ 
 \hline
  $\sigma_{\theta^{(2,2)}}$ & $ v_{1}, v_{3}, w_{1} , w_{2}$  \\
 \hline
  $\sigma_{\theta^{(3,1)}} $ & $ v_{3}, v_{4}, w_{1} , w_{3}$  \\
 \hline
  $\sigma_{\theta^{(3,2)}}$ & $v_{1},  v_{2}, w_{2}, w_{4}$  \\
 \hline
\end{tabular}
\end{center}
and thus $h(\Sigma(\theta);t^{2})  =   1 + 5t^{2} + 5t^{4} + t^{6}$.
Moreover, below  we get the $f$-vector of $\Sigma_{i}(\theta)$
\begin{center}
\begin{tabular}{|c|c|l|r|} 
\hline  
$\sigma\in \Sigma_{i}(\theta)$ & Number of cones\\  
 \hline\hline
${\rm dim}(\sigma) = 1$ & 11  \\  
\hline
${\rm dim}(\sigma) =2$ & 29 \\ 
\hline 
${\rm dim}(\sigma) = 3$, $|\sigma(1)| = 3$ & 20 \\  
\hline
 ${\rm dim}(\sigma) = 3$, $|\sigma(1)| = 4$ & 10\\ 
\hline 
 ${\rm dim}(\sigma) = 4$, $|\sigma(1)| = 5$  & 8\\ 
\hline 
 ${\rm dim}(\sigma) = 4$, $|\sigma(1)| = 6$  & 4\\ 
\hline 
\end{tabular}
\end{center}
and by Remark \ref{rem-hpolexplicit234}, $h(\Sigma_{i}(\theta);t^{2}) =  t^{8} + 7t^{6} + 12t^{4} +7t^{2} +1$ for $i=1,2,3$. Note that
in the above table, only cases ${\rm dim}(\sigma) =  1,2,4$ need to be computed, since the rest follows from the vanishing of the Euler characteristic of the polytope of $\Sigma_{i}(\theta)$ and  Poincar\'e duality.
We conclude that
$$P_{\widetilde{X}}(t)  = t^{8} +t^{7} +15t^{6} + 5t^{5} +28t^{4} + 5t^{3} +15t^{2} +t + 1 \text{ and }$$
$$P_{X}(t)  = t^{8} +t^{7} +7t^{6} + 5t^{5} +14t^{4} + 5t^{3} +7t^{2} +t + 1.$$
\end{example}

\end{document}